%% file: KlugMiller_Sep15.tex
\documentclass[11pt]{amsart}

\usepackage{amsmath,amsthm,verbatim,amssymb,amsfonts,amscd, graphicx}
\usepackage{graphics}
\usepackage{tikz-cd}
\usepackage{pinlabel}
\usepackage{xcolor}
\usepackage{stackrel}
\usepackage[margin=1.5in]{geometry}%\theoremstyle{plain}
\newtheorem{theorem}{Theorem}[section]
\newtheorem{corollary}[theorem]{Corollary}
\newtheorem{lemma}[theorem]{Lemma}
\newtheorem{proposition}[theorem]{Proposition}

\newtheorem{question}[theorem]{Question} 
 
\theoremstyle{definition}
\newtheorem{definition}[theorem]{Definition}
\newtheorem{remark}[theorem]{Remark}

\theoremstyle{definition}
\newtheorem{example}[theorem]{Example}
\newtheorem{move}[theorem]{Move}

\newcommand{\leftrarrows}{\mathrel{\raise.75ex\hbox{\oalign{%
  $\scriptstyle\leftarrow$\cr
  \vrule width0pt height.5ex$\hfil\scriptstyle\relbar$\cr}}}}
\newcommand{\lrightarrows}{\mathrel{\raise.75ex\hbox{\oalign{%
  $\scriptstyle\relbar$\hfil\cr
  $\scriptstyle\vrule width0pt height.5ex\smash\rightarrow$\cr}}}}
\newcommand{\Rrelbar}{\mathrel{\raise.75ex\hbox{\oalign{%
  $\scriptstyle\relbar$\cr
  \vrule width0pt height.5ex$\scriptstyle\relbar$}}}}

\makeatletter
\def\leftrightarrowsfill@{\arrowfill@\leftrarrows\Rrelbar\lrightarrows}
\newcommand{\xleftrightarrows}[2][]{\ext@arrow 3399\leftrightarrowsfill@{#1}{#2}}
\makeatother

\definecolor{violet}{rgb}{.6,.6,0}
\definecolor{green}{rgb}{.0,.8,0}

\newcommand{\Z}{\mathbb{Z}}
\newcommand{\R}{\mathbb{R}}

\newcommand{\CP}{\mathbb{C}P}

\newcommand{\F}{\mathbb{F}}

\let\int\relax
\newcommand{\int}{\mathring}

\newcommand{\boundary}{\partial}
\newcommand{\into}{\hookrightarrow}

\DeclareMathOperator{\pt}{{pt}}
\DeclareMathOperator{\fq}{{fq}}
\DeclareMathOperator{\km}{{km}}

\makeatletter
\DeclareRobustCommand\onedot{\futurelet\@let@token\@onedot}
\def\@onedot{\ifx\@let@token.\else.\null\fi\space}
\def\eg{{e.g}\onedot} 
\def\ie{{i.e}\onedot}

\def\vs{{vs}\onedot}

\makeatother

\title[Concordance of Surfaces and the Freedman-Quinn Invariant]{Concordance of Surfaces in 4-Manifolds and the Freedman-Quinn Invariant}

\author{Michael R. Klug}
\address{Department of Mathematics\\University of California, Berkeley\\  Berkeley, CA 94720, USA}
\email{michael.r.klug@gmail.com}

\author{Maggie Miller}
\address{Department of Mathematics\\Massachusetts Institute of Technology\\  Cambridge, MA 02139, USA}
\email{maggiehm@mit.edu}

\thanks{MK was supported by the Max Planck Institute for Mathematics (MPIM) in Bonn, Germany during the time of this project. MM was supported by MPIM during part of this project, as well as NSF Grants DGE-1656466 and DMS-2001675.}

\begin{document}

\maketitle

\begin{abstract}
We prove a concordance version of the 4-dimensional light bulb theorem for $\pi_1$-negligible compact orientable surfaces, where there is a framed but not necessarily embedded dual sphere.  That is, we show that if $F_0$ and $F_1$ are such surfaces in a 4-manifold $X$ that are homotopic and there exists an immersed framed 2-sphere $G$ in $X$ intersecting $F_0$ geometrically once, then $F_0$ and $F_1$ are concordant if and only if their Freedman-Quinn invariant $\fq$ vanishes. The proof of the main result involves computing $\fq$ in terms of intersections in the universal covering space and then applying work of Sunukjian in the simply-connected case. 
%In its simplest form $\fq$ is an invariant of a pair of based-homotopic 2-spheres originally introduced by Freedman and Quinn and recently studied by Schneiderman and Teichner. 
\end{abstract}

\section{Introduction}
  The main goal of this paper is to prove a concordance analogue of Gabai's 4-dimensional light bulb theorem. The 4-dimensional light bulb theorem strengthens homotopy of embedded 2-spheres $R$ and $R'$ in a 4-manifold $X^4$ to isotopy, given a dual sphere $G$ which intersects both $R$ and $R'$ exactly once and a condition on how the homotopy interacts with 2-torsion in $\pi_1(X)$. 
  
  In this paper we work in the smooth category unless otherwise specified. All manifolds are smooth and oriented; all maps between manifolds are smooth. At the end of \S\ref{sec:genus}, we remark on how results extend to the topological category.

In \S\ref{sec:3d} and \S\ref{sec:lightbulb} we discuss some context and motivation for our work.  In \S\ref{sec:fq} we discuss the Freedman-Quinn invariant in our context.  In \S\ref{sec:2spheres} we extend the work of Sunukjian and prove our main result for spheres (which follows from work of Stong as discussed below) and in \S\ref{sec:genus} we extend this to higher genus surfaces.  In \S\ref{sec:examples} we give examples to illustrate the necessity of the various conditions in the statement of our theorem.  In \S\ref{sec:questions} we discuss some remaining questions.  

%In \S\ref{sec:lightbulb} we give a detailed statement of the light bulb theorem. The second author later proved a concordance analogue of the light bulb theorem: If there exists a 2-sphere $G$ which intersects $R$ exactly once (but with no condition on $G\cap R'$) and the condition on 2-torsion in $\pi_1(X)$ is satisfied, then $R$ and $R'$ are concordant.

\begin{definition}
Let $A$ and $B$ be $k$-dimensional submanifolds of an $n$-manifold $M^n$, $k \leq n$. Then $A$ and $B$ are {\emph{concordant}} if there exists a $k$-manifold $\Sigma$ and a smooth embedding $f:\Sigma\times I\hookrightarrow M\times I$ so that $f(\Sigma\times 0)=A\times 0$ and $f(\Sigma\times 1)=-B\times 1$. Note that if $A$ and $B$ are ambiently isotopic, then they are concordant.
\end{definition}

\begin{definition}
Given manifolds $V$ and $Z$, a map $f : V \to Z$ is called {\emph{$\pi_1$-negligible}} if the resulting map on $\pi_1$ is trivial. A connected submanifold $A$ of $Z$ is called {\emph{$\pi_1$-negligible}} if  $\pi_1(A)$ maps trivially to $\pi_1(Z)$ with respect to the map induced by inclusion.
\end{definition}

\begin{definition}
Given a surface $F$ embedded in a 4-manifold $X$, a \emph{framed dual} to $F$ is an immersed sphere $G$ in $X$ that intersects $F$ in a single point, such that the normal bundle of $G$ has even Euler number.  If the normal bundle of $G$ has odd Euler number, then we call $G$ an \emph{unframed dual}.

Given a three-manifold $Y$ embedded in a 5-manifold $W$, a \emph{framed dual} to $Y$ is an embedded sphere $G$ in $W$ that intersects $Y$ in a single point, such that the normal bundle of $G$ is trivial. If the normal bundle of $G$ is nontrivial, then we call $G$ an \emph{unframed dual}.

These two definitions are related -- if $G$ is a dual sphere for a surface $F$ in $X^4$ and $F$ is a boundary component of a 3-manifold $Y$ in $X\times I$, then after a small homotopy $G$ is a dual sphere for $Y$.
\end{definition}

The following is the main result of this paper:

\begin{theorem}\label{maintheorem}
Let $F_0$ and $F_1$ be two homotopic embedded orientable genus-g $\pi_1$-negligible surfaces in a 4-manifold $X^4$. Assume at there exists an immersed sphere $G$ in $X^4$ that is a framed dual to $F_0$.  Then $F_0$ and $F_1$ are concordant if and only if $\fq(F_0,F_1)=0$. 
\end{theorem}

We will first prove Theorem \ref{maintheorem} in the case that $F_0$ and $F_1$ are 2-spheres (in \S\ref{sec:2spheres}) and then extend to positive-genus surfaces in \S\ref{sec:genus}. In \S\ref{sec:examples}, we give three explicit examples showing: \begin{itemize}\item The necessity of $\fq(F_0,F_1)=0$ (Example \ref{hannahexample}, due to Schwartz \cite{Hannah}),
\item The sphere $G$ being framed (Example \ref{stongexample}, constructed by the authors using Stong's \cite{stong} $\km$ invariant),
\item The conclusion of concordance rather than isotopy (Example \ref{satoexample}, due to Sato \cite{Sato}).
\end{itemize}

We make use of recent work of Schneiderman and Teichner, who discuss the invariant $\fq$ associated to a pair of based-homotopic 2-spheres in a 4-manifold (this invariant originally appeared in work of Freedman-Quinn \cite{fq} - hence the name). We define $\fq$ for based homotopic 2-spheres in \S\ref{sec:fq} (and discuss free homotopy \vs based homotopy in \S\ref{sec:htpy}). We extend to positive-genus surfaces whose fundamental groups include trivially into the ambient 4-manifold in \S\ref{sec:genus}.

The original version of this paper included the statement of Theorem \ref{maintheorem} for spheres as the main theorem and the extension to higher genus surfaces as an application.  Later, we were made aware of some relevant work of Freedman in Quinn in chapter 10 of \cite{fq} - namely Theorem 10.9, and the subsequent correction and extension of this work by Stong in \cite{stong}.  In particular, Theorem \ref{maintheorem} for spheres follows from the work of Stong in \cite{stong} (which builds on and corrects \cite{fq}).  The authors are currently preparing an exposition/interpretation of Stong's work directly in the context of constructing/obstructing concordances of surfaces in 4-manifolds.  A modification of Stong's proof also gives another proof of Theorem \ref{maintheorem} -- although using rather different techniques compared to what we use in this paper, as we will explain in our forthcoming work.  In the context of spheres in a 4-manifold, Stong's work identifies two obstructions to two homotopic spheres being concordant: $\fq$ and a secondary obstruction which Stong calls the 5-dimensional Kervaire-Milnor invariant.  Moreover, Stong shows that when these obstructions vanish, a concordance exists.  It is useful to note that when either sphere has a framed immersed geometrically-dual sphere (as in Theorem \ref{maintheorem}), the secondary Kervaire-Milnor invariant automatically vanishes. More generally, this Kervaire-Milnor invariant vanishes if either sphere is not $s$-characteristic -- i.e. intersects some immersed 2-sphere $R$ in $X^4$ with parity not equal to $R\cdot R\pmod{2}$.

Theorem \ref{maintheorem} has the following corollary (suggested by Sunukjian in \cite{NS} and in personal communication).
\begin{corollary}\label{cor:top}
Let $F_0$ and $F_1$ be homotopic genus-g orientable surfaces embedded in a 4-manifold $X^4$, with $\pi_1(X^4)$ good in the sense of Freedman and Quinn \cite{fq}. Assume all of the following:
\begin{itemize}
    \item $\pi_1(F_i)$ maps trivially to $\pi_1(X)$ under inclusion for $i=0,1$,
    \item there exists a framed immersed sphere in $X^4$ that intersects $F_0$ geometrically once,
    \item there exists a (potentially unframed immersed sphere in $X^4$ that intersects $F_1$ geometrically once,
    \item $\fq(F_0,F_1)=0$.
\end{itemize}
Then there exists a homeomorphism of pairs $(X,F_0)\cong (X,F_1)$.
\end{corollary}

In this paper, Corollary \ref{cor:top} is the only time we leave the smooth category. A version of this corollary originally appeared in \cite{NS} (for $\pi_1(X)$ with no 2-torsion) and later in \cite{Maggie} (when $F_0$ is a 2-sphere and its dual is embedded). We restate the proof of Sunukjian \cite{NS} here.

\begin{proof}
By Theorem \ref{maintheorem}, there is a concordance $C$ between $F_0$ and $F_1$ in $X\times I$. Let $\widetilde{X}$ denote the universal cover of $X$ and let $\widetilde{C}$ be the union of all lifts of $C$ in $\widetilde{X}\times I$. Let $\widetilde{F}_i$ denote the union of all lifts of $F_i$. Since the meridian of $F_i$ is nullhomotopic in $X-\mathring{N}(F_i)$, the cover $\widetilde{X}\times I-\mathring{N}(\widetilde{F_i})$ is simply-connected and hence is the universal cover of $X-\mathring{N}(F_i)$. Similarly, $\widetilde{X}\times I-\mathring{N}(\widetilde{C})$ is the universal cover of $X\times I-\mathring{N}(C)$. An application of the Mayer-Vietoris theorem yields that $\widetilde{X}\times I-\mathring{N}(\widetilde{C})$ is an $h$-cobordism from $\widetilde{X}-\mathring{N}(\widetilde{F}_0)$ to $\widetilde{X}-\mathring{N}(\widetilde{F}_1)$. Therefore, $X\times I-\mathring{N}(C)$ is an $h$-cobordism from $X-\mathring{N}(F_0)$ to $X-\mathring{N}(F_1)$. Since $\mathring{N}(C)$ and $(X\times I-\mathring{N}(C))\cup \mathring{N}(C)=X\times I$ are both products and Whitehead torsion is additive, we conclude that $X\times I-\mathring{N}(C)$ has vanishing Whitehead torsion and hence is an $s$-cobordism. Finally, since $\pi_1(X\times I-\mathring{N}(C))\cong\pi_1(X)$ is a good group, $X\times I-\mathring{N}(C)$ is topologically a product \cite{fq}. This product structure yields a homeomorphism $(X,F_0)\cong (X,F_1)$.
\end{proof}

\section{Three-dimensional motivation}\label{sec:3d}
Both the light bulb theorem and the concordance analogue are motivated by 3-dimensional theorems in $S^2\times S^1$. For notational purposes, let $J:=\pt\times S^1$ and $G:=S^2\times\pt$ in $S^2\times S^1$. %, which we state in this section as easy motivation for the light bulb theorem and related statements.
\begin{theorem}[3-dimensional light bulb trick (folklore)]
Let $K$ be a knot smoothly embedded in $S^2\times S^1$ intersecting $G$ transversally exactly once. Then $K$ is isotopic to $J$.
\end{theorem}
\begin{proof}

Isotope the knot $K$ to agree with $J$ near $G$. In the complement of $G$, $K$ and $J$ are homotopic arcs related by a sequence of crossing changes. Using the dual $G$, each of these crossing changes can be achieved as isotopy rather than homotopy. See Figure \ref{fig:trick}.
\end{proof}

\begin{figure}
	  \labellist
\small\hair 2pt
 \pinlabel $S^2\times S^1$ at 0 150
  \pinlabel $G=S^2\times\pt$ at -40 55
 \pinlabel $K$ at 135 130
\endlabellist
\includegraphics[width=110mm]{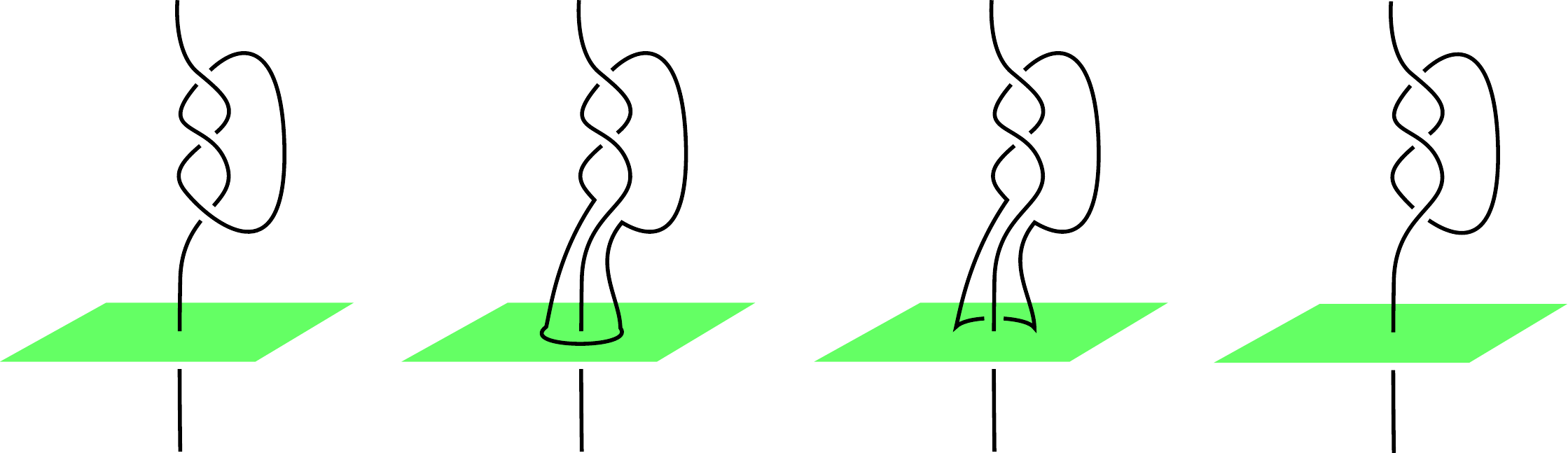}
\caption{The 3-dimensional light bulb trick allows us to realize crossing changes of $K$ by isotopy.}\label{fig:trick}
\end{figure}

\begin{theorem}[\cite{yildiz}, \cite{dnpr}]
Let $K$ be a knot smoothly embedded in $S^2\times S^1$ in the homology class $[\pt\times S^1]$. Then $K$ is concordant to $J$.
\end{theorem}

We summarize the proof of Yildiz \cite{yildiz}.
\begin{proof}

Since $\pi_1(S^2\times S^1)\cong\Z$ is abelian, $K$ is homotopic to $J$ via some homotopy $H$. We show how to modify the track of the homotopy which we also call $H : S^1 \times I \to (S^2 \times S^1) \times I$ so that it is embedded. The levels where the track of the homotopy is not an embedding are the levels where crossing changes occur.   

In Figure \ref{fig:yildiz}, we show how to change the track of the homotopy along 2-dimensional 1-handles (bands) together with 2-dimensional 2-handles (disk) to obtain an embedding of $S^1 \times I$.  
Each time a crossing change occurs in the track of $H$, we change the map so that a 1-handle (band) is added with the effect that the crossing change is achieved but a small meridian disk is also introduced.  This small meridian disk is then carried throughout the rest of the new map.  By considering the Euler characteristic, we see that this results in a new map that is a planar surface in  $(S^2 \times S^1) \times I$ with boundary on one end the original knot $K$ and one the other end $J$ together with several small meridians of $J$.  These meridians form an unlink and so by capping off the planar surface with disjoint disks bounding all of these meridians, we obtain the desired cobordism.  

%link of $\pt\times S^1$ and several meridians $U_1,\ldots, U_n$. In $S^2\times S^1$, these meridian circles are unknotted, so can be deleted by surgering the link along 2-dimensional 2-handles (disks). This sequence of operations describes a handle decomposition of an annulus embedded in $(S^2\times S^1)\times I$ with boundary $(K\times 0)\sqcup((-\pt\times S^1) \times 1).$
\end{proof}

\begin{figure}
\vspace{.2in}
 	  \labellist
\small\hair 2pt
 \pinlabel $S^2\times S^1$ at -75 200
  \pinlabel $G$ at -20 50
 \pinlabel $K$ at 75 210
\endlabellist
    \includegraphics[width=110mm]{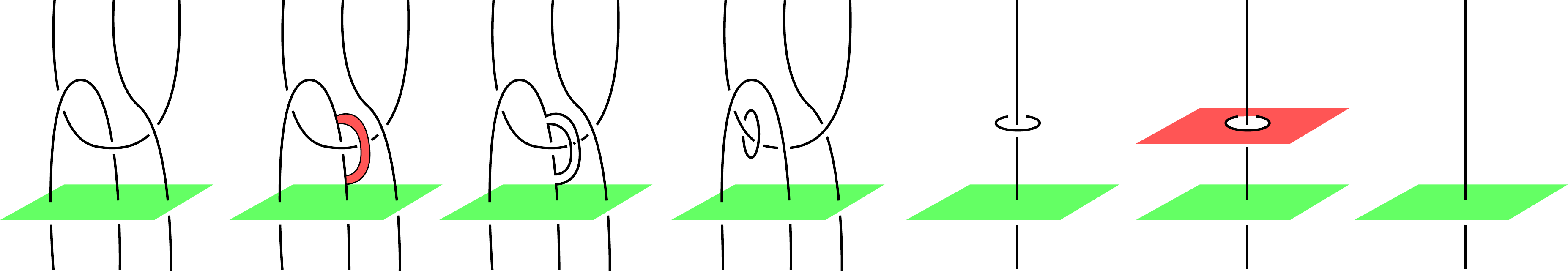}
    \caption{When $K\subset S^2\times S^1$ intersects $G=S^2\times\pt$ algebraically once, we can achieve crossing changes of $K$ by concordance.}
    \label{fig:yildiz}
\end{figure}

\section{Four-dimensional motivation}\label{sec:lightbulb}

In this section, we describe past results %mention some theorems
related to our main theorem to %try and
give context to our work. This section is not necessary to understand the proof of Theorem \ref{maintheorem}. %; none of what is mentioned is logically necessary for what follows.  Note that isotopic spheres in a 4-manfiold are necessarily concordant.  We mention recent results concerning isotopy of spheres when the spheres have a common embedded geometric dual.  This is a very strong restriction, but under nice circumstances (like absence of 2-torsion in the fundamental group of the 4-manifold), it leads us to conclude that the spheres are isotopic.  We will drop the condition that our spheres under consideration have a common geometric dual, and instead of discussing isotopy we will discuss concordance of the spheres.  

In 2017, Gabai \cite{Dave} proved the following theorem about isotopy of surfaces in 4-manifolds.
\begin{theorem}[4D light bulb theorem, \cite{Dave}]\label{thm:lightbulboriginal}
Let $S_0,S_1,$ and $G$ be 2-spheres smoothly embedded in a smooth 4-manifold $X^4$. Assume that $S_0$ and $S_1$ are homotopic, $G$ has trivial normal bundle, and $S_0$ and $S_1$ both transversally intersect $G$ in exactly one point. Finally, assume $\pi_1(X)$ has no 2-torsion. Then $S_0$ and $S_1$ are ambiently smoothly isotopic.
\end{theorem}

In Theorem \ref{thm:lightbulboriginal}, the assumption that $\pi_1(X)$ has no 2-torsion is essential. Schwartz \cite{Hannah} later gave explicit examples of triples $(S_0,S_1,G)$ of 2-spheres in a 4-manifold $X$ with $\pi_1(X)\cong\Z/2$ that satisfy the other hypotheses of Theorem \ref{thm:lightbulboriginal}, but yet $S_0$ and $S_1$ are not isotopic. We discuss these examples in \S \ref{sec:examples}.

%Gabai \cite{Dave} proves the 4D light bulb theorem in 4-manifolds with 2-torsion in fundamental group by assuming an extra hypothesis on the homotopy between $S_0$ and $S_1$. We describe the more general statement in Section \ref{sec:generallightbulb}, after developing some necessary language in Section \ref{sec:fq}.
%\mike{Do we need this paragraph?}\maggie{No, I just like being redundant. But we can take it out}

The second author then proved a concordance version of the theorem, explicitly using the 4-dimensional light bulb theorem. %of the 4D light bulb theorem.
\begin{theorem}[\cite{Maggie}]\label{thm:maggie}
Let $S_0,S_1,$ and $G$ be 2-spheres smoothly embedded in a smooth 4-manifold $X^4$. Assume that $S_0$ and $S_1$ are homotopic, $G$ has trivial normal bundle, and $S_0$ transversally intersects $G$ in exactly one point. Finally, assume $\pi_1(X)$ has no 2-torsion. Then $S_0$ and $S_1$ are smoothly concordant.
\end{theorem}

Note that Theorem \ref{thm:maggie} requires that one of the 2-spheres ($S_0$) admit a dual sphere $G$. This assumption is essential; in general we do not expect arbitrary homotopic 2-spheres in a 4-manifold (even without 2-torsion in $\pi_1$) to be concordant. In Theorem \ref{maintheorem}, we weaken the condition that $S_0$ have a dual sphere to the meridian of $S_0$ being nullhomotopic in $X-S_0$. This weaker condition requires the presence of an immersed dual for $S_0$ rather than an embedded dual.

%\begin{remark}
%If $\pi_1(X)=1$, then Theorem \ref{thm:lightbulboriginal} follows from \cite[Theorem 6.1]{NS}, which shows that arbitrary oriented homotopic surfaces in a simply-connected manifold are concordant. We discuss this theorem later in this paper; it is stated explicitly as Theorem \ref{simply-connected}.
%\end{remark}

The conclusion of Theorem \ref{thm:maggie} (and Theorem \ref{maintheorem}) cannot be isotopy, rather than concordance. Sato \cite{Sato} constructs a 2-sphere $K$ in $S^2\times S^2$ which is homotopic to $S^2\times\pt$ but not isotopic to $S^2\times\pt$. Theorem \ref{thm:maggie} (and indeed \cite[Theorem 6.1]{NS}) implies that $K$ is concordant to $S^2\times\pt$. See Example \ref{satoexample}.

Again, Schwartz's \cite{Hannah} examples show that the 2-torsion assumption in Theorem \ref{thm:maggie} is necessary, as her counterexamples to the 4D light bulb theorem in the presence of 2-torsion are also counterexamples to an analog of Theorem \ref{thm:maggie} in the presence of 2-torsion. See Example \ref{hannahexample}. %However, in \cite{Maggie}, the second author gives a more general statement of Theorem \ref{thm:maggie} when $\pi_1(X)$ has 2-torsion. The theorem then follows with an additional hypothesis on the homotopy between $S_0$ and $S_1$; this assumption exactly mirror the general version of the 4D light bulb theorem. We discuss this generalization in \S\ref{sec:generallightbulb}.

The 4D lightbulb theorem was reproved and generalized by Schneiderman and Teichner in \cite{ST}.  %They fix a homotopy between the two 2-spheres consisting of a sequence of finger moves followed by a sequence of Whitney moves. They then consider the immersed sphere at the center of this homotopy and show that the choice of Whitney moves % Whitney disks for this immersed sphere
%affects the isotopy class of the resulting embedded sphere. % that results from performing the Whitney moves.  Using the dual sphere critically in many ways,
%Schneiderman and Teichner show that the choices of pairings, arcs, and Whitney disks
%does not affect the isotopy class of the resulting embedded 2-sphere if there is no 2-torsion in $\pi_1(X)$. % for a fixed choice of sheets.  The choice of sheets is only relevant when there is 2-torsion and the
When $\pi_1(X)$ has 2-torsion, Schneiderman and Teichner show that it is necessary and sufficient for the  %authors relate this to the
{\emph{Freedman-Quinn invariant}} $\fq(S_0,S_1)$ of the two spheres $S_0,S_1$ to vanish in order for there to exist an isotopy. We define this invariant in \S\ref{sec:fq}, but for now we discuss the full statement of the light bulb theorem.

%\maggie{Should the basepoint specifically be a common intersection with $G$?}
%\mike{and probably based-maps that are based homotopic... right?}

\begin{theorem}[general 4D light bulb theorem, \cite{Dave} \cite{ST}]\label{thm:lightbulb}
Let $S_0,S_1$, and $G$ be 2-spheres smoothly embedded in a smooth 4-manifold $X^4$. Assume that $S_0$ and $S_1$ are homotopic, $G$ has trivial normal bundle, and $S_0$ and $S_1$ both transversally intersect $G$ in exactly one point. Then $S_0$ and $S_1$ are ambiently smoothly isotopic if and only if $\fq(S_0,S_1)=0.$
\end{theorem}

Gabai did not state Theorem \ref{thm:lightbulb} in terms of $\fq$, but his full theorem statement is equivalent to the constructive direction of Theorem \ref{thm:lightbulb}. Schneiderman and Teichner then noted that the invariant $\fq$ obstructs isotopy in the remaining cases. %by using the invariant $\fq$.%, which we define in \S\ref{sec:fq}. %. (Recall from Proposition \ref{lem:nobasegen} that we may take $S_0$ and $S_1$ to be {\emph{based}}-homotopic without loss of generality.) %Note that in Theorem \ref{thm:lightbulb}, we do not specify that $S_0$ amd $S_1$ are {\emph{based}}-homotopic. This is possible due to the following lemma in \cite{ST}, which we are implicitly using in Theorem \ref{thm:lightbulb} to define $\fq(S_0,S_1)$.

The impetus of this current paper was a pair of subtle errors in an argument in \cite{NS} involving 2-torsion and framings (see \S\ref{sec:2spheres}). Theorem \ref{maintheorem} corrects this error by adding the assumption that the Freedman-Quinn invariant $\fq(S_0,S_1)$ vanishes and that the dual sphere $G$ is framed.

\section{The Freedman-Quinn Invariant}\label{sec:fq}
We first review the definition of the Freedman-Quinn invariant $\fq$ as in \cite{ST}. We remind the reader that this invariant first appeared in work of Freedman-Quinn \cite{fq}, and indeed all of the results in this section are contained in \cite{fq} (albeit phrased slightly differently). However, we will use the conventions of \cite{ST}, and use them to give a new definition of $\fq$ via intersections of lifts in a universal cover. After reading the definition of the codomain of $\fq$, the reader uninterested in the details can simply take Proposition \ref{fqdef4} as the definition of $\fq$ (and Proposition \ref{mudef4} as the corresponding definition of the associated invariant $\mu_3$), and skip to Section 5.  To arrive at this interpretation, we start with the 4-dimensional and 6-dimensional definitions of $\fq$ as in \cite{ST}.  We include the 4-dimensional definition because we find it to be a natural starting point.  The 6-dimensional definition is then interpreted in terms of intersections in the universal covering space. 

%This consists of the original 6-dimensional definition and a 5-dimensional method of computation. % that are given in \cite{ST}.
%Both of these definitions can be translated into corresponding methods of computation using intersections of lifts in the universal covering space, which we present in \S\ref{sec:covers}. %, and this is what we present next. 

To define the codomain of $\fq$, one needs to define a certain homomorphism $\mu_3 : \pi_3(X) \to \mathbb{F}_2T_X$. Each method of computing $\fq$ is accompanied by a corresponding method of computing $\mu_3$.  We give the proofs of the different interpretations for both $\mu_3$ and $\fq$; however, the proofs are rather redundant and the reader might want to just look at the proofs for one of the two.   %which is used in the definition of the group where $\fq$ takes values.  In addition to $\fq$ admitting the aforementioned four methods of computation, $\mu_3$ does also, and the arguments are essentially identical in both cases.  We will use the 5-dimensional covering space method of computing $\fq$ and $\mu_3$ in the proof of the results that follow of the next section.  

\subsection{Schneiderman-Teichner's interpretations}

Let $X$ be a smooth, oriented, based 4-manifold. % with chosen basepoint which will often be implicitly used.
We write \[T_X := \{ g \in \pi_1(X) : g^2 = 1, g \neq 1 \}.\]
We call this the {\emph{2-torsion subset}} of $\pi_1(X)$. In this section, we will describe Schneiderman and Teichner's \cite{ST} definition/method for computing the Freedman-Quinn invariant: the original definition is 6-dimensional, while the second (equivalent) definition is 5-dimensional. The 5-dimensional definition is most similar to the techniques that will be used to prove Theorem \ref{maintheorem}, but the 6-dimensional definition is easier to state and prove to be well-defined.

\subsubsection{A 6-dimensional definition of $\fq$}\label{6dsec}

We write $\mathbb{F}_2 T_X$ to denote the $\mathbb{F}_2$-vector space with basis $T_X$.  We have an embedding $\mathbb{F}_2 T_X \hookrightarrow \mathbb{Z}\pi_1(X) / \langle g + g^{-1}, 1\rangle$ given by $g \mapsto g$.  We now describe a homomorphism $\mu_3 : \pi_3(X) \to \mathbb{F}_2 T_X$. Choose a basepoint for $S^3$.  Given a based map $f : S^3 \to X$, cross the codomain with $\mathbb{R}^2$ to obtain a new map (we abuse notation) $f : S^3 \to X \times \mathbb{R}^2$, and then perturb $f$ so that it is generic and hence an embedding away from double points of intersection.

\begin{definition}[{$\mu_3$, \cite[Section 4]{ST}}]\label{mudef1}
For each double point $p$ in the image of $f$, choose an arc between the preimages $f^{-1}(p)$. This arc maps to a closed loop $\lambda_p$ in $f(S^3)$ based at $p$. Conjugate this loop by an arc contained in $f(S^3)$ from the basepoint to $p$ to obtain a based loop $\gamma_p$. The convention that the arc be contained in $f(S^3)$ ensures that the choice of the arc does not affect the resulting element of $\pi_1$.  We write $g_p:=[\gamma_p]\in\pi_1(X)$. See Figure \ref{fig:gp}. Thus, to every self-intersection of $f$, we associate an element of $\pi_1(X)$. Schneiderman--Teichner \cite{ST} show that this element is contained in $\{1\}\cup T_X$ (\ie they show $g_p^2=1$).

%Note that $f^{-1}(\gamma_p)$ is not a smooth arc, since $\gamma_p$ changes local sheets once at $p$. 
%If we find a based loop $\gamma'_p$ by using an arc from the basepoint through the {\emph{other}} local sheet containing $p$, then the segments of $f^{-1}(\gamma_p)$ and $f^{-1}(\gamma'_p)$ can be glued to obtain a based loop $\widetilde{\gamma}_p$ in $S^3$ with $f(\widetilde{\gamma}_p)$ the concatenation of $\gamma_p$ and $\gamma'_p$. Since $\widetilde{\gamma}_p$ is nullhomotopic and $[f(\widetilde{\gamma}_p)=g^2_p$, we conclude $g_p \in T_x\cup\{1\}$.

%Note that the preimage $f^{-1}(\lambda_p)$ in $S^3$ double covers $\lambda_p$. We conclude that $g_p\in T_X\cup\{1\}$.

 \begin{figure}
 	  \labellist
\small\hair 2pt
 \pinlabel $\ast$ at 55 130
  \pinlabel $p$ at 120 125
 \pinlabel $\textcolor{blue}{\gamma_p}$ at 275 220
  \pinlabel $\textcolor{red}{f(S^3)}$ at 50 210
\endlabellist
 \includegraphics[width=50mm]{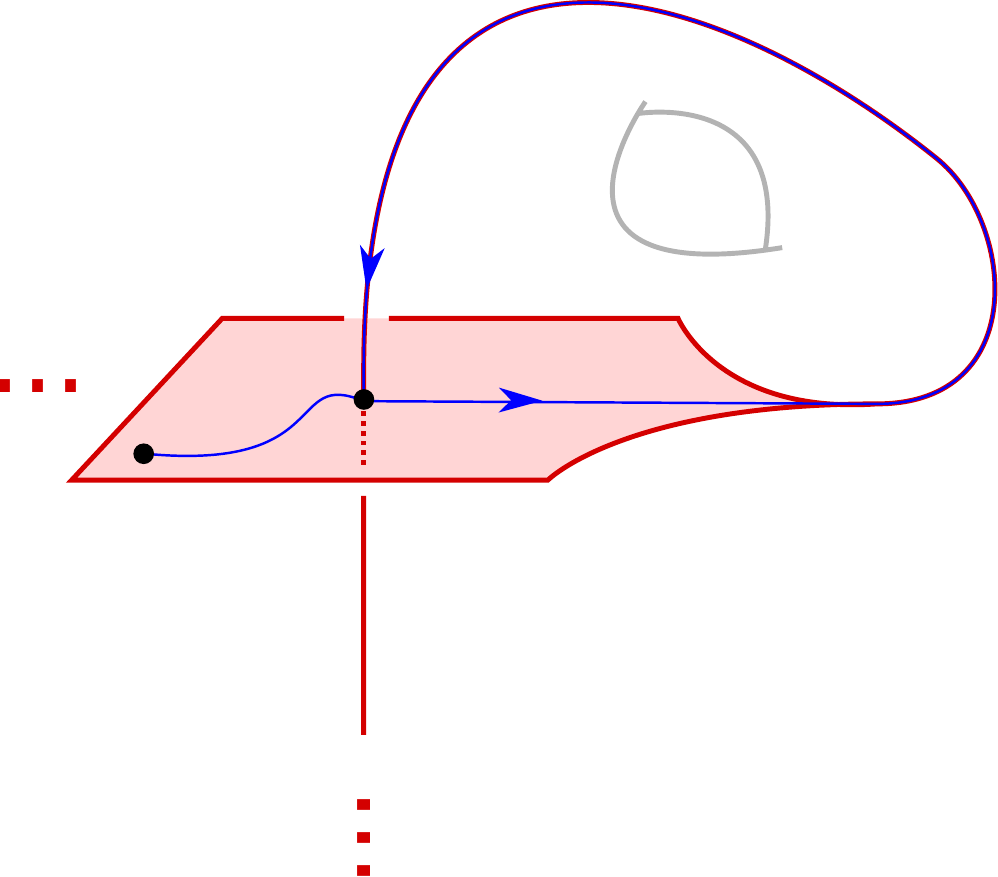}
	\caption{The image of a map $f:S^3\to X^4\times \R^2$ with isolated self-transverse self-intersections, one of which is point $p$. To associate an element of $\pi_1(X)$ to $p$, we choose an arc in $S^3$ between the two preimages of $p$. This arc maps to a loop in $f(S^3)$ through $p$. We then conjugate this loop by an arc contained in the image of $f$ between $p$ and the basepoint to obtain a based loop $\gamma_p$. We write $g_p:=[\gamma_p]\in\pi_1(X)$.}\label{fig:gp}
\end{figure}
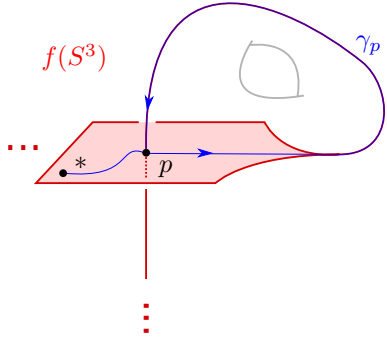

%\blue{MM: Trying to parse this orientation thing -- it doesn't really make sense to say which sheet the arc from the basepoint ends on, right? Because the basepoint is just a point in the 4-manifold, so it ends at the double-intersection point $p$.}

%In the construction of $g_p$, we implicitly chose an orientation of $\gamma_p$ %the arc $\gamma_1\subset S^3$ between the two preimages $f^{-1}(p)$, and hence an ordering of the two local sheets of $f$ near $p$.  %Changing the choice of the direction of the arc between the two preimages $f^{-1}(p)$ %(and hence needing to change the arc from the basepoint so that it goes to the other preimage)
    %changes $g_p$ into $g_p^{-1}=g_p$, so this choice of orientation does not affect the resulting $g_p$. The choice of arc $\gamma_2$ from $z$ to $p$ affects $\gamma_p$ up to conjugation in $\pi_1(X)$. \maggie{MM: I changed a lot of the orientation stuff to just only consider 2-torsion mod 2 in the first place, so it never matters. Is that okay?} %This yields a sign $\epsilon_p\in\{-1,1\}$, defined by the sign of the intersection between the sheets of $f$ near the terminal point and origin point of $\gamma_1$, in that order. 
We write
  %Note also that each double point $p$ has a sign, which we denote $\epsilon_p$, where we use the convention that the first sheet is the sheet that the arc from the basepoint ends.  We then define 
  \[\mu_3(f):= \sum_{p\in f\cap f,\hspace{.05in} g_p\neq 1}  g_p \in \F_2T_X.\label{eq:mu3}\]
  
   % Note that changing the choice of the direction of the arc between the two preimages $f^{-1}(p)$ %(and hence needing to change the arc from the basepoint so that it goes to the other preimage)
  %  changes $g_p$ into $g_p^{-1}=g_p$ and changes $\epsilon_p$ to $-\epsilon_p=\epsilon_p\pmod{2}$.
  As in \cite{ST}, one can verify that $\mu_3:\pi_3(X)\to\F_2 T_X$ is a homomorphism. %one can verify that, in fact, the image of $\mu_3$ lands in $\mathbb{F}_2 T_X$ and $\mu_3$ is a homomorphism.   
\end{definition}

\begin{definition}[{$\fq$, \cite[Section 4]{ST}}]\label{fqdef1}

Let $S_0,S_1: S^2 \hookrightarrow X$ be two based 2-spheres smoothly embedded in a compact, oriented, smooth based 4-manifold $X^4$ and assume that $S_0$ and $S_1$ are based-homotopic. Let $H:(S^2,\ast)\times I\to (X, \ast)\times I$ be an immersion with $H(S^2\times 0)=S_0\times 0$ and $H(S^2\times 1)=-S_1\times 1$ (\eg the track of a based homotopy from $S_0$ to $S_1$). By taking the product of the codomain with $\mathbb{R}$, we have a map from a 3-manifold to a 6-manifold, $M:=H(S^2\times I)\subset X\times\{0\}\times I \subset X \times\R\times I$. By perturbing this map slightly, we obtain an immersion with isolated double points.  As in the previous definition, for each point $p\subset M$ of self-intersection, we obtain a signed element $g_p$ of $T_X \cup \{1\}$. We define
\[\fq(S_0,S_1):=\sum_{p\in M\cap M, g_p \neq 1}g_p\in\F_2 T_X/\mu_3(\pi_3(X)).\]
\end{definition}

The invariant $\fq$ is associated to the pair of spheres $(S_0,S_1)$, and is independent of the map $H$.  Given an element $f$ of $\pi_3(X)$, we could surger $H$ along $f$ to obtain a new based map $H':S^2 \times I\to X\times I$ %Letting $M'$ be a perturbation,$H'(S^2\times I)\times\{0\}$, we find that $\sum_{p\in M\cap M}g_p$ and $\sum_{p'\in M'\cap M'}g_p'$
where $\fq(H')$ and $\fq(H)$ differ by $\mu_3(f)$. We quotient the codomain of $\fq$ by $\mu_3(\pi_3(X))$ to ensure that $\fq$ is well-defined. 

\subsubsection{A 5-dimensional construction of $\fq$}\label{5dsec}
Now we discuss Schneiderman and Teichner's \cite{ST} 5-dimensional method of computing $\mu_3$ and $\fq$.  Our Proposition \ref{fqdef2} is the content of Section 4.D of \cite{ST}.  The following proposition is proved in the same way.  

\begin{proposition}[{\cite{ST}}]\label{mudef2}
Fix a smooth based map $f:S^3\to X\times I$. Perturb $f$ to have transverse self-intersections. The self-intersections of $f$ are circles $\gamma$ which are double-covered by $f^{-1}(\gamma)$. These double covers can either be connected or disconnected; see Figure \ref{fig:kindsofcover}. Let $C_f$ denote the set of circles of self-intersection of $f$ which have connected preimages under $f$. Given $\gamma\in C_f$, we may choose an arc in $f(S^3)$ from the basepoint to $\gamma$ to obtain an element $[\gamma] \in \pi_1(X)$. Since $f^{-1}(\gamma)$ is nullhomotopic in $S^3$, we must have $[\gamma] \in T_X$.

Then we have \[\mu_3(f)=\sum_{\gamma \in C_f, [\gamma] \neq 1}[\gamma]\in \F_2 T_X.\]
\end{proposition}

 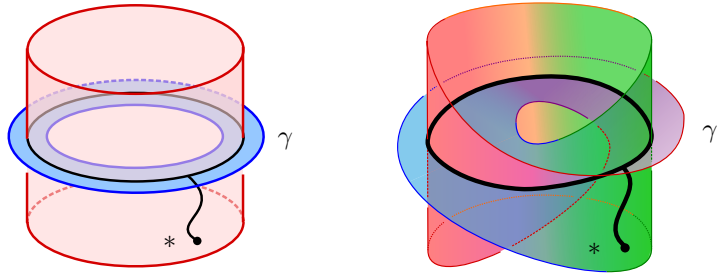
\begin{figure}
 	  \labellist
\small\hair 2pt
 \pinlabel $\gamma$ at 180 85
  \pinlabel $\ast$ at 105 20
 \pinlabel $\gamma$ at 455 90
  \pinlabel $\ast$ at 380 15
\endlabellist
 \includegraphics[width=90mm]{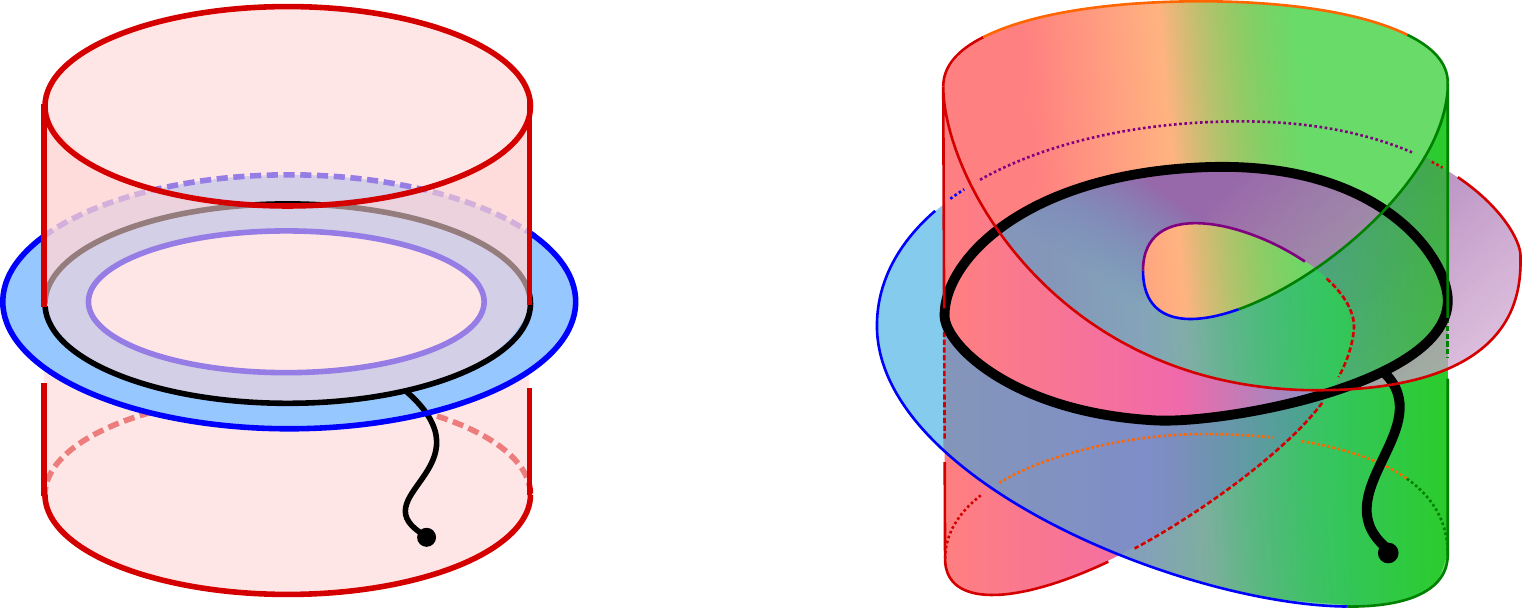}
	\caption{A circle $\gamma$ of self-intersection of an immersed $S^3$ or $S^2\times I$ in $X\times I$. Since the 3-manifold is simply-connected, $\gamma$ represents an element $g$ of $\pi_1(X)$. {\bf{Left:}} $\gamma$ has disconnected preimage under the immersion. We conclude $g=1$ and $\gamma\not\in C_f$. {\bf{Right:}} $\gamma$ has connected preimage under the immersion. We conclude $g^2=1$ and $\gamma\in C_f$. }\label{fig:kindsofcover}
\end{figure}

Note that the orientation of the circles $\gamma$ in the above does not affect the resulting element of $\pi_1$.  Additionally, there is a similar alternative way of computing $\fq$. %, which we now describe; this is the content of Section 4.D of \cite{ST}. 

\begin{proposition}
[{\cite{ST}}]\label{fqdef2}
Let $S_0,S_1: S^2 \to X$ be based 2-spheres smoothly embedded in a compact, oriented, smooth, based 4-manifold $X^4$. Let $H:(S^2,\ast)\times I\to (X,\ast)\times I$ be an immersion with self-transverse self-intersections and $H(S^2\times 0)=S_0\times 0$ and $H(S^2\times 1)=-S_1\times 1$ (\eg the track of a based homotopy from $S_0$ to $S_1$). Let $C_H$ denote the set circles of self-intersection of $H$ which have connected preimage under $H$. Given $\gamma\in C_H$, we may choose an arc in $f(S^3)$ from the basepoint to $\gamma$ to obtain an element $[\gamma] \in \pi_1(X)$. Since $f^{-1}(\gamma)$ is nullhomotopic in $S^2 \times I$, we must have $[\gamma] \in T_X$. Then
\[\fq(S_0,S_1)=\sum_{\gamma \in C_H, [\gamma] \neq 1}[\gamma]\in\F_2 T_X/\mu_3(\pi_3(X))\]
\end{proposition}

\subsection{Computing $\fq$ using intersections of lifts in covering spaces}\label{sec:covers}

For the purposes of this paper, it will be useful to give new methods of computing $\mu_3$ and $\fq$ in terms of covering maps. We will again give 5-dimensional and 6-dimensional variations on the definition of both $\mu_3$ and $\fq$.

\subsubsection{Computing $\fq$ using covering spaces: 6-dimensional version}

Let $X^4$ be a smooth based 4-manifold and choose a basepoint for the universal cover $\widetilde{X}$ above the basepoint of $X$.

\begin{proposition}\label{mudef3}
Fix a based map $f:S^3\to X$ and cross the codomain with $\R^2$  to obtain a new map $f:S^3\to X\times\R^2$ and perturb $f$ to be generic and hence an embedding away from double points of intersection.

Let $Y$ indicate the union of images of all lifts of $f$ to $\widetilde{X}\times\R^2$, and let $Y_1$ be the image of the lift of $f$ based at the basepoint for $\widetilde{X}$. Perturb $Y$ equivariantly so that $Y_1$ intersects other components of $Y$ transversally in isolated points.

Now $\pi_1(X)$ acts on $\widetilde{X}$, and in particular on $Y$ by permuting the lifts of $f$. Let $Y_g= g \cdot Y_1$.  %Given $g\in \pi_1(X)$, let $a_g$ denote the number of points of $Y_1\cap Y_g$ fixed by $g$. Note that if $g\not\in T_X\cup\{1\}$, then $a_g=0$.
Then

\[\mu_3(f)=\sum_{g\in T_X}\frac{1}{2}|Y_1\cap Y_g|\cdot g\in\F_2 T_X.\]
\end{proposition}

\begin{proof}
Let $\pi:\widetilde{X}\to X$ be the covering projection. Fix a point $p\in Y_1\cap Y_g$. Let $q\in Y_1\cap Y_g$ be the $g$-translate of $p$, where $g\in T_X$.

Let $A$ be an arc in $Y_1$ from the basepoint of $Y_1$ to $p$, and $B$ an arc in $Y_g$ from $p$ to $q$. %followed by an arc in $Y_g$ from $p$ to the basepoint of $Y_g$.
Then $\gamma_p:=\pi(\overline{A})\pi(B)\pi(A)$ is a based loop in $X$ with $[\gamma_p]=g_{\pi(p)}$, as in Definition \ref{mudef1}.

%Similarly,
We note that if $s$ is a self-intersection of $\pi(Y)$ with $g_s\neq1$, then $s$ lifts to a point of self-intersection of $Y_1$ and $Y_{g_s}$. Since $g_s$ acts on $Y_1\cap Y_{g_s}$ by permuting pairs of points, there are exactly two lifts of $s$ in $Y_1\cap Y_{g_s}$.

We therefore conclude that for every two points in $Y_1\cap Y_g$, we find one self-intersection $s$ of $\pi(Y)$ with $g_s=g$. Thus, \[\mu_3(f)=\sum_{q\in \pi(Y)\cap\pi(Y),g_q\neq 1}g_q=\sum_{g\in T_X}\frac{1}{2}|Y_1\cap Y_g| \cdot g.\]

\end{proof}

The following result is proved similarly: % to Proposition \ref{mudef3}:
we compare the proposed formula for $\fq$ to that in Proposition \ref{fqdef1}.

\begin{proposition}\label{fqdef3}
Let $S_0, S_1 : S^2 \to X^4$  be based embeddings of spheres in $X^4$ that are based homotopic in $X$.  Let $Y$ be the image of an immersion $H: S^2 \times I \to \widetilde{X} \times \R \times I$ that has $H|_{S^2 \times 0} = S_0$ when we consider $S_0$ as a map into $X \times \{0\} \times \{0\}$ and similarly $H|_{S^2 \times 1} = -S_1$ when we consider $S_1$ as a map into $X \times \{0\} \times \{1\}$.

Given $g \in\pi_1(X)$, let $Y_g:=g \cdot Y_1$. Perturb the $Y_g$ equivariantly so that $Y_1$ and $Y_g$ have isolated transverse intersections for all $g$. Then %Set $a_g$ to be the number of points in $Y_1\cap Y_g$. Note that if $g\not\in T_X\cup\{1\}$, then $a_g=0$.  Then
\[\fq(S_0,S_1)=\sum_{g \in T_X}\frac{1}{2}|Y_1\cap Y_g|\cdot g\in\F_2 T_X/\mu_3(\pi_3(X)).\]
\end{proposition}

\subsubsection{Computing $\fq$ using covering spaces: 5-dimensional version}\label{sec:def4}

Finally, we give yet another pair of definitions of $\mu_3$ and $\fq$: 5-dimensional definitions involving coverings. These constructions will be essential in the proof of Theorem \ref{maintheorem}.

\begin{proposition}\label{mudef4}
Fix a based map $f:S^3\to X$ and cross the codomain with $\R$  to obtain a new map $f:S^3\to X\times\R$ and perturb $f$ to be generic and hence an embedding away from circles of double points.  

Let $Y$ indicate the union of images of all lifts of $f$ to $\widetilde{X}\times\R$, and let $Y_1$ be the image of the lift of $f$ based at the basepoint for $\widetilde{X}$. Perturb $Y$ equivariantly so that $Y_1$ intersects other components of $Y$ transversally in circles.

Now $\pi_1(X)$ acts on $\widetilde{X}$, and in particular on $Y$ by permuting the lifts of $f$. Let $Y_g= g \cdot Y_1$.  Given $g\in \pi_1(X)$, let $c_g$ denote the number of components of $Y_1\cap Y_g$ fixed by $g$. Note that if $g\not\in T_X\cup\{1\}$, then $c_g=0$. %Note moreover that the action of $\pi_1(X)$ has no fixed points, so if $g$ fixes a circle $C$, then $g$ acts on $C$ by rotation through $\pi$.
Then

%Fix a smooth map $f:S^3\to X\times I$. Lift the image of $f$ to a 3-sphere $Y_1$ immersed in $(\widetilde{X}\times I,\widetilde{z})$ and perturb $Y_1$ to have transverse self-intersections and transverse intersections with translates $Y_\gamma:=\gamma Y_1,\gamma\in\pi_1X$. Now the intersection of $Y_1\cap Y_\gamma$ is a collection of circles. Set $a_\gamma$ to be the number of components of $Y_1\cap Y_\gamma$ which are individually fixed setwise by $\gamma$. Note that if $\gamma\not\in T_X\cup\{1\}$, then $\gamma$ cannot fix any intersection $Y_\alpha\cap Y_\beta$ setwise. (Note moreover that $\gamma$ has no fixed points, so if $\gamma$ fixes a circle $C$ setwise, then $\gamma$ acts on $C$ by rotation through $\pi$.) %The self-intersections of $f$ are circles $\gamma$ which are double-covered by $f^{-1}(\gamma)$. These double covers can either be connected or disconnected; see Figure \ref{fig:kindsofcover}. Let $C_f$ denote the set of circles of self-intersection of $f$ which have connected preimages under $f$. Given $\gamma\in C_f$, we may choose an arc in $f(S^3)$ from $z$ to $\gamma$ to view $\gamma$ as an element of $\pi_1(X)$. Since $f^{-1}(\gamma)$ is nullhomotopic in $S^3$, we must have $\gamma\in T_X$.

\[\mu_3(f)=\sum_{g \in T_X}{c_g}\cdot g\in \F_2 T_X.\]
\end{proposition}

%\begin{remark}\label{rem:rotation}
%Note in Definition \ref{mudef4}, if $\gamma\not\in T_X\cup\{1\}$ then $\gamma$ cannot fix any intersection $Y_\alpha\cap Y_\beta$ setwise. For $\gamma\in T_X$, $\gamma$ might fix some circle $C$ in $Y_1\cap Y_\gamma$ setwise, but $\gamma$ has no fixed points. We conclude that $\gamma$ acts on $C$ by rotation through an angle of $\pi$.
%\end{remark}

\begin{proof}
We will prove this by showing that the given formula for $\mu_3$ is equivalent to that in Proposition \ref{mudef2}.

	%We will use the singular double circles way of computing $\mu_3$, as described above.
	Let $p:\widetilde{X}\to X$ denote the universal covering map. Fix $f \in \pi_3(X)$, a generic perturbation $f : S^3 \to X \times I$, and an element $g \in T_X$.  Recall from Proposition \ref{mudef2} that every circle of self-intersection in the image of $f$ corresponds to an element of $\pi_1(X)$. We give bijections

%$$
%	\{ C \in \pi_0(Y_1 \cap Y_\gamma) : \text{$\gamma$ fixes $C$} \}  \xleftrightarrows[\text{preimage contained in $Y_1 \cap Y_\gamma$}]{\text{project}} \{ \text{Self-intersection circles in image of $f$ corresponding to $\gamma$ which have connected preimage under $f$} \}
%$$

$$
	\{\pi_0(Y_1 \cap Y_g)\mid\text{$g(C)=C$} \}  \xleftrightarrows[\text{p}] %contained in $Y_1 \cap Y_\gamma$}]
{\text{$p^{-1}\cap (Y_1\cap Y_g)$}}  \{\text{$\pi_0(f\cap f)\mid [C]=g$, $f^{-1}(C)$ connected}\}.$$
	First, we show that these maps are well-defined and then we verify that they are inverses of one another. Let $C$ be a circle of intersection of $Y_1$ and $Y_g$ % \in \pi_0(Y_1 \cap Y_g)$
	that is fixed setwise by $g$ (\ie $g$ acts on $C$ by rotation through $\pi$). %since the action is free, it must be rotation by $\pi$, so projecting $\widetilde{\gamma}$
	Then $p:C\to p(C)$ is a double cover of a circle $p(C)$ of self-intersection in the image of $f$, where $p(C)$ represents $g$. Since $C$ is connected, $f^{-1}(p(C))$ is connected. %double covers its image, and therefore gives a singular circle representing $g$. 

	On the other hand, suppose $C$ is a circle of self-intersection in the image of $f$ corresponding to $g\in T_X$ with $f^{-1}(C)$ connected. Then $\widetilde{C}:=p^{-1}(C)\cap(Y_1\cap Y_g)$ double covers $C$. Since $g$ is nontrivial, $\widetilde{C}$ is a connected circle formed by two arc lifts of $C$. The action of $g$ permutes these arcs, fixing $\widetilde{C}$ setwise. %, acting by rotation through $\pi$. % (\ie permutation of the two subarc lifts of $C$).%either consists of two circles permuted by $\gamma$ or one circle fixed setwise by $\gamma$.

	% connectedsingular circle in the image of $f$ representing $g$, then the preimage in $\pi_0(f_1 \cap f_g)$ is either 2 circles interchanged by the action of $g$, or a single circle fixed by $g$.  But the former case can not happen since $\gamma$ represents a nontrivial element of $\pi_1(X)$.  Thus both maps above are well-defined.  

	Now consider the two defined maps, $p$ (projection) and $p^{-1}\cap(Y_1\cap Y_g)$ (lifting to $Y_1\cap Y_g$). The composition in either order is clearly the identity. %Taking the preimage and then the projection immediately gives the identity. 
	%Moreover, the projection map is injective since for $C \in \pi_0(Y_1 \cap Y_g)$, we have $C=p^{-1}(p(C))\cap(Y_1\cap Y_g)$. %is edouble covers its projection $p(C)$ and every point on every circle in the image of $f$ representing $g$ has exactly two preimages with respect to the projection restricted to $f_1 \times f_g$.
	This completes the proof.   
\end{proof}

The following result follows by the same argument as in Proposition \ref{mudef4}. We similarly compare the proposed formula for $\fq$ to that in Proposition \ref{fqdef2}.

\begin{proposition}\label{fqdef4}
Let $S_0, S_1 : S^2 \to X^4$  be based embeddings of spheres in $X$ that are based-homotopic in $X$.  Let $H$ be an immersion $H: (S^2,\ast) \times I \to (X,\ast) \times I$ that has $H|_{S^2 \times 0} = S_0$ when we consider $S_0$ as a map into $X \times \{0\}$ and similarly $H|_{S^2 \times 1} = -S_1$ when we consider $S_1$ as a map into $X \times \{1\}$. 

Let $Y$ indicate the union of images of all lifts of $H$ to $\widetilde{X}\times I$, and let $Y_1$ be the image of the lift of $f$ based at the basepoint for $\widetilde{X}$. Perturb $Y$ equivariantly so that $Y_1$ intersects other components of $Y$ transversally in circles.

Now $\pi_1(X)$ acts on $\widetilde{X}$, and in particular on $Y$ by permuting the lifts of $H$. Let $Y_g= g \cdot Y_1$.  Given $g\in \pi_1(X)$, let $c_g$ denote the number of components of $Y_1\cap Y_g$ fixed by $g$. Note that if $g\not\in T_X\cup\{1\}$, then $c_g=0$. %Note moreover that the action of $\pi_1(X)$ has no fixed points, so if $g$ fixes a circle $C$, then $g$ acts on $C$ by rotation through $\pi$.
Then

\[\fq(S_0,S_1):=\sum_{g\in T_X}c_g\cdot g\in\F_2 T_X/\mu_3(\pi_3(X)).\]
\end{proposition}

\subsection{Based homotopy versus free homotopy}\label{sec:htpy}
In this section, we have given many definitions of $\fq(S_0,S_1)$ when $S_0,S_1$ are based-homotopic 2-spheres in a 4-manifold $X^4$. In the main theorem of this paper, we consider 2-spheres which are homotopic, but do not specify that this homotopy is based. Schneiderman and Teichner \cite[Lemma 2.1]{ST} show that when $S_0$ and $S_1$ have a common geometric dual (and agree near the dual), then if $S_0$ and $S_1$ are homotopic then they are based-homotopic. We need an analogous lemma in order to deal with basepoints.

\begin{proposition}\label{lem:nobasegen}
Let $S_0$ and $S_1$ be 2-spheres smoothly embedded in a smooth 4-manifold $X^4$. Let $G$ be a 2-sphere immersed in $X$. Assume that $S_0$ intersects $G$ transversally once at a point $z$ and that $S_0$ and $S_1$ are homotopic. Then after an isotopy of $S_1$, $S_0$ and $S_1$ are based-homotopic with basepoint $z$.
\end{proposition}

\begin{proof}
Take $H:S^2\times I\to X\times I$ to be a free homotopy of $S_0$ to $S_1$. Note $S_i$ represents an element $[S_i]\in\pi_2(X,z)$. Fix $z_0\in S^2$ so that $H(S^2\times 0)=z\times 0$. Then $H(z_0\times I)$ represents an element $g$ of $\pi_1(X,z)$ with $g[S_0]=[S_1]$.

Take $H$ to be transverse to $G$ and consider $L=H^{-1}(G\times I)$, a properly embedded 1-dimensional submanifold of $S^2\times I$. Note the boundary of $L$ consists of $z_0\times0$, and an odd number of points of the form $y_i\times 1$. Let $L_0$ be the component of $L$ containing $z_0\times 0$ and assume the other endpoint of $L_0$ is $y_0\times 1$. Compose $H$ with an isotopy of $S_1$ taking $H(y_0)$ to $z$ by isotopy in $G$ and set the resulting composed homotopy to be $H$. Now $\partial L_0=z_0\times\{0,1\}$. Since $S^2\times I$ is simply-connected, this implies $H(L_0)$ represents $g$. Since $H(L_0)\subset G\times I\cong S^2\times I$ (away from self-intersections), $g\in\pi_1(X,z)$ is the identity and we conclude $[S_0]=[S_1]\in\pi_2(X,z)$. See Figure \ref{fig:basedhtpy} for an illustration.
\end{proof}

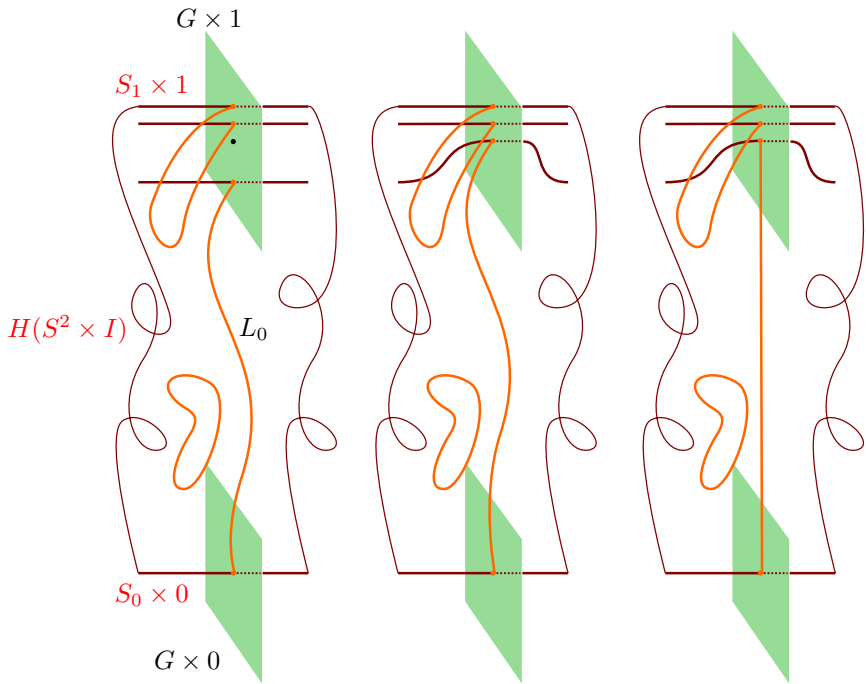
\begin{figure}
	  \labellist
\small\hair 2pt
 \pinlabel $\textcolor{red}{S_1\times 1}$ at 90 1350
 \pinlabel $\textcolor{red}{S_0\times 0}$ at 90 200
  %\pinlabel $z\times 0$ at 250 200
  % \pinlabel $z\times 1$ at 400 1250
 \pinlabel $G\times 0$ at 170 50
 \pinlabel $G\times 1$ at 220 1500
 \pinlabel $\textcolor{red}{H(S^2\times I)}$ at -100 800
 \pinlabel $L_0$ at 320 800
\endlabellist
    \centering
    \includegraphics[width=100mm]{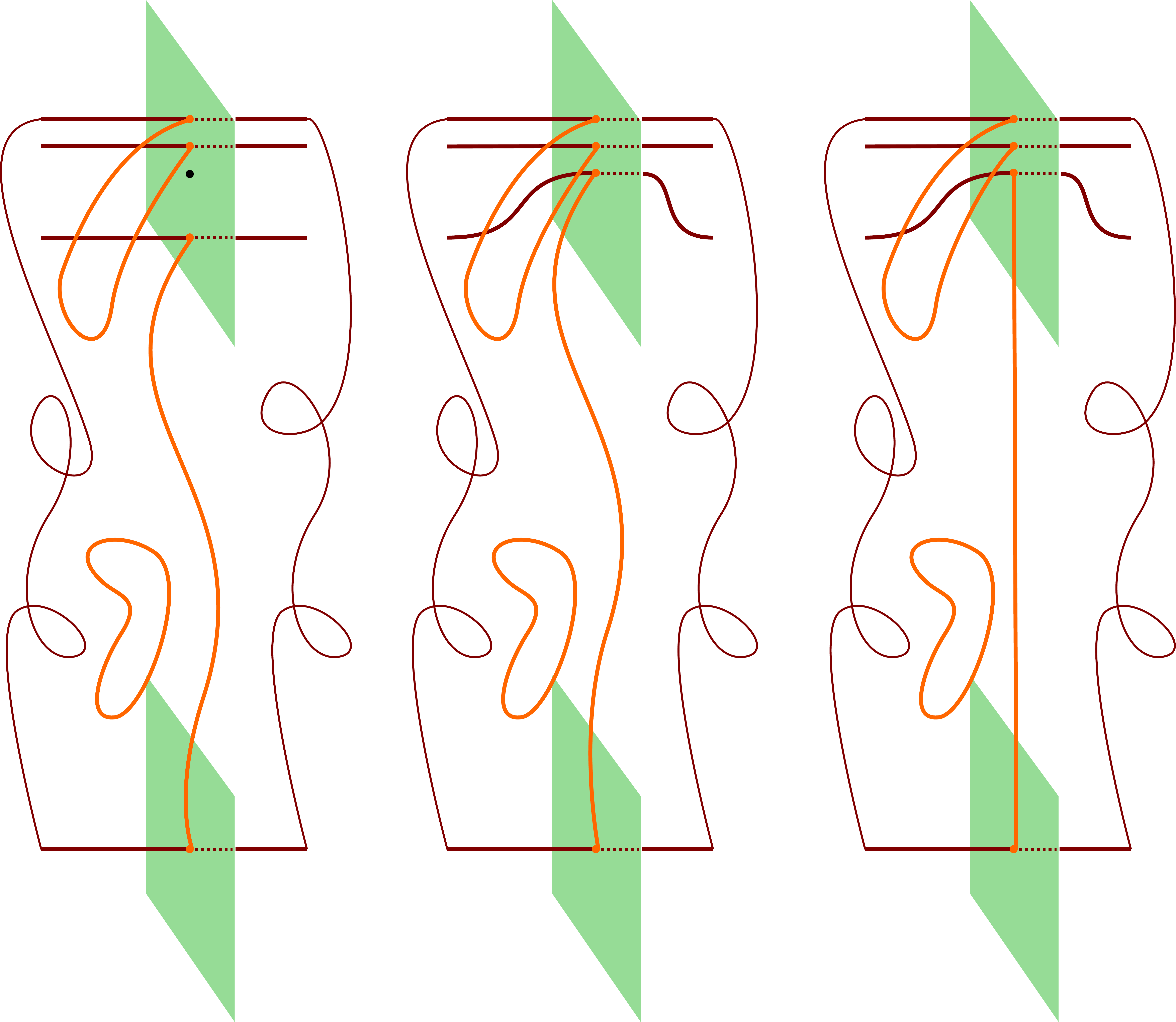}
    \caption{The proof of Lemma \ref{lem:nobasegen}. {\bf{Left:}} $H$ is a homotopy from $S_0$ to $S_1$. The arc $L_0$ is the component of $H^{-1}(G\times I)$ with $z\times 0$ as an endpoint. {\bf{Middle:}} We isotope $S_1$ near $G$ and concatenate with $G$ to find a new homotopy where $L_0$ has boundary $z\times\{0,1\}$. {\bf{Right:}} Let $p:X\times I\to X$ be projection. Since $p\circ H(L_0)$ is contained in $G$ (away from self-intersections of $G$), $p\circ H(L_0)$ is a nullhomotopic loop. Contracting this loop yields a homotopy based at $z$.}
    \label{fig:basedhtpy}
\end{figure}

In Theorem \ref{maintheorem} (or in the 4D light bulb theorem), we assume that spheres $S_0$ and $S_1$ are homotopic. By Proposition \ref{lem:nobasegen}, there exists another 2-sphere $S'_1$ isotopic to $S_1$ so that $S_0$ and $S'_1$ are based-homotopic. The sphere $S_0$ is concordant (resp. isotopic) to $S_1$ if and only if it is concordant (resp. isotopic) to $S'_1$. Thus, we may without loss of generality take $S_0$ to be based-homotopic to $S_1$.

\section{Concordance of surfaces in 4-manifolds}\label{sec:2spheres}

In \cite{NS}, Sunukjian proves the following theorem that generalizes the result of Kervaire that all 2-knots are slice \cite{kervaire}:

\begin{theorem}[{\cite[Theorem 6.1]{NS}}] \label{simply-connected}
Let $X^4$ be a simply-connected 4-manifold. Let $\Sigma_0$ and $\Sigma_1$ be compact, oriented, embedded, homologous surfaces in $X^4$ of the same genus.  Then $\Sigma_0$ and $\Sigma_1$ are concordant.% \subset X \times \{0\}$ and $\Sigma_1 \subset X \times \{1\}$ are concordant.   
\end{theorem}

Sunukjian then goes on to use Theorem \ref{simply-connected} to study manifolds with nontrivial fundamental group \cite[Theorem 6.2]{NS}. The main strategy is to require $\pi_1(\Sigma_i)$ to include trivially into the 4-manifold so that $\Sigma_i$ can be lifted to the universal cover $\widetilde{X}$. In $\widetilde{X}\times I$, the various lifts of $\Sigma_0\times\{0\},-\Sigma_1\times\{1\}$ then cobound embedded copies $Y_g,g\in\pi_1(X)$ of $\Sigma_i\times I$ which pairwise intersect. In the cover, one may attempt to equivariantly surger these 3-manifolds to force them to be disjoint, and then use the proof methods of the simply-connected case to do further equivariant surgery to obtain a disjoint collection of concordances. These concordances then project to a concordance in $X\times I$. 

%In Theorem 6.2 of \cite[Thm. 6.2]{NS}, the author uses the above and equivariant surgery to try to prove a corresponding criterion for nonsimply-connected 4-manifolds.
There is a problem in this approach, as pointed out in \cite[\S1]{Maggie}, that calls attention to 2-torsion in $\pi_1(X)$. %Particularly, a 
A difficulty arises when some element $g$ of $\pi_1(X)$ fixes a circle $C$ of intersection in $Y_1\cap Y_g$. If we attempt to surger $Y_1$ at $C$ to remove this intersection, then in order to do the surgery equivariantly we must then surger $Y_g$ at $C$ as well -- potentially causing us to accidentally create a new circle of intersection near $C$ if we are not careful during this surgery operation. In fact, this obstacle is precisely the motivation for the definitions of $\mu_3$ and $\fq$ given in Propositions \ref{mudef4} and \ref{fqdef4}. There is an additional subtlety that requires us to assume that the dual sphere $G$ is framed; this is necessary in one of the constructive moves used by Sunukjian  (Move \ref{move1}).

%Here we modify the hypothesis and, following the ideas in \cite{NS}, give a result corresponding to Theorem \ref{simply-connected} but with $\pi_1(X)$ non-trivial.  The primary new hypothesis that we add is the vanishing of the Freedman-Quinn invariant.  %In the next section, we will give an example that shows the necessity of our condition on the Freedman-Quinn invariant.  

%\mike{Wait - fq is clearly necessary by your remark that you can just compute fq = 0 using the concordance.  We want to say that we give an example that shows the necessity of the meridian vanishing, eh?}

%\begin{maintheorem}
%Let $S_0$ and $S_1$ be two homotopic embedded 2-spheres in a 4-manifold $X^4$. Assume that the meridian of $S_0$ is null-homotopic in $X - S_0$.  Then $S_0$ and $S_1$ are concordant in $X \times I$ if and only if $\fq(S_0,S_1)=0$. 

%\end{maintheorem}

Before proving Theorem \ref{maintheorem}, we review the techniques used in \cite{NS} to modify embedded 3-manifolds in a 5-manifold. For our modification of the argument, we will also need to use these techniques on immersed 3-manifolds in a 5-manifold.  The general %idea is that we are given a 3-manifold inside of a 5-manifold, and we would like to perform Dehn surgeries on the 3-manifold to change its topology, but we must perform these surgeries ambiently inside of the 5-manifold.  
principle is to surger the 3-manifold along 4-dimensional handles embedded in the ambient 5-manifold. 

\begin{definition}[Ambient Dehn 1-surgery]

Let $Y^3$ be a 3-manifold embedded in a 5-manifold $W^5$.  Suppose that we are given an arc $\alpha$ in $W$ with endpoints on $Y$ and interior away from $Y$. Frame $\alpha$; now the unit 3-ball bundle over $\alpha$ yields a $I\times B^3$, along the boundary of which we may surger $Y$ to obtain an embedded 3-manifold $Y'$. Choose the framing of $\alpha$ so that $Y'\cong Y\# S^1\times S^2$. We say that $Y'$ is obtained from $Y$ by {\emph{ambient Dehn 1-surgery}} along $\alpha$. The ``1" refers to the fact that $Y$ is being surgered along a 4-dimensional 1-handle.
\end{definition}

\begin{definition}\label{def:ambientDehn}[Ambient Dehn 2-surgery]
Let $Y^3$ be an oriented 3-manifold embedded in an oriented 5-manifold $W^5$.  Suppose that we are given an oriented circle $\gamma \subset Y$ with an integral framing.
If some pushoff of $\gamma$ into $W$ is null-homotopic in $W$, then there exists an embedded disk $\Delta \subset W$ with $Y \cap \Delta = \gamma$ ($\Delta$ is embedded since we are in dimension 5).  If $\Delta$ has the property that there exists a 2-dimensional subbundle $B$ of its trivial 3-dimensional normal bundle $N_W(\Delta)$ such that $B$ has a trivialization extending the trivialization coming from the framing on the boundary circle in $N_Y(\gamma)$, then the unit disk bundle over $\Delta$ in $B$ yields a copy of $D^2 \times D^2$ in $W$ that intersects $Y$ in a tubular neighborhood of $\gamma$, such that $(Y - S^1 \times D^2) \cup (D^2 \times S^1)$ is the 3-manifold obtained by performing the Dehn surgery on $Y$ along $\gamma$ using the specified framing. We call this move {\emph{ambient Dehn 2-surgery}}, where the ``2" refers to the fact that $Y$ is being surgered along a 4-dimensional 2-handle.  

For a framing $\phi$ of $N_Y \gamma$ and a given disk $\Delta$, we say that a framing  is a \emph{$\Delta$-admissible} framing if such a subbundle $B \subset N_W(\Delta)$ as above exists, with a framing on $B$ that extends $\phi$.  
\end{definition}

\begin{lemma}\label{lem:framing}
    Fix a framing $\phi$ of $N_Y(\gamma)$ where the notation is as in Definition \ref{def:ambientDehn}.  Then the $\Delta$-admissible framings on $\gamma$ are exactly the result of of the even integers acting on $\phi$, or the result of all of the odd integers acting on $\phi$, where the free and transitive action of $\mathbb{Z}$ on the set of framings of $\gamma$ is as in Figure \ref{fig:framinginteger}.
\end{lemma}

	\begin{figure}
	    \centering
	    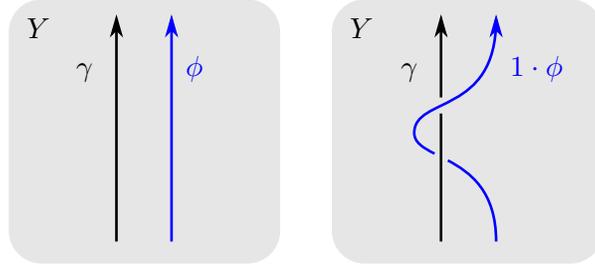
	    \caption{{\bf{Left:}} A framing $\phi$ of $N_Y(\gamma)$, where $\gamma$ is a curve in 3-manifold $Y$. {\bf{Right:}} There is an action of $\mathbb{Z}$ on the set of framings of $N_Y(\gamma)$. Here we draw $1\cdot\phi$.}
	    \label{fig:framinginteger}
	\end{figure}

\begin{proof}
Note that there is a unique framing, up to homotopy through framings $F : N_W(\Delta) \xrightarrow{\sim} D^2 \times \mathbb{R}^2$.  By restricting $F$ to the bundle over $\gamma$, we have 
$$
F : N_W(\gamma) = N_Y(\gamma) \oplus \xi \xrightarrow{\sim} S^1 \times \mathbb{R}^3
$$
where $\xi$ is the line bundle as is Figure \ref{fig:framings}.  The framing $\phi$ gives another framing 
$$
\phi \oplus i : N_Y(\gamma) \oplus \xi \xrightarrow{\sim} S^1 \times \mathbb{R}^3
$$
where $\xi$ is oriented as in Figure \ref{fig:framings} and $i$ is the corresponding trivialization.  Comparing these two trivializations, we have
$$
(\phi \oplus i) \circ F^{-1} : S^1 \times \mathbb{R}^3 \xrightarrow{\sim} S^1 \times \mathbb{R}^3
$$
which can be thought of as a map $[\phi] \in \pi_1(\text{GL}^+(\mathbb{R}^3)) \cong \pi_1(\text{SO}(3)) \cong \mathbb{Z}/2$.

	\begin{figure}
	    \centering
	    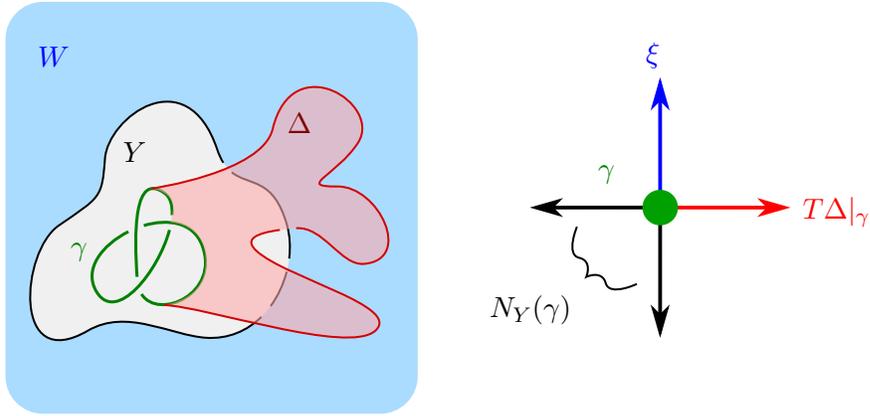
	    \caption{{\bf{Left:}} A schematic of a disk $\Delta$ with boundary $\gamma\subset Y^3\subset W^5$. {\bf{Right:}} A cross-section of $\gamma$. The normal bundle $N_\gamma(W)$ is 4-dimensional. It contains the 2-dimensional subbundle $N_Y(\gamma)$ and the line bundles $T\Delta|_\gamma$ and $\xi$, all of which are normal to each other.}
	    \label{fig:framings}
	\end{figure}

Now $\phi$ is a $\Delta$-admissible framing if and only if there exists a 2-dimensional subbundle $B \subset N_W \Delta$ extending $N_Y(\gamma)$ together with a framing $\Phi : B \xrightarrow{\sim} D^2 \times \mathbb{R}^2$ such that the following commutes:
$$
\begin{tikzcd}
B \arrow{r}{\Phi} & D^2 \times \mathbb{R}^2  \\
N_Y(\gamma) \arrow{r}{\phi} \arrow[u, hook,"\partial"] & S^1 \times \mathbb{R}^2 \arrow[u, hook,"\partial"]
\end{tikzcd}
$$
or equivalently, if the trivialization $F$ can be chosen so that the following commutes:
$$
\begin{tikzcd}
N_W(\Delta) \arrow{r}{F} & D^2 \times \mathbb{R}^2  \\
N_Y(\gamma) \oplus \xi \arrow{r}{\phi} \arrow[u, hook,"\partial"] & S^1 \times \mathbb{R}^3 \arrow[u, hook,"\partial"]
\end{tikzcd}
$$
But this is equivalent to $[\phi] = 0$.  To complete the proof, note that when $1 \in \mathbb{Z}$ acts on $\phi$ as in Figure \ref{fig:framinginteger}, $[1 \cdot \phi] \neq [\phi]$.
\end{proof}

\begin{lemma}\label{lem:tubetoframed}
    Let $\Delta$ be a disk in a 5-manifold with $\gamma:=\boundary\Delta$ contained in a 3-manifold $Y$ (and $\mathring{\Delta}\cap Y=\emptyset$). Let $R$ be a 2-sphere in $W$ with trivial normal bundle that is disjoint from $\Delta$ and $Y$. Let $\widetilde{\Delta}$ be a disk obtained from $\Delta$ by tubing $\Delta$ to $R$. Then a framing $\phi$ of $N_Y(\gamma)$ is $\Delta$-admissible if and only if it is $N_Y(\widetilde{\Delta})$ admissible.
\end{lemma}

\begin{proof}
See Figure \ref{fig:tubetoframed}.

Note that $\Delta$ can be obtained from $\widetilde{\Delta}$ by tubing $\widetilde{\Delta}$ to a parallel copy of $R$ with opposite orientation, so it is sufficient to prove one direction of implication.

Let $\phi$ be a framing of $N_Y(\gamma)$, i.e. a choice of two non-vanishing vector fields $v_1,v_2$ on $\nu(\gamma)$ that span $N_Y(\gamma)|_p$ for each $p$ in $\gamma$. Assume $\phi$ is $\Delta$-admissible, so that $v_1$ and $v_2$ can be extended to non-vanishing vector fields $V_1$ and $V_2$ (respectively) on a neighborhood of $\Delta$ so that their restriction to $\Delta$ spans some 2-dimensional subbundle $B_\Delta$ of $N_W(\Delta)$. 

Now $\widetilde{\Delta}$ is obtained by tubing $\Delta$ to $R$ along a framed arc $\delta$. Homtope $V_1,V_2$ so that $V_1(p)$ and $V_2(p)$ agree with the framing of $\delta$ at $p=\delta\cap\Delta$, and then use the framing of $\delta$ to extend $V_1,V_2$ to a neighborhood of $\delta$.

Since $R$ has trivial normal bundle, we can choose non-vanishing vector fields $\omega_1,\omega_2$ on near $R$ whose restriction to $R$ spans a 2-dimensional subbundle $B_R$ of $N_W(\Delta)$ whose total space is $R\times D^2$. Moreover, we can take $\omega_1,\omega_2$ to agree with $V_1,V_2$ near $\delta\cap R$. Therefore, we can combine $V_1$ with $\omega_1$ and $V_2$ with $\omega_2$ to obtain a framing of $\widetilde{\Delta}$ that extends $\phi$. We conclude that $\phi$ is $\widetilde{\Delta}$-admissible.

On the other hand, if $\phi$ is {\emph{not}} $\Delta$-admissible, then $1\cdot\phi$ is $\Delta$-admissible. By the above argument, $1\cdot\phi$ is $\widetilde{\Delta}$-admissible. Then by Lemma \ref{lem:framing}, $\phi$ is not $\widetilde{\Delta}$-admissible.

\end{proof}

\begin{figure}
    \centering
    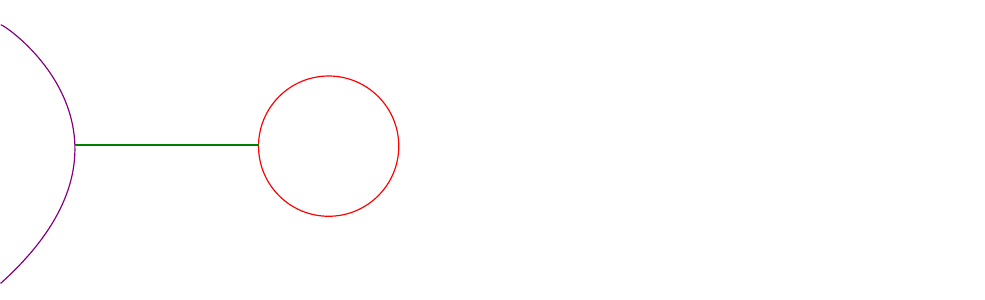
    \caption{Here, $\widetilde{\Delta}$ is a disk obtained by tubing a disk $\Delta$ to a framed sphere $R$ along a framed arc $\delta$. Given a 2-dimensional framing $\phi$ of $\Delta$, we can amalgamate framings of $\delta$ and $R$ with $\phi$ to obtain a framing of $\widetilde{\Delta}$ that agrees with $\phi$ near $\boundary\widetilde{\Delta}=\boundary\Delta$.}
    \label{fig:tubetoframed}
\end{figure}

%\begin{remark}
%Note that when performing ambient Dehn 2-surgery along a disk $\Delta$, the choice of framing on $\Delta$ determines the framing of the underlying Dehn surgery. Given an unframed $\Delta$, we can certainly frame $\Delta$ and perform {\emph{some}} ambient Dehn 2-surgery, but we may not be able to frame $\Delta$ to produce a desired framing. In the ensuing arguments, we will generally not care about the framing of the underlying Dehn surgery, so do not have to overly worry about framings on disks.   Framing issues are handled by Sunukjian in the proof of Theorem \ref{simply-connected}, which we cite.%This is because all framing issues are handled by Sununkjian in proving the simply-connected version of our theorem.  
%\end{remark}

We can also perform ambient Dehn 2-surgeries along many disjoint disks simultaneously. That is, suppose that say $\pi_1(W-Y) = 1$, and we are given several disjoint curves $\gamma \subset Y$ along which we want to do an ambient Dehn surgery.  The existence of disjoint embedded disks meeting $Y$ in $\gamma$ is guaranteed by our condition that $\pi_1(W-Y) = 1$ (again since we are in a 5-manifold surfaces will generically not intersect). Again, we may use the disks to perform ambient Dehn 2-surgery to $Y$ along $\gamma$. The choice of framing of the disks determines the framing of this Dehn surgery. If we initially specify a choice of framing of each curve in $\gamma$, then there is no guarantee that the disks will have the desired 2-dimensional subbundles of their normal bundles such that the framing trivializations extend.

In the proof of Theorem \ref{simply-connected} in \cite{NS}, Sunukjian begins by taking some 3-manifold $Y$ in $X^4 \times I$ with boundary $\Sigma_0 \times\{0\}\sqcup -\Sigma_1\times\{1\}$.  The fundamental group conditions ensure the existence of disks bounding surgery curves in $Y$. % will always exist.
The %n spin structures and the
fact that all 3-manifolds are spin-cobordant is used to show that a framed link $\gamma$ that surgers $Y$ to $\Sigma_0 \times I$ can be chosen so that the specified surgery can be carried out ambiently, as described above.  In our proof we will be using Theorem \ref{simply-connected}, or rather the following result which is proved implicitly in \cite{NS}:

\begin{theorem}[{\cite[Theorem 6.1]{NS}}] \label{simply-connected - heart}
Let $Y^3 \subset W^5$ be a properly embedded submanifold, where $Y$ is compact, $W$ is not necessarily compact and is simply-connected, and $\pi_1(W-Y)$ is cyclic. Let $Y'$ be any other compact 3-manifold with $\partial Y' \cong \partial Y$.  Then there is an ambient Dehn surgery that can be performed on $Y$ in $W^5$ that yields an embedded 3-manifold diffeomorphic to $Y'$.  
\end{theorem}

There are two %other instances
situations in which %where
we will use ambient Dehn 2-surgery. Together, these two moves will be used to eliminate double points of 3-manifolds immersed in 5-manifolds, assuming that the 3-manifolds have framed dual spheres.

\begin{move}[Remove an intersection]\label{move2}
Let $Y_1, Y_g$ be embedded 3-manifolds in $W^5$ with a single circle of self-intersection $\gamma$ so that $\gamma$ is an unknot in $Y_1$.  Let $\Delta$ be a disk bounded by $\gamma$ contained in $Y_1$. Let $B$ be the 2-dimensional subbundle of $N_W(\Delta)$ that is normal to $Y_1$. Then by performing ambient Dehn 2-surgery on $Y_g$ along $\Delta\times D^2$ in $B$, we obtain a new 3-manifold $Y'_g$ that is disjoint from $Y_1$. Thus, this move, ``eliminates double points."  See Figure \ref{fig:ambientdehn} (bottom row) or Figure 6 in \cite{NS}.
\end{move}

%To reiterate, %All issues involving framings 
%All of the issues with worrying about framings
%appear only in the proof of Theorem \ref{simply-connected - heart}. % and we will use this result and ultimately not need to worry about framings.  
%for the purposes of our argument, we generally do not need to worry about framings on curves or disks. (That is, we generally want to perform ambient Dehn surgery on a specific link without minding the specific surgery framing.)

	\begin{figure}
	    \centering
	    \includegraphics[width=125mm]{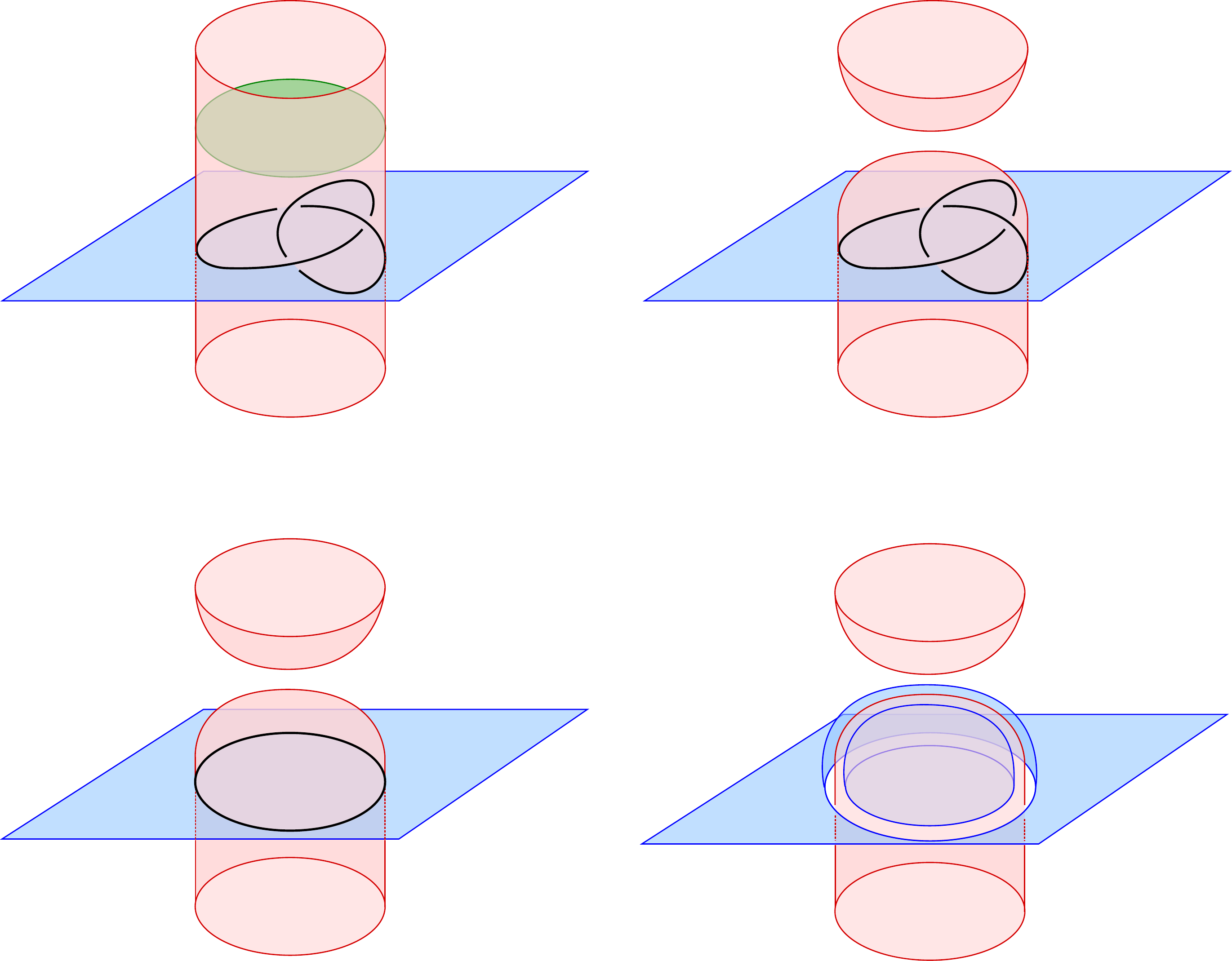}
	    \caption{{\bf{Top Left:}} Two 4-manifolds $Y_1$ and $Y_g$ intersect in a circle $C$. (If $Y_1=Y_g$, then this is a self-intersection.) {\bf{Top Right:}} We perform ambient Dehn 2-surgery on $Y_1$ at a solid torus parallel to $C$ (according to the surgery framing). Note that we require $Y_g$ to have a framed dual sphere in order to ensure that this framing is $\Delta$-admissible for some disk $\Delta$. Now the circle of intersection is an unknot in $Y_1$, but this changes the topology of $Y_1$. {\bf{Bottom Left:}} The circle $C$ bounds a disk $D$ in $Y_1$. {\bf{Bottom Right:}} We use $D$ perform ambient Dehn 2-surgery on $Y_g$. The chosen framing of $D$ is trivializable over the 2-dimensional subbundle of $N_W(D)$ that is normal to $Y_1$. Then after the surgery, we have removed the circle $C$ of intersection (without introducing any new intersections). This changes the topology of $Y_g$. }
	    \label{fig:ambientdehn}
	\end{figure}

	Move \ref{move2} may seem to have overly strong hypotheses on when it can be performed: there is generally no reason to suspect that $Y_1\cap Y_g$ has any unlinked, unknotted components. However, we now consider another move that can simply the intersection link between $Y_1$ and $Y_g$.

\begin{move}[Unknot a self-intersection]\label{move1}

Take $Y_1$, $Y_g$ to be 3-manifolds immersed in $W$ and assume $Y_g$ has a framed dual sphere $G$ that does not intersect $Y_1$. %Since $W$ is 5-dimensional, we can perturb now without loss of generality we can take $G$ to be embedded.

Note that $Y_1$ and $Y_g$ intersect in a 1-manifold. %therefore there is a 1-manifold of self-intersections, and
For simplicity, we will assume that this 1-manifold is just a single circle $C$.  Fix a disk $D$ in $W$ with $D\cap Y_1=\boundary D=C$ and a parallel copy $\Delta$ of $D$ that meets $Y_1$ at some parallel curve $\gamma$ to $C$. Since $C$ and $\gamma$ are parallel, there is a natural choice of framing $\phi$ of $N_Y(\gamma)$, so that $C$ is isotopic in $\nu(\gamma)\setminus\gamma$ to a pushoff of $\gamma$ according to $\phi$.

Suppose $\phi$ is $\Delta$-admissible. Then by surgering $Y_1$ along $\Delta$ according to a framing extending $\phi$, we obtain a 3-manifold $Y_1$ in which $C$ is unknotted. On the other hand, suppose that $\phi$ is {\emph{not}} $\Delta$-admissible. Let $\Delta'$ be the result of isotoping $\Delta$ in a neighborhood of $C$ to achieve a crossing change of $C$ and $\gamma$, so that $\Delta'$ intersects $Y_g$ at one point in its interior. Then tube $\Delta'$ to $G$ along an arc in $Y_g$ ending at $\mathring{\Delta'}\cap Y_g$ to obtain a disk $\widetilde{\Delta}$ whose interior is disjoint from $Y_g$. By Lemma \ref{lem:tubetoframed} (using the fact that $G$ is framed), $\phi$ is also not $\widetilde{\Delta}$-admissible. However, while $\widetilde{\gamma}:=\boundary \widetilde{\Delta}$ is parallel to $C$, the framing $\widetilde{\phi}$ that $C$ induces on $\widetilde{\gamma}$ is homotopic to $\pm1\cdot\phi$. By Lemma \ref{lem:framing}, $\widetilde{\phi}$ is $\widetilde{\Delta}$-admissible.  Then by surgering $Y_1$ along $\widetilde{\Delta}$ according to a framing extending $\widetilde{\phi}$, we obtain a 3-manifold $Y_1'$ in which $C$ is unknotted.

In Figure \ref{fig:unknotmovesketch}, we give a schematic of the intuition for obtaining $\widetilde{\Delta}$. In Figure \ref{fig:tildegamma}, we draw a more accurate picture of $\widetilde{\Delta}$.  By using this move we can, ``unknot the intersections" - see Figure \ref{fig:ambientdehn} (top row).  This is also pictured in Figure 5 in \cite{NS}.
\end{move}
	
Note that we used the framing on $G$ to unknot self-intersections of $Y$, as we needed to achieve surgery with a specific framing on a knot in $Y$. However, we do not need $G$ to be framed in order to remove intersections that are already unknotted, as then there is a natural disk with natural choice of framing.	

	\begin{figure}
	    \centering
	    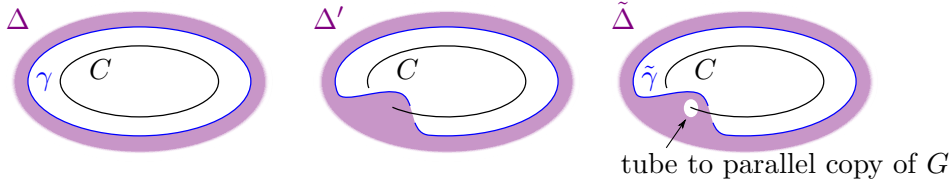
	    \caption{To obtain $\Delta'$ from $\Delta$, we isotope $\gamma=\boundary\Delta$ once through $C$. This introduces an intersection point between $\Delta'$ and $Y_g$. By tubing $\Delta'$ to a parallel copy of $G$ at this intersection point, we obtain a disk $\widetilde{\Delta}$ that does not intersect $Y_g$. Moreover, since $G$ is framed, if a framing $\phi$ of $N_Y(\gamma)$ is not $\delta$-admissible, then it is also not $\widetilde{\Delta}$ admissible (where $N_Y(\gamma)$ and $N_Y(\widetilde{\gamma})$ are identified by the pictured isotopy).}
	    \label{fig:unknotmovesketch}
	\end{figure}

	\begin{figure}
	    \centering
	    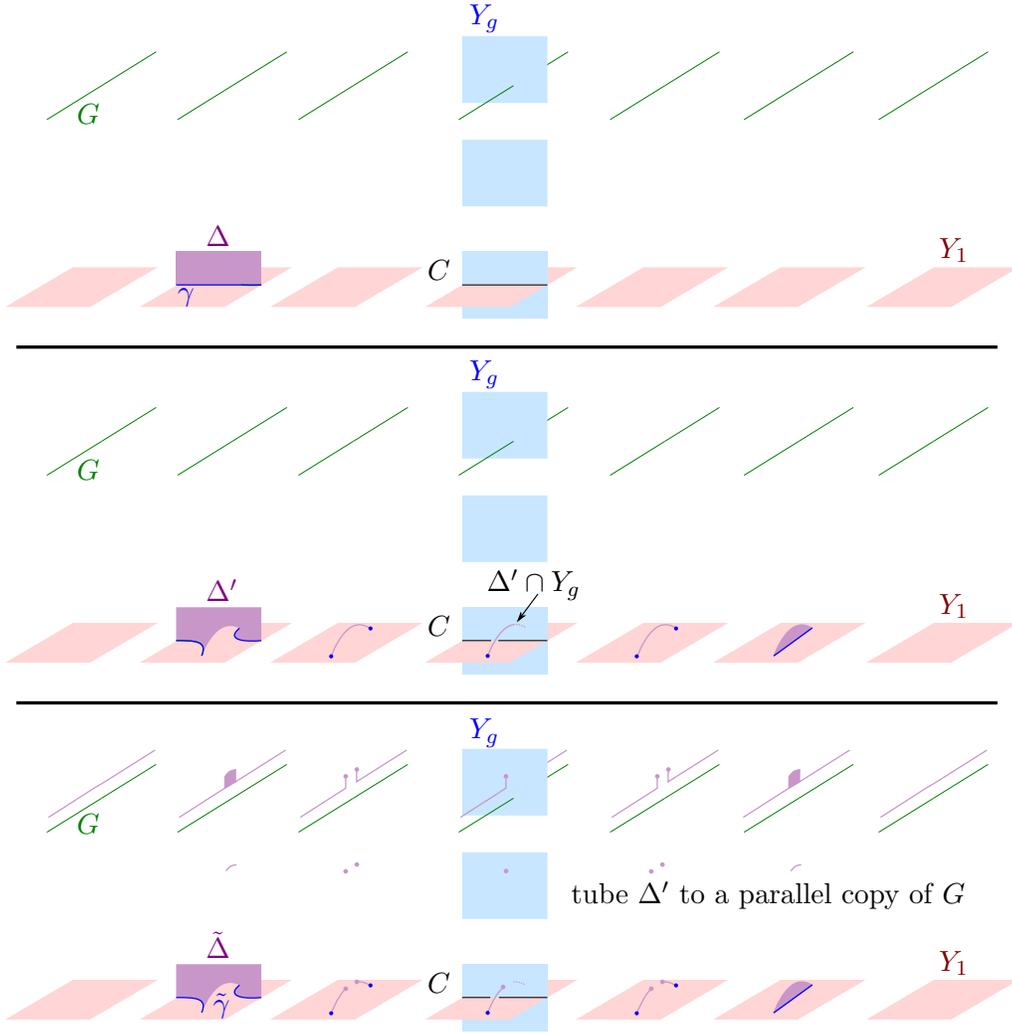
	    \caption{A more realistic depiction of the move described in Figure \ref{fig:unknotmovesketch}, which is part of Move \ref{move1}. Each picture is 5-dimensional, consisting of a two-dimensional array of 3-dimensional drawings. {\bf{Top:}} The disk $\Delta$. {\bf{Middle:}} The disk $\Delta'$. {\bf{Bottom:}} The disk $\widetilde{\Delta}$. }
	    \label{fig:tildegamma}
	\end{figure}

\begin{proof}[Proof of Theorem \ref{maintheorem} when $F_0$ and $F_1$ are 2-spheres]

Let $S_0:=F_0$ and $S_1:=F_1$, and assume that each $F_i$ is a 2-sphere. (We rename the surfaces to stick to the convention that $S_i$ is a 2-sphere while $F_i$ is a surface of arbitrary genus.)

	Let $\widetilde{X}$ denote the universal cover of $X$. Choose a preferred basepoint above the basepoint of $X$.  %First, we use the condition that the meridian of $S_0$ is nullhomotopic in order to ensure that our 3-manifolds have complements with trivial fundamental group, and therefore the our disks for ambient Dehn surgeries will exist.
	By Lemma \ref{lem:nobasegen}, we may take $S_0$ and $S_1$ to be based-homotopic with no loss of generality. (We implicitly used this fact to state Theorem \ref{maintheorem} and make sense of whether $\fq(S_0,S_1)$ vanishes or not.)

	Fix a based homotopy from $S_0$ to $S_1$ and look at its track $H : S^2 \times I \to X \times I$.  Let $H_1$ denote the lift of $H$ to $\widetilde{X} \times I$ based at the basepoint of $\widetilde{X}$. % that uses the chosen preferred lift of the basepoint of $X$ to $\widetilde{X}$. 
	Let $H_g$ denote the $g$-translate $gH_1$ of $H_1$, where $g\in\pi_1(X)$ and $\pi_1(X)$ acts on $\widetilde{X}$ as usual.% composed with the action of an element $g \in \pi_1(X)$ on $\widetilde{X} \times I$. 
	
	The meridian of $S_0$ in $X \times \{0\}$ lifts to a meridian of $H_1$. Seifert-van Kampen shows that $\pi_1( \widetilde{X} \times I - H_1)$ is normally generated by this meridian.  Because the meridian of $S_0$ is null homotopic in $X - S_0$, the lift of the meridian is null-homotopic in $\widetilde{X} \times I - H_1$ and therefore $\pi_1 ( \widetilde{X} \times I - H_1) = 1$.  

	Additionally, $\pi_1(\widetilde{X} \times I - \cup_{g} H_g) = 1$.  To see this, note that the group is normally generated by meridians of each of the $H_g$, and each of these meridians can be taken as a lift of the meridian of $S_0$.  Additionally, the nullhomotopy of the meridian of $S_0$ in $X - S_0$ %can be lifted to each of the meridians of the components and all of these nullhomotopies will
	lifts to $\widetilde{X}-\widetilde{S}_0\subset \widetilde{X}\times I-\cup_g Y_g$. %therefore lift to $\widetilde{X} \times I - \cup_{g} H_g$.  

	If $\fq(S_0,S_1)\neq 0$, then we know immediately that $S_0$ and $S_1$ are not concordant,  since, by \cite{ST}, a concordance could be used to compute $\fq(S_0,S_1)=0$. We now assume $\fq(S_0,S_1)=0$. We will use the interpretations of $\mu_3$ and $\fq$ that appear in \S\ref{sec:def4} (Definitions \ref{mudef4} and \ref{fqdef4}), namely
	the 5-dimensional definitions through covering spaces. %That is, we assume that for every $g\in T_X$, the action of $g$ fixes an even number of components (setwise) of $H_1\cap H_g$.

	%use the geometric interpretations of $fq$ and $\mu_3$ that where introduced in the previous section.   %Given $g \in T_X$, call the element of $\mathbb{F}_2 T_X$ whose $g$-coefficient is given by the number of circles (mod 2) in $H_1 \cap H_g$ fixed by $g$ must vanish when considered $fq(H)$ in $\mathbb{F}_2$. 
	
	Fix $\fq(H)\in\F_2 T_X$, where the coefficient of $g$ in $\fq(H)$ is the number of components of $H_1\cap H_g$ fixed (setwise) by the action of $g$. Since $\fq(S_0,S_1)=0$, 
	%We know, since $fq(S_0, S_1) = 0$ that
	there exists $f \in \pi_3(X)$ so that $\mu_3(f) = \fq(H)$.  We alter $H$ near $X\times 0$ by surgering $H$ along $f$ near the basepoint. That is, delete a small ball $B$ in $H\cap f(S^3)$ from $H$ and reglue $\overline{f(S^3)-B}$. This yields a new homotopy $H'$ between $S_0$ and $S_1$ with the property that $\fq(H')=\fq(H)+\mu_3(f)=2\fq(H)=0$. Set $H:=H'$. In words, for every $g\in \pi_1(X)$, the action of $g$ fixes an even number of components of $H_1\cap H_g$ setwise.%in putting in the map $f$ to obtain a new homotopy, which we will again call $H$ so that now the element $fq(H) \in \mathbb{F}_2 T_X$ just described is $0 \in \mathbb{F}_2 T_X$ (as opposed to just vanishing in the quotient).  In other words, with this new $H$, for each $g \in T_X$ there are an even number of circles in $H_1 \cap H_g$ that are fixed by $g$. 

    \begin{proposition}\label{prop:makezero}
    We can perform equivariant ambient 1-surgeries to $\cup_{g} H_g$ to obtain a collection of immersed cobordisms $\{Y_g\}$ with $\boundary Y_i=\boundary H_i$ with the property that for each $g\in T_X$, $g$ fixes no components of $Y_1\cap Y_g$ setwise.
    \end{proposition}

	%We will be looking at the action of $\pi_1(X)$ on the components of the intersections between the different translates $H_g$.  When we say that a circle is fixed by an element of $\pi_1(X)$, we mean that it is fixed as a component.  The subtle point (and the location of the error in \cite{NS}) is that some circles of intersection between the different $H_g$ can be fixed by some elements of $\pi_1(X)$.  This is how one is led to focus on 2-torsion elements of $\pi_1(X)$ and we will be able to modify things so that this does not occur, using the hypothesis that $fq(S_0,S_1) = 0$. 
	\begin{proof}
	Refer to Figure \ref{fig:tubing}.  By assumption, %from earlier,
	there are an even number of circles $H_1 \cap H_g$ that are fixed by $g$, for each $g \in T_X$.  Let $\gamma_1$ and $\gamma_2$ in $H_1\cap H_g$ denote two such circles and let $\alpha$ denote an arc in $H_1$ connecting $\gamma_1$ and $\gamma_2$. Take the interior of $\alpha$ to avoid self-intersections of $H_1$ and intersections of $H_1$ with $H_{g'}$ for any $g'\in\pi_1(X)$. (This is possible because $H_1$ is connected and 3-dimensional while these intersections are 1-dimensional.)
	
%	as well as from $H_1 \cap H_a$, for all $a \in \pi_1(X)$. 
Now for each $h \in \pi_1(X)$, we have an arc $h \cdot \alpha \subset H_h$ from $h \cdot \gamma_1$ to $h \cdot \gamma_2$, with $h \cdot \gamma_1, h \cdot \gamma_2 \in \pi_0(H_h \cap H_{hg})$.  Use all of these arcs $h \cdot \alpha$ to equivariantly perform ambient Dehn 1-surgeries to $H_{hg}$.  Let $Y_s$ denote the manifold obtained after performing these 1-surgeries to $H_s$.  %We examine the effect of this inside of $H_1 \cap H_g$, since in all other translation intersections, it is just a translation. 
The effect of this is to replace the circles of intersection $\gamma_1, \gamma_2$, with new circles of intersection $\gamma'_1, \gamma'_2\in \pi_0(Y_1 \cap Y_g)$ where now $g \cdot \gamma'_1 = \gamma'_2$. (Translations of $\gamma_1$ and $\gamma_2$ are similarly altered, but these circles are not contained in $Y_1$.) Thus, in $Y_1 \cap Y_g$ there are now two fewer circles fixed by $g$. % and therefore,
Since by assumption there are an even number of circles in $Y_1 \cap Y_g$ fixed by $g$, by repeating this process we can eliminate all of these circles. We take $\{Y_g\}$ to be the resulting cobordisms after performing all equivariant 1-surgeries.    

	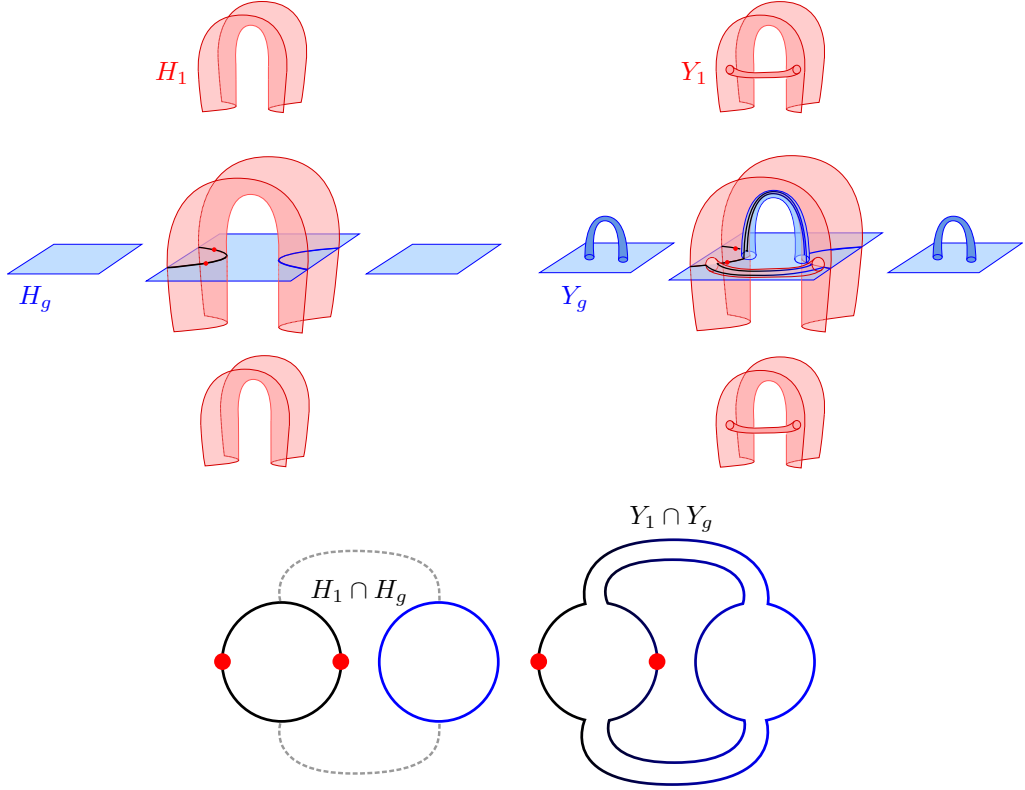
\begin{figure}
	  \labellist
\small\hair 2pt
 \pinlabel $\textcolor{blue}{H_g}$ at 50 300
 \pinlabel $\textcolor{red}{H_1}$ at 290 700
 \pinlabel $\textcolor{blue}{Y_g}$ at 1000 300
 \pinlabel $\textcolor{red}{Y_1}$ at 1210 700
 \pinlabel $H_1\cap H_g$ at 620 -220
 \pinlabel $Y_1\cap Y_g$ at 1170 -90
\endlabellist
	    \centering
	    \includegraphics[width=135mm]{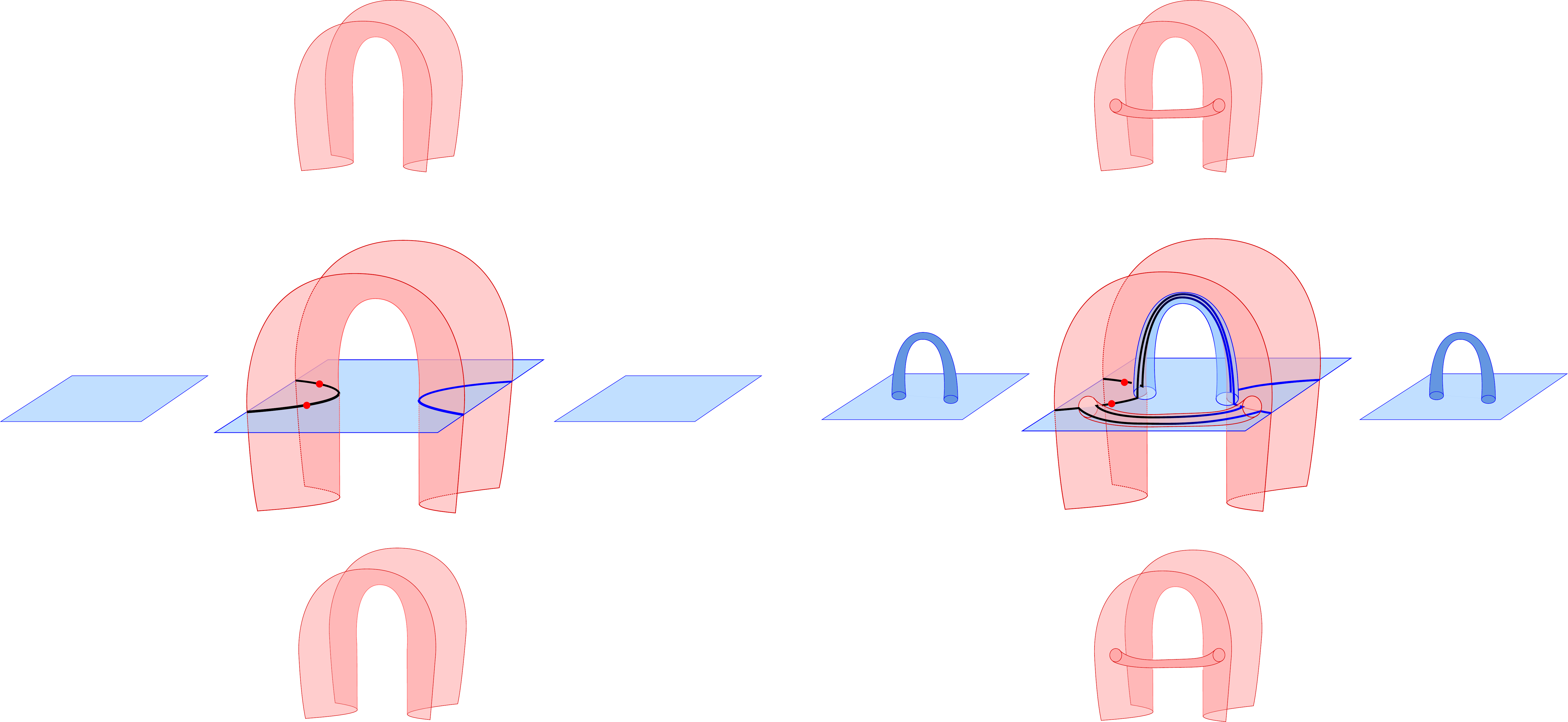}\\
	    \vspace{9mm}
	    \includegraphics[width=80mm]{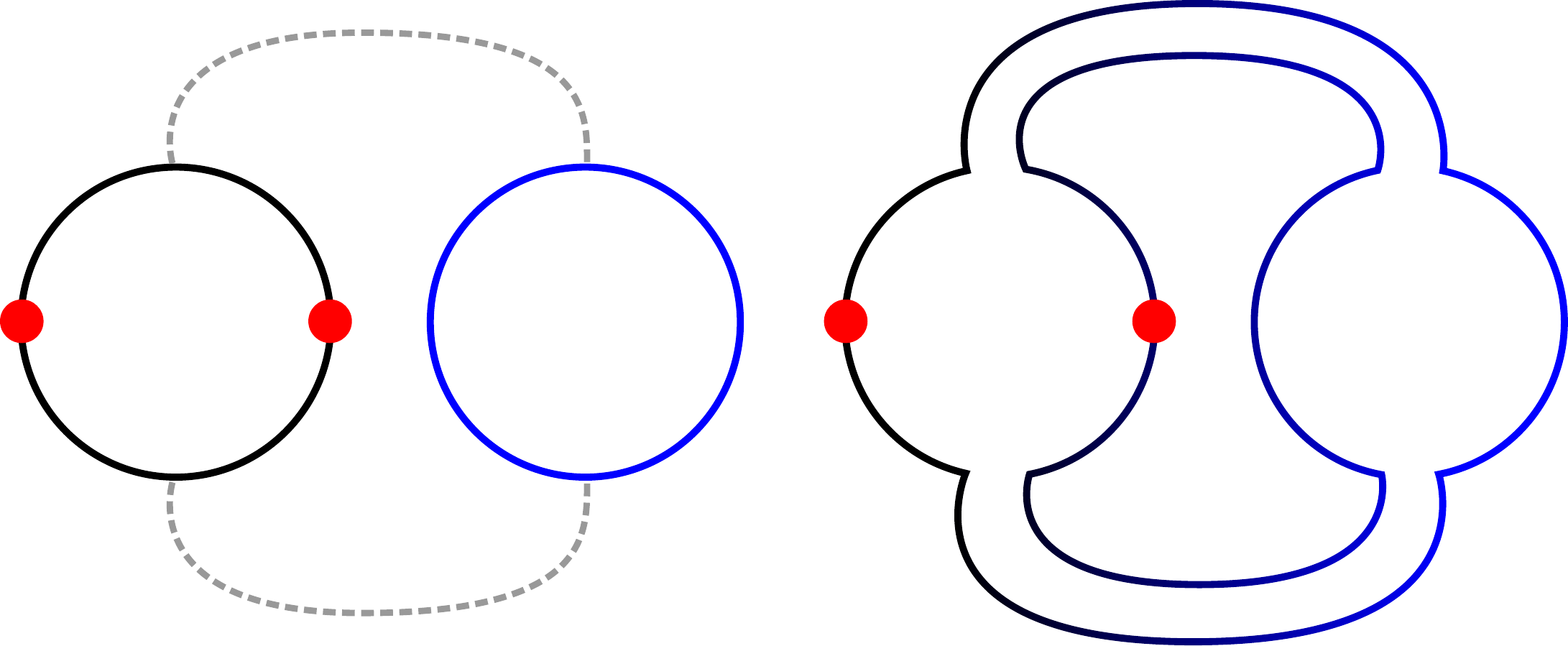}
	    \caption{{\bf{Top row, left:}} A schematic of $H_1$ and $H_g$, featuring arcs of two circles of $H_1\cap H_g$ fixed componentwise by the action of $g$. {\bf{Top row, right:}} We equivariantly perform ambient 1-surgeries to change these circles of intersection. The resulting manifolds are $Y_1$ and $Y_g$. {\bf{Bottom row:}} A combinatorial description of the top row move; we indicate two points permuted by the action of $g$. Here we see that after the surgery, the two circles are permuted by $g$.}
	    \label{fig:tubing}
	\end{figure}
%\maggie{Okay so $Y_g$ is $H_g$ after the 1-surgeries}

	%So now we are in the position where no circle of intersection between different translates is fixed by any $g \in \pi_1(X)$.  
	
	Note that since $\{Y_g\}$ is obtained from $\{H_g\}$ by ambient 1-surgeries, $\pi_1(\widetilde{X}\times I-\cup_{g} Y_g)=1$.
	\end{proof}

	\begin{proposition}\label{prop:makedisjoint}
    We can perform equivariant ambient Dehn 2-surgeries to $\{Y_g\}$ to obtain a disjoint collection of embedded cobordisms $\{M_g\}$ between lifts of $S_0$ to lifts of $S_1$ with the property that $\pi_1(X-\cup_{g}M_g)=1$.
    \end{proposition}
    
%	Our goal is to make the $H_g$ all disjoint.  In doing this, we will need to do several rounds of equivariant surgery on the $H_g$, and initially they will no longer each be $S^2 \times I$ topologically.  All of our moves will not change the fact that the complement of all of the translates has a trivial fundamental group, and therefore, once they are all disjoint

In Proposition \ref{prop:makedisjoint}, the manifolds $M_g$ are not necessarily products. %After proving Proposition \ref{prop:makedisjoint},
We will invoke the proof of Theorem \ref{simply-connected} in \cite{NS} to conclude that we can further surger the results equivariantly to obtain disjoint $\pi_1(X)$-translates of $S^2 \times I$, as desired.  This uses the moves outlined at the beginning of the section. % therefore remains to show how we can make all of the $H_g$ disjoint and preserve the fact that they are all $\pi_1(X)$-translates. 

\begin{proof}[Proof of Proposition \ref{prop:makedisjoint}]
We follow the strategy of \cite{NS}. 
%We can now essentially carry out the ideas in \cite{NS}.
%This involves performing several different kinds of ambient Dehn surgeries on the immersed cobordisms $\{Y_g\}$.

Note that the dual sphere $G$ of $S_0$ lifts to dual spheres of each lift of $S_0$ in $\widetilde{X}$. We can push these dual spheres into $\widetilde{X}\times I$ to obtain dual spheres for each $Y_g$, with the dual for $Y_g$ disjoint from $Y_h$ for $g\neq h$. This allows us to Perform \ref{move1}.

First, we deal with self-intersections of $Y_1$. We perform Move \ref{move1} so that all self-intersections of $Y_1$ are unknotted in $Y_1$, at the cost of changing the topology of $Y_1$. For each ambient Dehn 2-surgery on a disk $\Delta$ with boundary in $Y_1$,
%the circles of self-intersection, in order to make the $Y_g$ all embedded. 
%Given a circle of self-intersection $\gamma$ in $Y_1$, choose a disk $\Delta\subset \widetilde{X}\times I$ with boundary $\gamma$ so that the interior of $\Delta$ is disjoint from $\cup_{g}Y_g$. Push the boundary of $\Delta$ slightly off $\gamma$ into $Y_1$.  %all of the $Y_g$. and
%Pick an arbitrary framing on $\Delta$. (The disks in this argument must be framed so that we can perform surgery equivariantly. %; we want $\Delta$ and all of its translates to have the same framing (after translation).
%However, we do not actually care about the specific framing on $\Delta$.)
%Perturb the interior of $\Delta$ as necessary so that $\Delta$ is disjoint from all of its 
we also surger each translate of $\Delta$ under the action of $\pi_1(X)$. (Generically, these framed disks are all disjoint.) That is, we perform equivariant ambient Dehn 2-surgery on each $Y_g$ using $g\cdot \Delta$. %to obtain a new manifold, which we again call $Y_g$.  %Perform this surgery equivariantly %Then after performing surgery on $Y_1$ using $\Delta$, we obtain a new manifold, which we will again cal
Abusing notation, we call the resulting manifold $Y_g$. %^l $Y_1$, were

Now $\gamma$ is unknotted in $Y_1$ and unlinked with every other component of self-intersection of $Y_1$. (See Figure \ref{fig:ambientdehn}.) Repeat this procedure for every circle of self-intersection of $Y_1$ so that now the self-intersections of $Y_1$ form an unlink $L$ in $Y_1$.

%We do this, choosing all disjoint disks, for every circle of self-intersection of $Y_1$, and we do these ambient Dehn surgeries equivariantly.  Also, we choose our disks $\Delta$, so that they are all disjoint from all of their translates.	Now all of the circles of-self intersection in $Y_1$ form an unlink in $Y_1$, so

Choose a collection of framed disks in $Y_1$ with boundary $L$. Perform equivariant ambient Dehn 2-surgery along the translates of each disk in this collection (this is Move \ref{move2}), so that now $Y_g$ is embedded for every $g\in\pi_1(X)$.%inside $Y_1$ that bound the self-intersection circles, and perform surgeries on them as in [insert a figure] where we choose framings on the disks so that [....].  Doing this equivariantly to all of the $Y_g$, we now have each $Y_g$ embedded.  

We similarly perform equivariant ambient Dehn 2-surgeries on $\{Y_g\}$ to pairwise remove intersections as follows. %	Handling the circles of intersection between distinct translates is similar, going through the same sequence of two Dehn surgeries, the first to make things unlinks, and the second to remove intersections, but requires just a little more care. 
Given $g \in \pi_1(X)$ with $g^2 \neq 1$, then we again %, just as with the self intersections, we can
equivariantly perform Move \ref{move1} 
%find a set of disjoint framed disks bounding the curves
at each curve in $Y_1 \cap Y_g$ %who are disjoint from all of their translates. % and such that these disks are disjoint from all translates.  Then
%After doing equivariant surgery on $\{Y_h\}$ using these disks and their translates, we see 
so that $Y_1 \cap Y_g$ is now an unlink in $Y_1$. (Similarly, $Y_h\cap Y_{hg}$ is an unlink in $Y_h$ for all $h \in \pi_1(X)$.)  We can then perform Move \ref{move2} %ambient Dehn 2-surgery
on $Y_g$ once for each curve in $Y_1\cap Y_g$ (and equivariantly on all other translates) so that $Y_1$ becomes disjoint from $Y_g$. (Similarly, $Y_h$ is disjoint from $Y_{hg}$ for all $h \in \pi_1(X)$.) 

If $g \in T_X$, then the action of $g$ permutes the circles in $Y_1 \cap Y_g$ in pairs (here we use Proposition \ref{prop:makezero}). %can all be paired up,
Fix a circle $\gamma$ in $Y_1\cap Y_g$. Then $g^2\cdot\gamma=\gamma\neq g\cdot\gamma$. %\gamma, g \cdot \gamma$, with $\gamma \neq g \cdot \gamma$, by our previous argument involving $fq$. 
Pick one curve from each such pair $\{\gamma,g\cdot\gamma\}$, giving a set of curves $\{ \gamma_1,...,\gamma_k \}$ with $Y_1 \cap Y_g$ equal to the disjoint union of $\{ \gamma_1,...,\gamma_k \}$ and $\{ g \cdot \gamma_1,...,g \cdot \gamma_k \}$.  Equivariantly perform Move \ref{move1} on $Y_1$ %Find disjoint framed disks $\Delta$ bounding $ a
at $\gamma_1,\ldots,\gamma_k$, so that now $Y_h\cap Y_{hg}$ is an unlink in $Y_h$ for all $h \in \pi_1(X)$. % (disjoint from $\cup_h Y_h$ in $\mathring{\Delta}$) and perturb $\Delta$ to be disjoint from all $\pi_1(X)$-translates of $\Delta$. %take all translates of these disks (perturb the interior of the original disk so that these translates are all disjoint).  %Note that each curve in some $Y_a \cap Y_b$ is the boundary of exactly one such disk.  
%Perform equivariant ambient Dehn 2-surgery along the disks $\Delta$ and then again perform equivariant ambient Dehn 2-surgery using disks bounded by the resulting unlink in $Y_1$.
Then equivariantly perform Move \ref{move2} on $Y_g$ at $\gamma_1,\ldots,\gamma_k$ so that $Y_h\cap Y_{hg}=\emptyset$ for all $h\in\pi_1(X)$.  Repeat this for every $g\in T_X$.%, so that $Y_a\cap Y_b=\emptyset$ for $a\neq b\in\pi_1(X)$.%$Y_1$ is now disjoint from every $Y_h$, $h\in\pi_1(X)$.

Let $\{M_g\}$ denote these resulting mutually embedded manifolds. Since all surgeries were done equivariantly, $M_g=g\cdot M_1$. We conclude $M_g$ is embedded and $M_g\cap M_h=\emptyset$ for $g\neq h\in\pi_1(X)$.

%Then performing the surgeries on these disks and then again, as before, the surgeries on the resulting unlinks, results in the $Y_g$ all being disjoint.  At this point the proof of Theorem \ref{simply-connected} can be used to make equivariantly surger the $Y_g$ to be $S^2 \times I$, thus yielding the result.  

All of the ambient Dehn 2-surgeries in this proposition took place in the interior of $\widetilde{X}\times I$, and thus far from the lifts of $S_0$ and the immersed dual of $S_0$. We thus conclude that a meridian of $M_h$ is nullhomotopic in $\widetilde{X}\times I\setminus\cup_{g}M_g$, so $\pi_1(X-\cup_{g}M_g)=1$.

\end{proof}
	
	Now by applying Theorem \ref{simply-connected - heart}, we can equivariantly ambiently Dehn 2-surger the $\{ M_g \}$ to be disjoint concordances. (That is, Theorem \ref{simply-connected - heart} gives instructions on how to ambiently Dehn 2-surger $M_1$ to obtain an embedded $S^2\times I$. When performing each of these ambient Dehn 2-surgeries, we simultaneously perform all translates, which are generically along disjoint disks.) Projecting $M_1$ to $X\times I$ thus yields a concordance from $S_0$ to $S_1$, completing the proof of Theorem \ref{maintheorem} in the case that $F_0=S_0$ and $F_1=S_1$ are 2-spheres. %, and hence the image under $\widetilde{X} \times I \to X \times I$ yields the desired result.  
	
\end{proof}

\section{Higher genus surfaces}\label{sec:genus}

    In this section, we will prove Theorem \ref{maintheorem} when the surfaces are of arbitrary genus and also define $\fq(F_0,F_1)$ when $F_0,F_1$ are $\pi_1$-negligible based-homotopic positive-genus surfaces. % whose fundamental groups include trivially into $\pi_1(X)$.
	%With suitable hypotheses, we find that the main arguments of this paper also work with 2-spheres replaced by closed, orientable, positive-genus surfaces.  Our starting point is again Theorem \ref{simply-connected} from \cite{NS}, which we plan to extend to nonsimply-connected manifolds by adding a hypothesis that a version of the Freedman-Quinn invariant vanishes.  

%\begin{remark}
%We remark briefly on a possible point of confusion in the proof of Theorem \ref{simply-connected} in \cite{NS} with respect to higher genus surfaces.  It is observed in \cite{NS}, that any two spin 3-manfiolds with boundary differ by spin surgery if and only if they have the same spin structures on their boundaries.  In the proof of Theorem \ref{simply-connnected} in \cite{NS}, we fix a 3-manifold $Y^3 \subset X \times I$ and in the case where $<w_2(X \times I - Y), h> = 0$ for all $h \in H_2(X \times I - Y; \mathbb{Z})$, $Y$ is given a spin structure such that all spin surgeries on $Y$ can be carried out ambiently.  It is then concluded that $Y$ can be surgered to be a product, however it is worth justifying why the spin structures on the two boundary components of $Y$ are the same.  \maggie{MM: not to self: seems like really we just find a product with the right boundary and glue, got confused}
	
%[Insert argument]
%\end{remark}

	Let $F$ be a closed orientable surface and let $F_0,F_1 : F \hookrightarrow X$ be two based-homotopic embeddings of $F$. Further assume that $F_0$ and $F_1$ are $\pi_1$-negligible. Then we can define $\mu_3(H)$ for any generic map $H : (F,\ast) \times I \to (X,\ast) \times \mathbb{R} \times I$ with $H(F\times0)=F_0\times\{0\}\times\{0\}$ and $H(F\times 1)=-F_1\times\{0\}\times\{1\}$ as in Definition \ref{mudef1} by again assigning an element of $\pi_1(X)$ to each self-intersection of $H$. The condition that $\pi_1(H)$ maps trivially to $\pi_1(X)$ ensures these group elements will be well-defined and be contained in $T_X\cup\{1\}$; we sum these elements to obtain % that restricts to $F_0$ included in $X \times \{0\} \times \{0\}$ and $F_1$ included in $X \times \{0\} \times \{0\}$ on the boundary of $F \times I$.  Note that we must have the condition on $\pi_1$ in order for this to make sense.
	 $\mu_3(H) \in \mathbb{F}_2T_X$. % \subset \mathbb{Z}\pi_1(X) / \langle g + g^{-1}, 1\rangle$.
	 
	 To define $\fq(F_0,F_1)$ as we did when $F_0, F_1$ were 2-spheres, we need only to see that the choice of $H$ does not affect $\mu_3(H)$ (up to $\mu_3(\pi_3(X)))$. We must reprove a lemma analogous to Lemma 4.4 in \cite{ST}.  

\begin{lemma}\label{lem:compress}
	Let $H : F \times I \to X \times \mathbb{R} \times I$ be a generic map where the two ends are contained in $X \times \{0\} \times \{0\}$ and $X \times \{0\} \times \{1\}$ (respectively) and are identical (as submanifolds of $X\times\R$, with opposite induced orientations) and are $\pi_1$-negligible.  Then $\mu_3(H) \in \mu_3(\pi_3(X))$.  
\end{lemma}

\begin{proof}
    By \cite[Lemma 4.4]{ST}, it is sufficient to find a map 
 %   If there exists a map 
%	If we can find a map
$J : S^2 \times I \to X \times \mathbb{R} \times I$ with the above property of having identical ends and with $\mu_3(H) = \mu_3(J)$.  %We can then apply Lemma 4.4 from \cite{ST}. 

Suppose $F_i$ is genus $g$. Fix an essential loop $\alpha$ in $F_0$ away from the basepoint of $F_0$.  Homotope $H$ rel boundary so that for all $t$, %Then we can change $H$ by a homotopy, fixing the ends, so that
$H(\alpha \times \{t\})$ agrees with $H(\alpha \times \{0\})$ after projecting to $X\times\R$. (Here we are using the condition that $H(\alpha\times\{t\})$ is nullhomotopic in $X$.) % except the last parameter of $X \times \mathbb{R} \times I$ is $t$ instead of $0$.
	Do further homotopy as necessary so that that all of the self-intersections of the homotopy occur away from $H(\alpha \times I)$.
	
	Since $X \times \mathbb{R}\times\{0\}$ is 5-dimensional, $\alpha\times\{0\}$ bounds a framed embedded disk $D_0$ in $X \times \mathbb{R} \times \{0\}$ that is disjoint from the image of $H$ except at its boundary, where the disk is normal to $H$.  Let $D_t$ be a copy of $D_0$ translated to $X\times\mathbb{R}\times\{1\}$. Perturb $D_0$ as necessary so each $D_t$ ($0<t<1$) that intersects $H$ in its interior does so transversally. % in its interior. % -- dimensionality tells us we may not be able to make $D_t$ disjoint from $H$ for all $t$.
	Now compress $H$ along the 4-dimensional 2-handle $\cup_{t\in[0,1]}D_t\times I$, where $D_t\times I\subset X\times\mathbb{R}\times\{t\}$. In words, we simultaneously compress each cross-section of $H$ along a copy of $D_t$. This yields a map 
	%We then attach $B$ to $H$ as a 3-dimensional 2-handle to obtain a map %We can then compress $H$ as in figure [Make me] \maggie{I'll make this soon} at all levels to obtain a map
	$H' : \Sigma_{g-1} \times I \to X \times \mathbb{R} \times I$.  
	
	We claim $\mu_3(H')=\mu_3(H)$. (In this argument, we use Definition \ref{mudef1}.) We note that $H'$ has exactly the same self-intersections as $H$, as well as two new self-intersections for each intersection of $H$ with the interior of some $D_t$. Each such pair of self-intersections contributes the same group element to $\mu_3(H')$; see Figure \ref{fig:mu3same}. Since the codomain $\F_2 T_X$ of $\mu_3$ is characteristic 2, %$\mu_3\in\F_2 T_X$,
	we conclude that $\mu_3(H')=\mu_3(H)$.
	
	Proceeding inductively on $g$, we eventually find a map $J : S^2 \times I \to X \times \mathbb{R} \times I$ as desired.

	%We show that $\mu_3(H) = \mu_3(J)$, therefore completing the argument.  To see this, note that .....[We need that figure].
\end{proof}

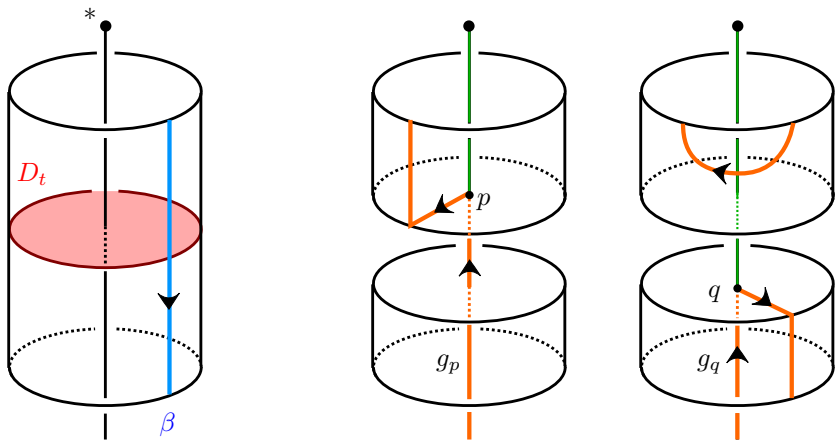
\begin{figure}
	  \labellist
\small\hair 2pt
 \pinlabel \textcolor{red}{$D_t$} at 50 530
  \pinlabel \textcolor{blue}{$\beta$} at 320 30
 \pinlabel $\ast$ at 165 850
 \pinlabel $p$ at 950 475
 \pinlabel $g_p$ at 880 150
 \pinlabel $q$ at 1410 290
 \pinlabel $g_q$ at 1400 150
\endlabellist
\includegraphics[width=110mm]{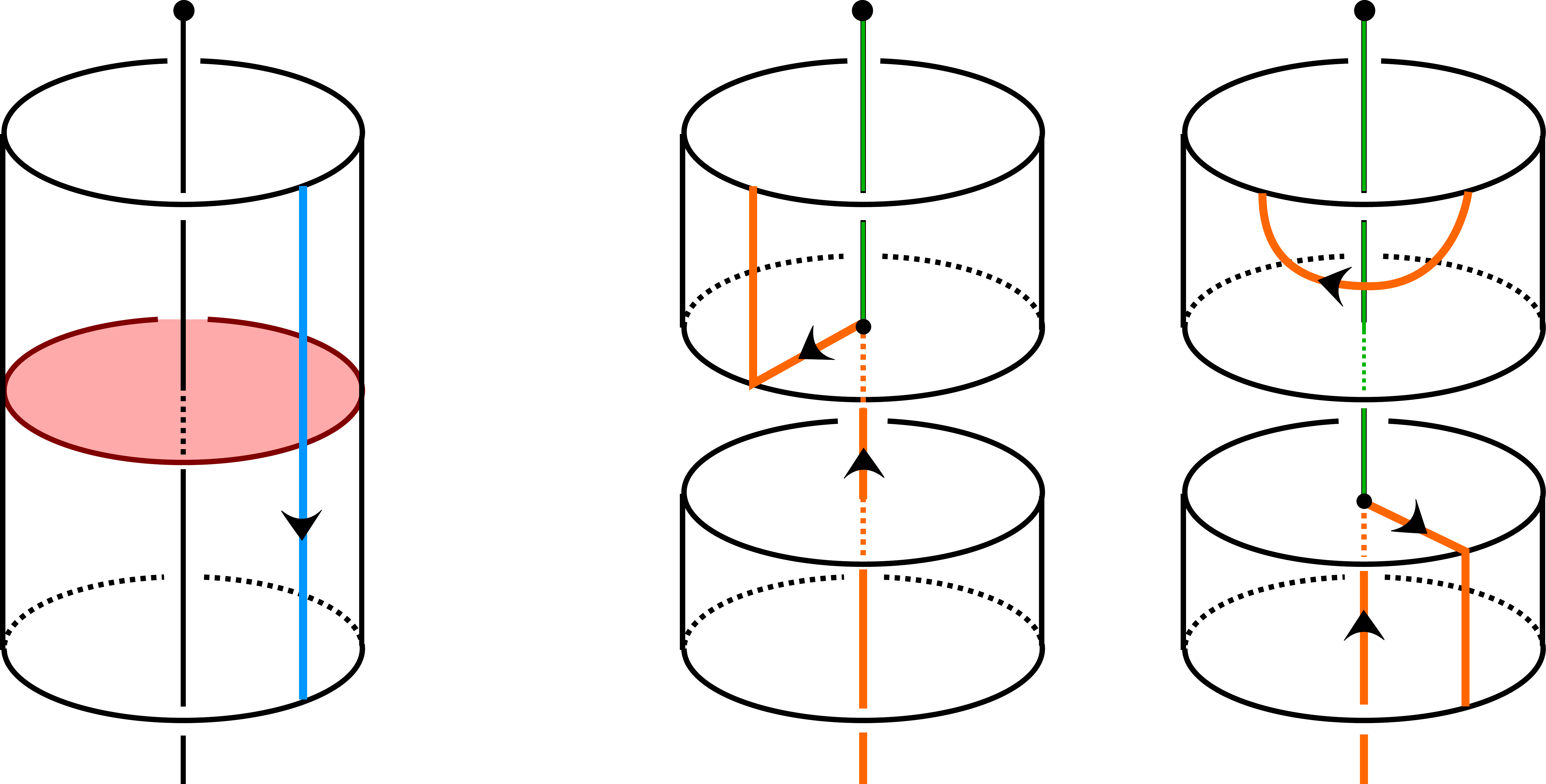}
\caption{In Lemma \ref{lem:compress}, we compress cross-sections of a homotopy $H$ of genus-$g$ surfaces to obtain a homotopy $H'$ of genus-$(g-1)$ surfaces. {\bf{Left:}} $D_t$ is a compression disk intersecting $H$ transversally. The boundary of $D_t$ is an essential curve $\alpha$ in $F_t$, so there exists a based curve $\beta$ in $F_t$ intersecting $\alpha$ exactly once. {\bf{Right:}} We draw two copies of $H'$ near $D_t$. Each point of intersection $D_t\cap H$ yields two points of self-intersection of $H'$. The group elements $g_p$, $g_q$ contributed by these points to $\mu_3(H')$ are related by $q_q=[\beta]g_p$. By assumption, $[\beta]=1$, so $g_p=g_q$.}\label{fig:mu3same}
\end{figure}

Thus, we can define $fq(F_0,F_1)$ for based-homotopic positive-genus surfaces. We now extend Lemma \ref{lem:nobasegen} to apply to positive-genus surfaces.

\begin{proposition}\label{lem:genusnobase}
Let $F_0$ and $F_1$ be $\pi_1$-negligible smoothly embedded genus-$g$ in a smooth 4-manifold $X^4$. Let $G$ be a 2-sphere immersed in $X$. Assume that $F_0$ intersects $G$ transversally once at a point $z$ and that $F_0$ and $F_1$ are homotopic. Then after a isotopy of $F_1$, $F_0$ and $F_1$ are based-homotopic with basepoint $z$.
\end{proposition}

The proof is essentially the same as for Lemma \ref{lem:nobasegen}, but we restate the proof here for completeness.

\begin{proof}
Let $H:\Sigma_g\times I\to X\times I$ be a free homotopy of $F_0$ to $F_1$. 
%Note $F_i$ represents an element $[F_i]\in\pi_2(X,z)$. Fix $z_0\in S^2$ so that $H(S^2\times 0)=z\times 0$. Then $H(z_0\times I)$ represents an element $g$ of $\pi_1(X,z)$ with $g[F_0]=[F_1]$.
Take $H$ to be transverse to $G$ and consider $L=H^{-1}(G\times I)$, a properly embedded 1-dimensional submanifold of $\Sigma_g\times I$. Note the boundary of $L$ consists of $z_0\times0$ and an odd number of points of the form $y_i\times 1$. Let $L_0$ be the component of $L$ containing $z_0\times 0$ and assume the other endpoint of $L_0$ is $y_0\times 1$. Compose $H$ with an isotopy of $F_1$ taking $H(y_0)$ to $z$ by isotopy in $G$ and set the resulting composed homotopy to be $H$. Now $\partial L_0=z_0\times\{0,1\}$. Since $\pi_1(H(\Sigma_g\times I))$ maps trivially into $\pi_1(X\times I)$, $H(L_0)$ is homotopic rel boundary to $z\times I$. Moreover, %letting $p:X\times I\to X$ denote projection,
we have $H(L_0)\subset G\times I$ away from self-intersections of $G\times I$, so $H(L_0)$ is a contractible based loop. Contracting this loop yields a based homotopy from $\Sigma_0$ to $\Sigma_1$.
%is simply-connected, this implies $H(L_0)$ represents $g$. Since $H(L_0)\subset G\cong S^2$, $g\in\pi_1(X,z)$ is the identity and we conclude $[F_0]=[F_1]\in\pi_2(X,z)$. See Figure \ref{fig:basedhtpy} for an illustration.
\end{proof}

\begin{proof}[Proof of Theorem \ref{maintheorem}]
%Theorem \ref{maintheorem} now follows immediately from the proof in Section \ref{sec:2spheres}.\qed

%The proof of Theorem \ref{maintheorem} applies directly to yield:%Similarly, all of the results in the first section hold for higher genus surfaces and therefore the proof of Theorem \ref{maintheorem} goes through unchanged to yield:

%\begin{theorem}\label{thm:genus}

%Let $F_0$ and $F_1$ be two homotopic embedded genus-g surfaces in a 4-manifold $X^4$. Assume $\pi_1(F_i)$ maps trivially into $\pi_1(X)$ and that a meridian of $F_0$ is nullhomotopic in $X-F_0$. 

%Then $F_0$ and $F_1$ are concordant if and only if $\fq(F_0,F_1)=0$.% so that $fq(F_0,F_1) = 0$ and that the meridian of $F_0$ is null-homotopic in $X - F_0$.  Then $F_0 \subset X \times \{ 0 \}$ and $F_1 \subset X \times \{ 1 \}$ are concordant in $X \times I$.  

%	[Oh god I forgot to include the homology with $\pi_1$ coefficients condition - this must show up in being able to lift the homotopy right at the very start]
%\end{theorem}

As in \S\ref{sec:2spheres}, we implicitly define $\fq(F_0,F_1)$ to be $\fq(F_0,F'_1)$, where $F'_1$ is isotopic to $F_1$ and based-homotopic to $F_0$ (via Lemma \ref{lem:genusnobase}). Since isotoping $F_1$ does not affect whether $F_0$ and $F_1$ are concordant, the choice of $F'_1$ does not matter. The proof of Theorem \ref{maintheorem} then follows exactly as in the case when $F_0,F_1$ are genus-zero as in \S\ref{sec:2spheres}.% in Theorem \ref{thm:genus}.
\end{proof}

%The analogue of Corollary \ref{cor:top} follows immediately from Theorem \ref{thm:genus} and the proof of Corollary \ref{cor:top}.
%\begin{corollary}
%Let $F_0$ and $F_1$ be genus-$g$ surfaces embedded in a 4-manifold $X^4$ with good fundamental group in the sense of Freedman and Quinn \cite{fq}. Assume that $F_0$ and $F_1$ are homotopic, $\pi_1(F_i)$ maps trivially into $\pi_1(X)$ for each $i$, and the meridian of $F_i$ is nullhomotopic in $X-F_i$ for each $i$. Then if $\fq(F_0,F_1)=0$, there exists a homeomorphism of pairs $(X,F_0)\cong (X,F_1)$.
%\end{corollary}

All of the above results also hold in the topological category - as mentioned also in \cite{NS}.  One way to see this is by smoothing any topological manifolds involved away from a point and then applying the above results.  
\section{Examples}\label{sec:examples}

	In this section we give two examples from the literature and an example of our own that illustrate the necessity of the hypotheses in Theorem \ref{maintheorem}.  Example \ref{hannahexample} is due to Hannah Schwartz and it demonstrates the necessity of the condition that $fq = 0$.  We use this example to illustrate our view of $fq$ in terms of lifts of covers from \cite{Hannah}. 
	
	We then construct Example \ref{stongexample}: two homotopic spheres with vanishing Freedman-Quinn invariant and a common embedded unframed dual sphere that are not concordant.  This illustrates the necessity of our condition that the dual sphere be framed as well as the necessity of the condition on the framing of the sphere in Gabai's lightbulb trick.  
	
	 Finally, Example \ref{stongexample} highlights the relationship between Theorem \ref{maintheorem} and a theorem in \cite{ST}, namely, we are assuming that one of the two spheres has a null homotopic meridian and concluding that the surfaces are concordant, whereas in \cite{ST}, the authors are assuming that both spheres have a common geometric dual and concluding that the surfaces are isotopic.  The assumptions we make (and the conclusion we derive) are weaker than those in the theorem in \cite{ST} (the 4-dimensional light bulb theorem).  The example we give, which appears in \cite{Sato}, demonstrates that this weaker conclusion is strictly necessary %our assumptions are strictly weaker
	and highlights the difference between concordance and isotopy.

\begin{example}\label{hannahexample}
    In \cite{Hannah}, Schwartz constructs two 2-spheres $S_0,S_1$ in a 4-manifold $X^4$ with $\pi_1(X)=\langle g\mid g^2=1\rangle$ containing a 2-sphere $G$ with trivial normal bundle intersecting $S_0$ and $S_1$ each transversally once. This example is interesting because Schwartz shows that $S_0$ and $S_1$ are not isotopic or even concordant. She then computes $\fq(S_0,S_1)=g$, illustrating that the Freedman-Quinn invariant can be used in this example to distinguish $S_0$ and $S_1$. In Figure \ref{fig:hannahex}, we draw a schematic of the universal cover of $X$ and the lifts of $S_0$ and $S_1$ in order to re-compute $\fq(S_0,S_1)$ using the definition of $\fq$ through covering spaces.

 \begin{figure}
 \includegraphics[width=135mm]{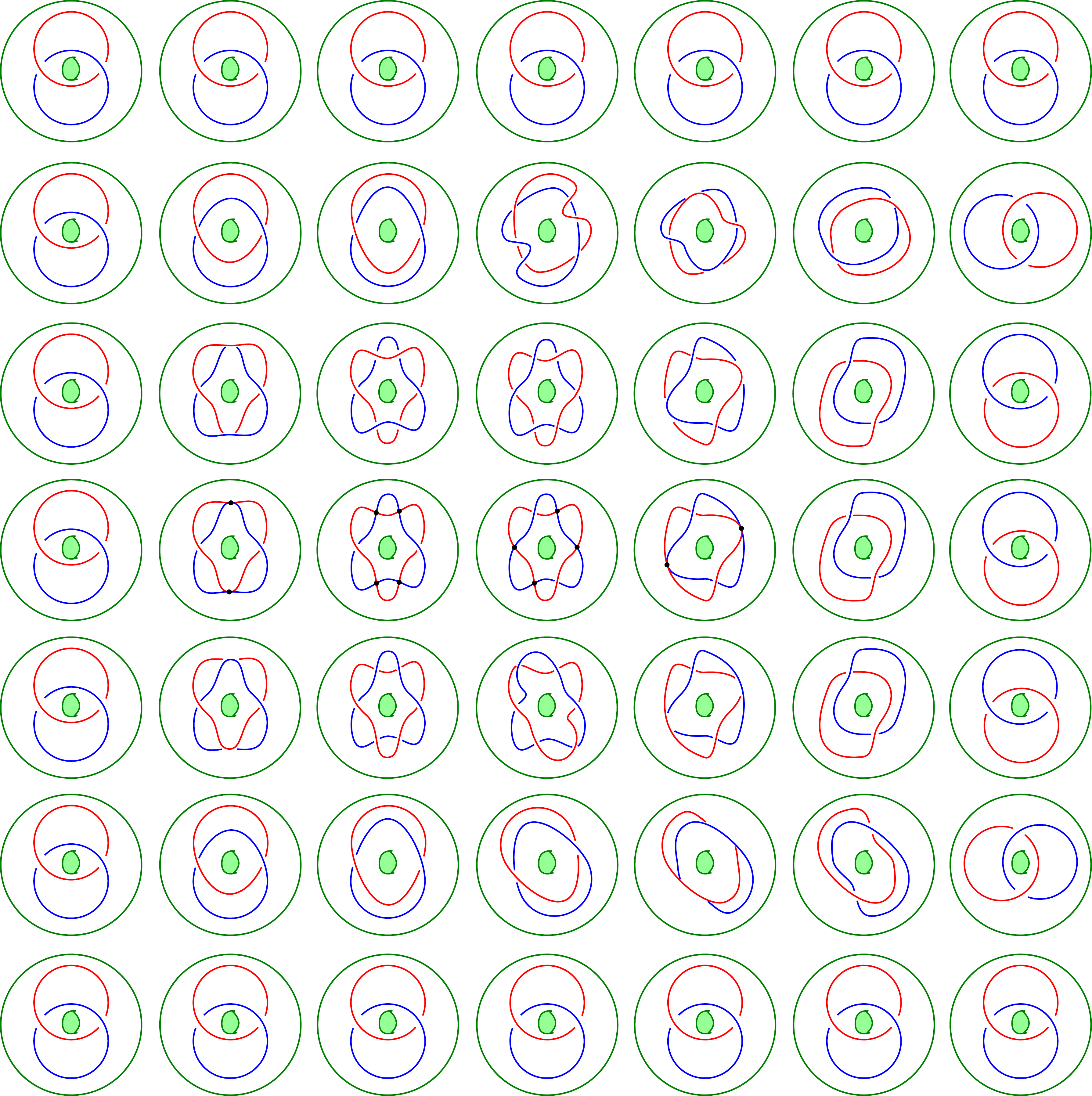}
	\caption{A schematic of $\widetilde{X}\times I$, where $X$ is the manifold constructed by Schwartz \cite{Hannah}. Each column represents a copy of $\widetilde{X}\times t$, so that the $I$ axis is horizontal. The action of $g$ is to rotate each cell through an angle of $\pi$. {\bf{Left column:}} A movie of the lift $\widetilde{S}_0$ of $S_0$ to $\widetilde{X}$. {\bf{Right column:}} A movie of the lift $\widetilde{S}_1$ of $S_1$ to $\widetilde{X}$. {\bf{Whole diagram:}} We find equivariant concordances between the components of $\widetilde{S}_0$ to the components of $\widetilde{S}_1$; one concordance is colored red and the other is colored blue. Even though these concordances are individually embedded, they intersect in a single circle $C$ which is preserved setwise by the action of $g$. We conclude that $\fq(S_0,S_1)=g\neq 0\in \F_2 T_x/\mu_3(\pi_3(X))=\F_2\{g\}$. Thus, as shown by Schwartz \cite{Hannah}, the spheres $S_0$ and $S_1$ are not concordant.}\label{fig:hannahex}
\end{figure}

	%The following example appears in \cite{Hannah} to demonstrate the necessity of considering 2-torsion (\ie $fq$) in order to conclude concordance.  We include this example to illustrate the method of computing $fq$ as explained in the first section.  

\end{example}

\begin{example}\label{stongexample}
We construct two 2-spheres $S_1,S_2$ that are based-homotopic, have embedded dual spheres and $\fq(S_1,S_2)=0$, and yet $S_1$ and $S_2$ are not concordant. In this example, the dual spheres to $S_1$ and $S_2$ are necessarily not framed; this demonstrates the necessity of the word, ``framed," in Theorem \ref{maintheorem}. 

Our obstruction to concordance is Stong's Kervaire-Milnor invariant \cite{stong}, denoted $\km$, which can be rephrased as secondary obstruction that exists for homotopic $s$-characteristic spheres in 4-manifolds that have vanishing Freedman-Quinn invariant.  In a forthcoming work we will give a detailed exposition of this invariant in the context of concordance of surfaces, so here we just state its simplest properties that will be of use to us.  The invariant $\km$ is valued in $H_1(X; \mathbb{Z}/2)$ and is computed from an immersion $H : S^2 \times I \to X \times I$ with ends two s-characteristic spheres $S_0 \subset X \times\{0\}$ and $S_1 \subset X \times \{1\}$.  Assuming that the submanifold of self-intersections of $H$ is just two circles that double cover their image under $H$, and such that these two circles form a Hopf link in $S^2 \times I$, then $\km$ is equal to the element of $H_1(X; \mathbb{Z}/2)$ that we obtain when applying $H$ to an arc in $S^2 \times I$ whose endpoints are points on the two different components  of the Hopf link that get identified by $H$ and consider the resulting circle as an element of $H_1(X; \mathbb{Z}/2)$.

Let $S_1$ be the unknotted sphere in $S^1\times S^3$. Let $S'_1$ be the immersed sphere obtained from $S_1$ by performing a finger move as in Figure \ref{fig:stonghopf}. In Figure \ref{fig:stonghopf}, we illustrate the framed Whitney disk $W$ along which a Whitney move would undo the finger move (thus yielding $S$ again). The boundary of $W$ consists of two arcs $\alpha,\beta$ in $S'$ between the self-intersections of $S'$. We also draw another pair of arcs $\alpha',\beta'$ between the self-intersections of $S'$, with the ends of $\alpha',\beta'$ contained in the same sheets as $\alpha,\beta$, respectively. However, $\alpha'$ intersects $\beta$ transversally once and $\beta'$ intersects $\alpha$ transversally once. If there exists a clean, framed Whitney disk $W'$ with boundary $\alpha'\cup \beta'$, then the homotopy from $S_1$ to a sphere $S_2$ that involves the pictured finger move followed by a Whitney move along $W'$ would have self-intersection a Hopf link. We could then conclude that $\km(S_1,S_2)$ is the generator of  $H_1(S^1\times S^3;\mathbb{Z}/2)\cong\mathbb{Z}/2$.

However, rather than a clean Whitney disk with boundary $\alpha'\cup\beta'$, in Figures \ref{fig:stongdisk1} and \ref{fig:stongdisk2} we obtain a framed Whitney disk $W'$ that intersects $S'$ twice. Now connect sum two copies of $\CP^2$ to $S^1\times S^3$ and tube $S_1$ (and hence $S'$) to a copy of $\CP^1$ in each $\CP^2$. This yields two disjoint unframed dual spheres to $S'$. By tubing $W'$ at each intersection with $S'$ to a different dual, we can obtain an embedded Whitney disk $\widetilde{W}$ in $(S^1\times S^3)\# 2\CP^2$ -- since we tube $W'$ to an even number of spheres with odd Euler number, $\widetilde{W}$ is still framed.

Now let $S_2$ be the 2-sphere in $(S^1\times S^3)\# 2\CP^2$ obtained from $S_1$ by a finger move about the $S^1$ factor in $S^1\times S^3$ (as in Figure \ref{fig:stonghopf}) followed by a Whitney move along $\widetilde{W}$. Note that $S_1$ and $S_2$ are based-homotopic, $S_1$ has an unframed dual, and  $\fq(S_1,S_2)=0$ automatically since $\pi_1((S^1\times S^3)\# 2\CP^2)$ has no 2-torsion. Moreover, $[S_i]$ is characteristic, so $\km(S_1,S_2)$ is defined. Then as desired, $\km(S_1,S_2)=1$ in $H_1((S^1\times S^3)\# 2\CP^2;\mathbb{Z}/2)\cong\mathbb{Z}/2$, so $S_1$ and $S_2$ are not concordant.  Note that this also gives an example where Gabai's lightbulb trick fails when the dual is not framed.  
\end{example}

\begin{figure}
    \centering
    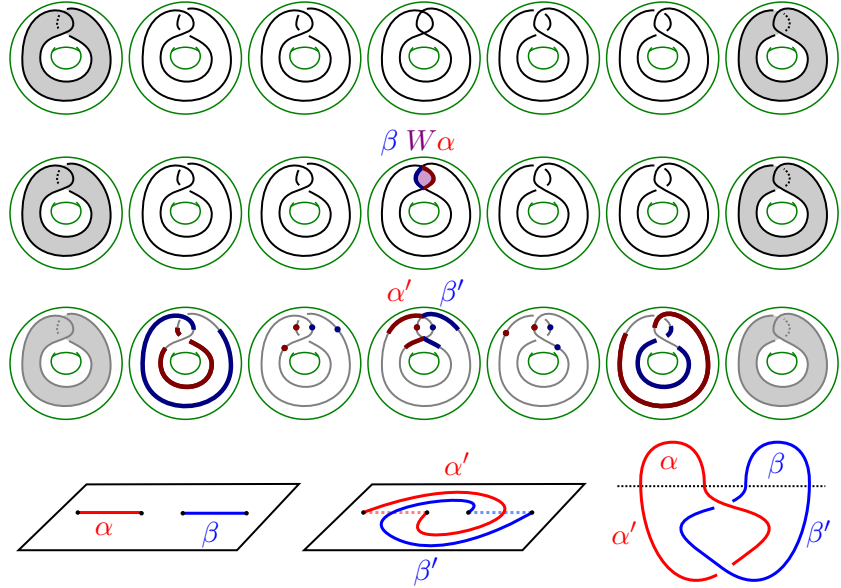
    \caption{{\bf{Top row:}} The 2-sphere $S'_1\subset S^1\times S^3$, obtained from performing a finger move to the unknotted sphere. Each picture is 3-dimensional; the horizontal axis is taken to be the fourth dimension. {\bf{Second row:}} The Whitney disk $W$. Performing a Whitney move to $S'_1$ along $W$ yields the unknotted sphere. {\bf{Third row:}} A pair of arcs $\alpha',\beta'$ in $S'_1$. {\bf{Bottom row:}} The arcs $\alpha,\beta,\alpha',\beta'$ in the domain of the immersion of $S'$. Performing a Whitney move using a disk with boundary $\alpha'\cup\beta'$ on $S'$ yields a homotopy whose self-intersection is a Hopf link.}
    \label{fig:stonghopf}
\end{figure}

\begin{figure}
    \centering
    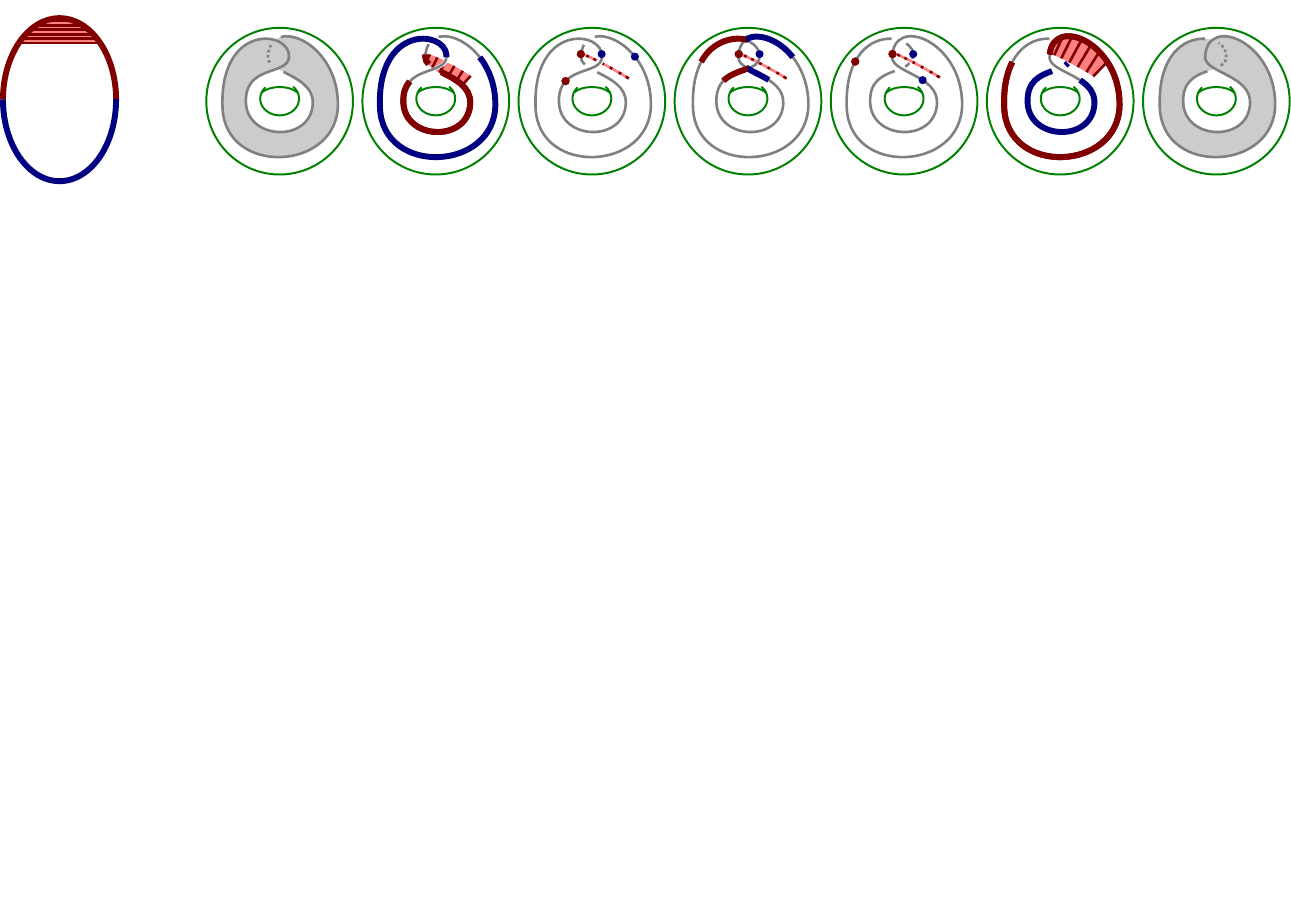
    \caption{In each row, we draw a portion of a Whitney disk $W'$ for $S'$, as in Figure \ref{fig:stonghopf}. On the left, we illustrate the corresponding portion of $W'$. We can explicitly see the two points in which $W'$ intersects $S'$.  We illustrate the disk in pieces like this to give some intuition for its construction - a complete view is given in the following figure.  }
    \label{fig:stongdisk1}
\end{figure}

\begin{figure}
    \centering
    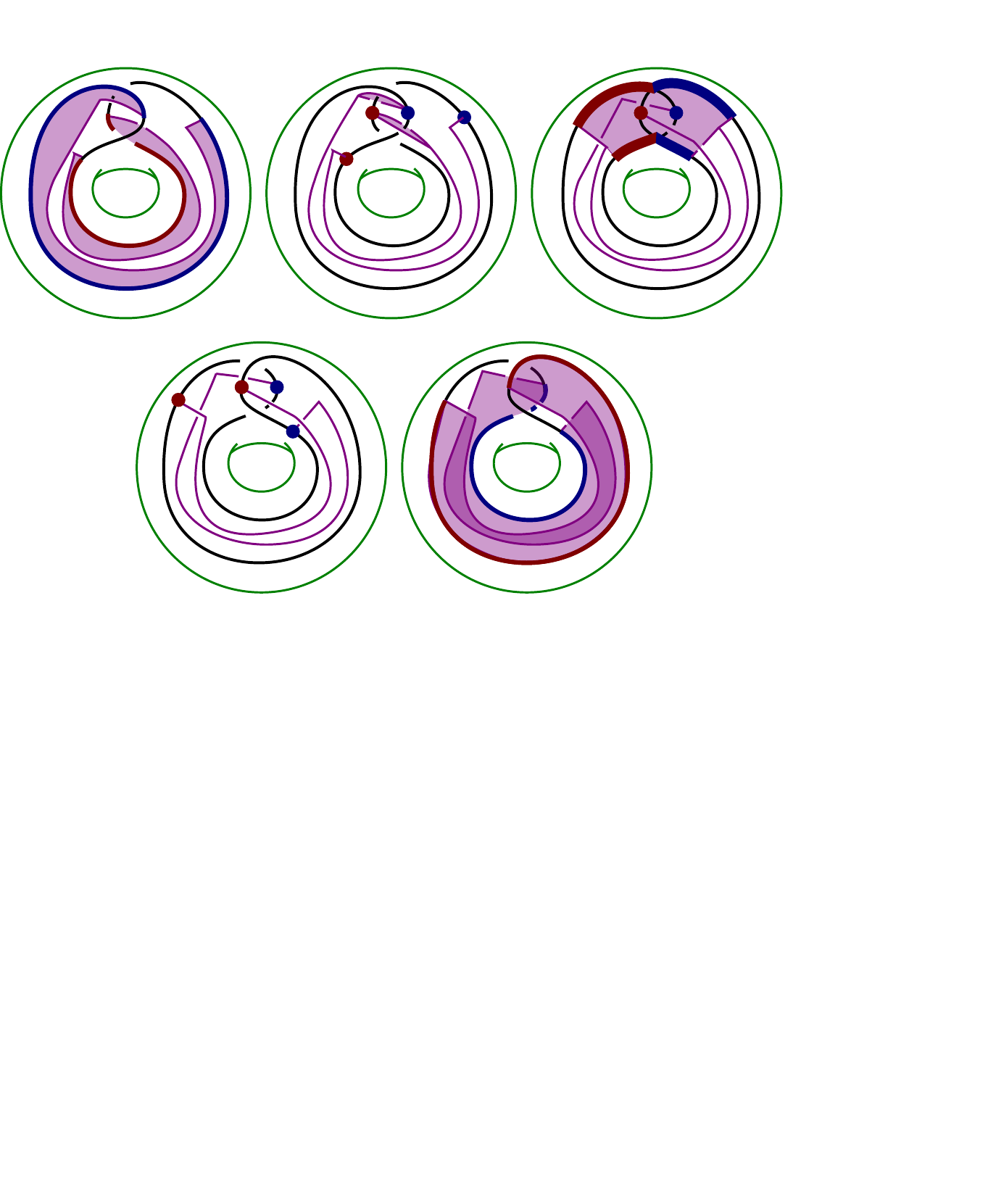
    \caption{The Whitney disk $W'$ from Figure \ref{fig:stongdisk1}. In the top two rows, we draw the whole disk in $S^1\times S^3$, perturbing slightly to avoid self-intersections. In the bottom two rows, we perturb $S'$ to make the framing of $W'$ clear. At every point in $W'$, consider a normal vector which points upward within one 3-dimensional cross-section. This induces a Whitney framing of $W'$.}
    \label{fig:stongdisk2}
\end{figure}

%\begin{figure}
 %   \centering
  %  \input{stongdisk3.pdf_tex}
   % \caption{The Whitney disk $W'$ from Figures \ref{fig:stongdisk1} and \ref{fig:stongdisk2}, enlarged for easy viewing.}
    %\label{fig:stongdisk3}
%\end{figure}

%in the homology class of $[S^2\times\pt]$ which is obtained from the $5$-twist spun trefoil knot $\tau^5 K_{2,3}$ in $S^4$. In Zeeman's \cite{zeeman} original construction of twist-spun knots, he shows that $\pi_1(S^4-\tau^5 K_{2,3})\cong G\times\Z$, where $G$ is the binary icosahedral group. Then surgery along a curve $\mu$ generating the $\Z$ summand of $\pi_1$ yields $S^2\times S^2$; Sato \cite{Sato} sets $K$ to be the image of $\tau^5 K_{2,3}$. This 2-sphere $K$ is homotopic to $S^2\times\pt$, yet not isotopic to $S^2\times\pt$. However, Theorem \ref{thm:maggie} (and indeed Theorem \ref[Theorem 6.1]{NS}) implies that $K$ is concordant to $S^2\times\pt$. See Figure \ref{fig:sato} for an explicit picture of $K$; it takes a nontrivial amount of work to extract $\mu$ from Zeeman's \cite{zeeman} or Sato's \cite{Sato}  

\begin{example}\label{satoexample}
	The following appears in \cite{Sato}.  We give two 2-spheres that are concordant but not isotopic in $S^2 \times S^2$.  The first sphere is $S_0:= S^2 \times \pt$ which satisfies the hypothesis that the meridian circle is null homotopic in the complement.  Any sphere in the same homotopy class of $S_0$ is concordant to $S_0$, by Theorem \ref{simply-connected}. We aim to construct a sphere $S_1$ with $[S_1] = [S_0]$ but $\pi_1(S^2 \times S^2 - S_0) = 1 \neq \pi_1(S^2 \times S^2 - S_1)$ and therefore conclude that $S_0$ and $S_1$ are not isotopic.

	Let $K$ be a 2-knot (a knotted 2-sphere) in $S^4$ and let $C$ be a loop embedded in $S^4 - K$ that is homologous to a meridian of $K$ in $S^4 - K$.  Surgery on $C$ yields $S^2\times S^2$, in which we may view $K$ as the sphere $K(C)$. %By removing a regular neighborhood of $C$, we obtain a 4-manifold diffeomorphic to $S^2 \times D^2$ from which, by gluing on another copy of $S^2 \times D^2$, we obtain $S^2 \times S^2$, in which $\pt\times S^2\cap S^4\setminus\mathring{N}(C)$ is a disk intersecting $K$ algebraically once. Thus, we can view $K$ as sitting inside of $S^2 \times S^2$.  We call the resulting sphere in $S^2 \times S^2$, $K(C)$.
	The assumption that $C$ is homologous to a meridian of $S^4 - K$ ensures that the resulting sphere $K(C)$ represents the homology class $[K(C)] = [S_0]$.

	By Seifert-van Kampen, $\pi_1(S^2 \times S^2 - K(C)) = \pi_1(S^4 - K)/\langle\langle C\rangle\rangle$, where $\langle\langle C\rangle\rangle$ is the normal subgroup generated by $C$.  %So if we choose any $K$ and $C$ where $\pi_1(S^4 - K) / <<C>> \neq 1$, then $K(C)$ will be an example of a desired $S_1$.  For example,
	Take $K$ to be the 5-twist spun trefoil knot \cite{zeeman}. We have $\pi_1(S^4 - K) = G \times \mathbb{Z}$ where $G = <a, b ; a^3 = b^5 = (ab)^2>$ is a perfect group of order 120 (the binary dodecahedral group). Take $C$ to generate the $\mathbb{Z}$ factor and set $S_1:=K(C)$. Thus, $S_0$ and $S_1$ satisfy the hypotheses of Theorem \ref{maintheorem} and are indeed concordant, but are not ambiently isotopic.
	
	See Figure \ref{fig:sato} for an illustration; we include this partly because extracting the generator of the $\mathbb{Z}$ summand from Zeeman's \cite{zeeman} original description of $\pi_1(S^4-K)$ takes some effort.
\end{example} 

 \begin{figure}
 	  \labellist
\small\hair 2pt
 \pinlabel $0$ at 1430 950
 \pinlabel $0$ at 1720 680
  \pinlabel $C$ at 800 680
 \pinlabel $a$ at 240 740
  \pinlabel $b$ at -30 1000
   \pinlabel $C$ at 400 330
   \pinlabel $0$ at 1335 330
   \pinlabel $0$ at 1200 620
\endlabellist
 \includegraphics[width=120mm]{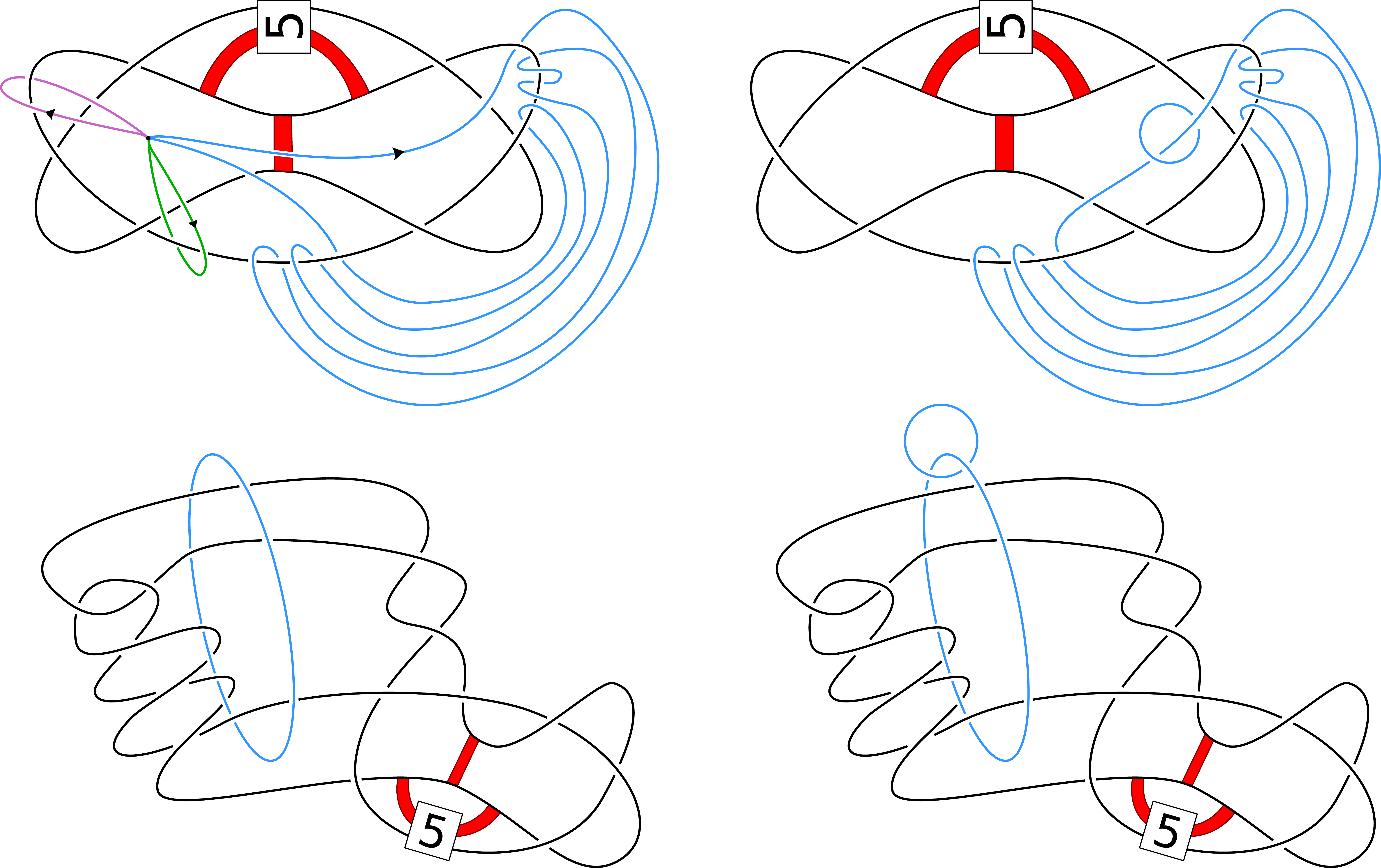}
	\caption{In this figure we use banded unlink diagrams to describe surfaces. We refer the reader to \cite{bandpaper} for exposition on these diagrams. {\bf{Top left:}} the standard diagram of the 5-twist spun trefoil knot $K$. We indicate two meridians $a$ and $b$ which generate $\pi_1(S^4-K)$ and a loop  $C$ in the class $w=b^{-1}a^{-1}b^2a^{-1}b^{-1}a$, which generates the $\Z$ summand of $\pi_1(S^4-K)=\langle a,b\mid baba^{-1}b^{-1}a^{-1}=1,a^5b=ba^5\rangle=\langle x,y,w\mid x^5=(xy)^3=(xyx)^2,[x,w]=[y,w]=1\rangle$ where $x=ab^{-1}$ and $y=b^{-1}a$ . {\bf{Top right:}} Surgery along $C$ yields a sphere $K(C)$ in $S^2\times S^2$ which is homotopic but not isotopic to $S^2\times\pt$. {\bf{Bottom left:}} We isotope the diagram of $K$ to simplify $C$. {\bf{Bottom right:}} An alternate picture of $K(C)\subset S^2\times S^2$.}\label{fig:sato}
\end{figure}

\begin{remark}
By connect summing the spheres in Example \ref{hannahexample} with $g$ trivial tori, we may obtain genus-$g$ $\pi_1$-negligible surfaces $F_0$ and $F_1$ in $X^4$ that are based-homotopic and have dual spheres, but have nonvanishing Freedman-Quinn invariant and thus fail to be concordant. Similarly, by connect summing the spheres in Example \ref{satoexample} with $g$ trivial tori, we may obtain genus-$g$ surfaces $F_0$ and $F_1$ in $S^2\times S^2$ that are based-homotopic with $\fq(F_0,F_1)=0$ and so that $F_0$ has a dual sphere, but $F_0$ and $F_1$ are not isotopic (although they are, of course, concordant).

Connect summing the spheres in Example \ref{stongexample} with trivial tori likely gives examples of positive-genus $\pi_1$-negligible surfaces illustrating the necessity of the dual sphere in Theorem \ref{maintheorem} being framed. However, the statements in Stong's work \cite{stong} do not immediately apply in this setting, so these examples would require a more detailed explanation. We plan to discuss this in forthcoming work.
\end{remark}

\section{Further questions: nontriviality of $\pi_1(F_i)\into\pi_1(X)$}\label{sec:questions}

So far, we have required that positive genus surfaces $F_0, F_1$ be $\pi_1$-negligible. This condition ensures that $F_i$ and homotopies of $F_0$ to $F_1$ lift to $\widetilde{X}$ and $\widetilde{X}\times I$ respectively, and allows us to define the Freedman-Quinn invariant $\fq(F_0,F_1)$ if $F_0$ and $F_1$ are based-homotopic. When $F_0$ and $F_1$ are based-homotopic but $\pi_1(F_i)$ maps nontrivially into $X^4$, then $\fq(F_0,F_1)$ is undefined.

\begin{question}
Let $F_0$ and $F_1$ be genus-g based-homotopic surfaces in a 4-manifold $X^4$. Assume that a meridian of $F_0$ is null-homotopic in $X-F_0$.

\begin{enumerate}
    \item What conditions on the homotopy between $F_0$ and $F_1$ ensure that $F_0$ and $F_1$ are concordant?\label{q1}
    \item What conditions obstruct $F_0$ and $F_1$ from being concordant?\label{q2}
    \item The same as Questions \ref{q1} and \ref{q2}, if we assume $[F_0]=[F_1]\in H_2(X;\Z\pi_1(X)).$
\end{enumerate}

Note that $[F_0]=[F_1]\in H_2(X;\Z\pi_1(X))$ implies that the lifts of $F_0$ and $F_1$ to $\widetilde{X}$ are componentwise homotopic. This condition holds $F_0$ and $F_1$ are $\pi_1$-negligible.%automatically when $\pi_1(F_i)$ includes trivially into $\pi_1(X)$.%, since we can lift $F_i$ and a homotopy of $F_i$ to $\widetilde{X},\widetilde{X}\times I$.
\end{question}

%It would interesting to have examples of surfaces satisfying the above hypotheses, except for the condition that they have trivial image in $\pi_1$, hence $fq$ is undefined, where there is no concordance between them.  

 \bibliographystyle{alpha}
 \bibliography{biblio}

%\begin{thebibliography}{9}

%\bibitem{Sato}
%Yoshihisa Sato, Locally Flat 2-Knots in $S^2 \times S^2$ with the Same Fundamental Group, Transactions of the American Mathematical Society, Vol. 323, No. 2 (Feb., 1991),
%pp. 911-920.

%\bibitem{Hannah}
%Hannah R. Schwartz, Equivalent non‐isotopic spheres in 4‐manifolds, Journal of Topology, Volume 12, Issue 4, December 2019, Pages 1396-1412

%\bibitem{NS}
%Nathan Sunukjian, Surfaces in 4-manifolds: Concordance, Isotopy, and Surgery

%\bibitem{ST}
%Rob Schneidermann and Peter Teichner, Homotopy versus Isotopy: Spheres with Duals in 4-Manifolds (2019), arXiv:1904.12350 [math.GT].

%\end{thebibliography}

\end{document}

%% file: framinginteger.pdf_tex
%% Accompanies image file 'framinginteger.pdf' (pdf, eps, ps)
\begingroup%
  \makeatletter%
  \providecommand\color[2][]{%
    \errmessage{(Inkscape) Color is used for the text in Inkscape, but the package 'color.sty' is not loaded}%
    \renewcommand\color[2][]{}%
  }%
  \providecommand\transparent[1]{%
    \errmessage{(Inkscape) Transparency is used (non-zero) for the text in Inkscape, but the package 'transparent.sty' is not loaded}%
    \renewcommand\transparent[1]{}%
  }%
  \providecommand\rotatebox[2]{#2}%
  \newcommand*\fsize{\dimexpr\f@size pt\relax}%
  \newcommand*\lineheight[1]{\fontsize{\fsize}{#1\fsize}\selectfont}%
  \ifx\svgwidth\undefined%
    \setlength{\unitlength}{223.99418868bp}%
    \ifx\svgscale\undefined%
      \relax%
    \else%
      \setlength{\unitlength}{\unitlength * \real{\svgscale}}%
    \fi%
  \else%
    \setlength{\unitlength}{\svgwidth}%
  \fi%
  \global\let\svgwidth\undefined%
  \global\let\svgscale\undefined%
  \makeatother%
  \begin{picture}(1,0.44477173)%
    \lineheight{1}%
    \setlength\tabcolsep{0pt}%
    \put(0,0){\includegraphics[width=\unitlength,page=1]{framinginteger.pdf}}%
    \put(0.11379504,0.31612183){\color[rgb]{0,0,0}\makebox(0,0)[lt]{\lineheight{1.25}\smash{\begin{tabular}[t]{l}$\gamma$\end{tabular}}}}%
    \put(0,0){\includegraphics[width=\unitlength,page=2]{framinginteger.pdf}}%
    \put(0.29746473,0.31612183){\color[rgb]{0,0,1}\makebox(0,0)[lt]{\lineheight{1.25}\smash{\begin{tabular}[t]{l}$\phi$\end{tabular}}}}%
    \put(0,0){\includegraphics[width=\unitlength,page=3]{framinginteger.pdf}}%
    \put(0.57373086,0.37917745){\color[rgb]{0,0,0}\makebox(0,0)[lt]{\lineheight{1.25}\smash{\begin{tabular}[t]{l}$Y$\end{tabular}}}}%
    \put(0,0){\includegraphics[width=\unitlength,page=4]{framinginteger.pdf}}%
    \put(0.65986849,0.31612183){\color[rgb]{0,0,0}\makebox(0,0)[lt]{\lineheight{1.25}\smash{\begin{tabular}[t]{l}$\gamma$\end{tabular}}}}%
    \put(0,0){\includegraphics[width=\unitlength,page=5]{framinginteger.pdf}}%
    \put(0.84353818,0.31612183){\color[rgb]{0,0,1}\makebox(0,0)[lt]{\lineheight{1.25}\smash{\begin{tabular}[t]{l}$1\cdot\phi$\end{tabular}}}}%
    \put(0.03028463,0.37917745){\color[rgb]{0,0,0}\makebox(0,0)[lt]{\lineheight{1.25}\smash{\begin{tabular}[t]{l}$Y$\end{tabular}}}}%
  \end{picture}%
\endgroup%

%% file: framings.pdf_tex
%% Accompanies image file 'framings.pdf' (pdf, eps, ps)
\begingroup%
  \makeatletter%
  \providecommand\color[2][]{%
    \errmessage{(Inkscape) Color is used for the text in Inkscape, but the package 'color.sty' is not loaded}%
    \renewcommand\color[2][]{}%
  }%
  \providecommand\transparent[1]{%
    \errmessage{(Inkscape) Transparency is used (non-zero) for the text in Inkscape, but the package 'transparent.sty' is not loaded}%
    \renewcommand\transparent[1]{}%
  }%
  \providecommand\rotatebox[2]{#2}%
  \newcommand*\fsize{\dimexpr\f@size pt\relax}%
  \newcommand*\lineheight[1]{\fontsize{\fsize}{#1\fsize}\selectfont}%
  \ifx\svgwidth\undefined%
    \setlength{\unitlength}{295.8674177bp}%
    \ifx\svgscale\undefined%
      \relax%
    \else%
      \setlength{\unitlength}{\unitlength * \real{\svgscale}}%
    \fi%
  \else%
    \setlength{\unitlength}{\svgwidth}%
  \fi%
  \global\let\svgwidth\undefined%
  \global\let\svgscale\undefined%
  \makeatother%
  \begin{picture}(1,0.52509038)%
    \lineheight{1}%
    \setlength\tabcolsep{0pt}%
    \put(0,0){\includegraphics[width=\unitlength,page=1]{framings.pdf}}%
    \put(0.04150466,0.44287956){\color[rgb]{0,0,1}\makebox(0,0)[lt]{\lineheight{1.25}\smash{\begin{tabular}[t]{l}$W$\end{tabular}}}}%
    \put(0,0){\includegraphics[width=\unitlength,page=2]{framings.pdf}}%
    \put(0.14994316,0.32090401){\color[rgb]{0,0,0}\makebox(0,0)[lt]{\lineheight{1.25}\smash{\begin{tabular}[t]{l}$Y$\end{tabular}}}}%
    \put(0,0){\includegraphics[width=\unitlength,page=3]{framings.pdf}}%
    \put(0.08337885,0.2017836){\color[rgb]{0,0.50196078,0}\makebox(0,0)[lt]{\lineheight{1.25}\smash{\begin{tabular}[t]{l}$\gamma$\end{tabular}}}}%
    \put(0,0){\includegraphics[width=\unitlength,page=4]{framings.pdf}}%
    \put(0.35912688,0.35788937){\color[rgb]{0.50196078,0,0}\makebox(0,0)[lt]{\lineheight{1.25}\smash{\begin{tabular}[t]{l}$\Delta$\end{tabular}}}}%
    \put(0,0){\includegraphics[width=\unitlength,page=5]{framings.pdf}}%
    \put(1.01551662,0.24896101){\color[rgb]{1,0,0}\makebox(0,0)[lt]{\lineheight{1.25}\smash{\begin{tabular}[t]{l}$T\Delta|_\gamma$\end{tabular}}}}%
    \put(0,0){\includegraphics[width=\unitlength,page=6]{framings.pdf}}%
    \put(0.61672011,0.12284816){\color[rgb]{0,0,0}\makebox(0,0)[lt]{\lineheight{1.25}\smash{\begin{tabular}[t]{l}$N_Y(\gamma)$\end{tabular}}}}%
    \put(0,0){\includegraphics[width=\unitlength,page=7]{framings.pdf}}%
    \put(0.81579313,0.44724741){\color[rgb]{0,0,1}\makebox(0,0)[lt]{\lineheight{1.25}\smash{\begin{tabular}[t]{l}$\xi$\end{tabular}}}}%
    \put(0.75486927,0.30079824){\color[rgb]{0,0.50196078,0}\makebox(0,0)[lt]{\lineheight{1.25}\smash{\begin{tabular}[t]{l}$\gamma$\end{tabular}}}}%
  \end{picture}%
\endgroup%

%% file: tubetoframed.pdf_tex
%% Accompanies image file 'tubetoframed.pdf' (pdf, eps, ps)
\begingroup%
  \makeatletter%
  \providecommand\color[2][]{%
    \errmessage{(Inkscape) Color is used for the text in Inkscape, but the package 'color.sty' is not loaded}%
    \renewcommand\color[2][]{}%
  }%
  \providecommand\transparent[1]{%
    \errmessage{(Inkscape) Transparency is used (non-zero) for the text in Inkscape, but the package 'transparent.sty' is not loaded}%
    \renewcommand\transparent[1]{}%
  }%
  \providecommand\rotatebox[2]{#2}%
  \newcommand*\fsize{\dimexpr\f@size pt\relax}%
  \newcommand*\lineheight[1]{\fontsize{\fsize}{#1\fsize}\selectfont}%
  \ifx\svgwidth\undefined%
    \setlength{\unitlength}{283.66885208bp}%
    \ifx\svgscale\undefined%
      \relax%
    \else%
      \setlength{\unitlength}{\unitlength * \real{\svgscale}}%
    \fi%
  \else%
    \setlength{\unitlength}{\svgwidth}%
  \fi%
  \global\let\svgwidth\undefined%
  \global\let\svgscale\undefined%
  \makeatother%
  \begin{picture}(1,0.30400618)%
    \lineheight{1}%
    \setlength\tabcolsep{0pt}%
    \put(0,0){\includegraphics[width=\unitlength,page=1]{tubetoframed.pdf}}%
    \put(0.12406351,0.19635319){\color[rgb]{0,0.50196078,0}\makebox(0,0)[lt]{\lineheight{1.25}\smash{\begin{tabular}[t]{l}$\delta$\end{tabular}}}}%
    \put(0.36143843,0.03675725){\color[rgb]{1,0,0}\makebox(0,0)[lt]{\lineheight{1.25}\smash{\begin{tabular}[t]{l}$R$\end{tabular}}}}%
    \put(-0.00304825,0.09432208){\color[rgb]{0.50196078,0,0.50196078}\makebox(0,0)[lt]{\lineheight{1.25}\smash{\begin{tabular}[t]{l}$\Delta$\end{tabular}}}}%
    \put(0,0){\includegraphics[width=\unitlength,page=2]{tubetoframed.pdf}}%
    \put(0.56964505,0.09432208){\color[rgb]{0.50196078,0,0.50196078}\makebox(0,0)[lt]{\lineheight{1.25}\smash{\begin{tabular}[t]{l}$\tilde{\Delta}$\end{tabular}}}}%
    \put(0,0){\includegraphics[width=\unitlength,page=3]{tubetoframed.pdf}}%
  \end{picture}%
\endgroup%

%% file: unknotmovesketch.pdf_tex
%% Accompanies image file 'unknotmovesketch.pdf' (pdf, eps, ps)
\begingroup%
  \makeatletter%
  \providecommand\color[2][]{%
    \errmessage{(Inkscape) Color is used for the text in Inkscape, but the package 'color.sty' is not loaded}%
    \renewcommand\color[2][]{}%
  }%
  \providecommand\transparent[1]{%
    \errmessage{(Inkscape) Transparency is used (non-zero) for the text in Inkscape, but the package 'transparent.sty' is not loaded}%
    \renewcommand\transparent[1]{}%
  }%
  \providecommand\rotatebox[2]{#2}%
  \newcommand*\fsize{\dimexpr\f@size pt\relax}%
  \newcommand*\lineheight[1]{\fontsize{\fsize}{#1\fsize}\selectfont}%
  \ifx\svgwidth\undefined%
    \setlength{\unitlength}{328.97242509bp}%
    \ifx\svgscale\undefined%
      \relax%
    \else%
      \setlength{\unitlength}{\unitlength * \real{\svgscale}}%
    \fi%
  \else%
    \setlength{\unitlength}{\svgwidth}%
  \fi%
  \global\let\svgwidth\undefined%
  \global\let\svgscale\undefined%
  \makeatother%
  \begin{picture}(1,0.20176031)%
    \lineheight{1}%
    \setlength\tabcolsep{0pt}%
    \put(0,0){\includegraphics[width=\unitlength,page=1]{unknotmovesketch.pdf}}%
    \put(0.03525538,0.10435128){\color[rgb]{0,0,1}\makebox(0,0)[lt]{\lineheight{1.25}\smash{\begin{tabular}[t]{l}$\gamma$\end{tabular}}}}%
    \put(0.00050879,0.16945342){\color[rgb]{0.50196078,0,0.50196078}\makebox(0,0)[lt]{\lineheight{1.25}\smash{\begin{tabular}[t]{l}$\Delta$\end{tabular}}}}%
    \put(0,0){\includegraphics[width=\unitlength,page=2]{unknotmovesketch.pdf}}%
    \put(0.35243136,0.16945342){\color[rgb]{0.50196078,0,0.50196078}\makebox(0,0)[lt]{\lineheight{1.25}\smash{\begin{tabular}[t]{l}$\Delta'$\end{tabular}}}}%
    \put(0,0){\includegraphics[width=\unitlength,page=3]{unknotmovesketch.pdf}}%
    \put(0.09574106,0.11109857){\color[rgb]{0,0,0}\makebox(0,0)[lt]{\lineheight{1.25}\smash{\begin{tabular}[t]{l}$C$\end{tabular}}}}%
    \put(0,0){\includegraphics[width=\unitlength,page=4]{unknotmovesketch.pdf}}%
    \put(0.72853276,0.10435128){\color[rgb]{0,0,1}\makebox(0,0)[lt]{\lineheight{1.25}\smash{\begin{tabular}[t]{l}$\tilde{\gamma}$\end{tabular}}}}%
    \put(0.44766357,0.11109857){\color[rgb]{0,0,0}\makebox(0,0)[lt]{\lineheight{1.25}\smash{\begin{tabular}[t]{l}$C$\end{tabular}}}}%
    \put(0.78901838,0.11109857){\color[rgb]{0,0,0}\makebox(0,0)[lt]{\lineheight{1.25}\smash{\begin{tabular}[t]{l}$C$\end{tabular}}}}%
    \put(0.69378631,0.16945342){\color[rgb]{0.50196078,0,0.50196078}\makebox(0,0)[lt]{\lineheight{1.25}\smash{\begin{tabular}[t]{l}$\tilde{\Delta}$\end{tabular}}}}%
    \put(0.70479078,0.00424433){\color[rgb]{0,0,0}\makebox(0,0)[lt]{\lineheight{1.25}\smash{\begin{tabular}[t]{l}tube to parallel copy of $G$\end{tabular}}}}%
    \put(0,0){\includegraphics[width=\unitlength,page=5]{unknotmovesketch.pdf}}%
  \end{picture}%
\endgroup%

%% file: tildegamma.pdf_tex
%% Accompanies image file 'tildegamma.pdf' (pdf, eps, ps)
\begingroup%
  \makeatletter%
  \providecommand\color[2][]{%
    \errmessage{(Inkscape) Color is used for the text in Inkscape, but the package 'color.sty' is not loaded}%
    \renewcommand\color[2][]{}%
  }%
  \providecommand\transparent[1]{%
    \errmessage{(Inkscape) Transparency is used (non-zero) for the text in Inkscape, but the package 'transparent.sty' is not loaded}%
    \renewcommand\transparent[1]{}%
  }%
  \providecommand\rotatebox[2]{#2}%
  \newcommand*\fsize{\dimexpr\f@size pt\relax}%
  \newcommand*\lineheight[1]{\fontsize{\fsize}{#1\fsize}\selectfont}%
  \ifx\svgwidth\undefined%
    \setlength{\unitlength}{381.90397007bp}%
    \ifx\svgscale\undefined%
      \relax%
    \else%
      \setlength{\unitlength}{\unitlength * \real{\svgscale}}%
    \fi%
  \else%
    \setlength{\unitlength}{\svgwidth}%
  \fi%
  \global\let\svgwidth\undefined%
  \global\let\svgscale\undefined%
  \makeatother%
  \begin{picture}(1,1.01070006)%
    \lineheight{1}%
    \setlength\tabcolsep{0pt}%
    \put(0,0){\includegraphics[width=\unitlength,page=1]{tildegamma.pdf}}%
    \put(0.41677861,0.03783294){\color[rgb]{0,0,0}\makebox(0,0)[lt]{\lineheight{1.25}\smash{\begin{tabular}[t]{l}$C$\end{tabular}}}}%
    \put(0.92137831,0.05966253){\color[rgb]{0.50196078,0,0}\makebox(0,0)[lt]{\lineheight{1.25}\smash{\begin{tabular}[t]{l}$Y_1$\end{tabular}}}}%
    \put(0,0){\includegraphics[width=\unitlength,page=2]{tildegamma.pdf}}%
    \put(0.19800931,0.07179325){\color[rgb]{0.50196078,0,0.50196078}\makebox(0,0)[lt]{\lineheight{1.25}\smash{\begin{tabular}[t]{l}$\tilde{\Delta}$\end{tabular}}}}%
    \put(0.45797479,0.28872914){\color[rgb]{0,0,1}\makebox(0,0)[lt]{\lineheight{1.25}\smash{\begin{tabular}[t]{l}$Y_g$\end{tabular}}}}%
    \put(0,0){\includegraphics[width=\unitlength,page=3]{tildegamma.pdf}}%
    \put(0.07063278,0.19443552){\color[rgb]{0,0.50196078,0}\makebox(0,0)[lt]{\lineheight{1.25}\smash{\begin{tabular}[t]{l}$G$\end{tabular}}}}%
    \put(0,0){\includegraphics[width=\unitlength,page=4]{tildegamma.pdf}}%
    \put(0.20520577,0.01800438){\color[rgb]{0,0,1}\makebox(0,0)[lt]{\lineheight{1.25}\smash{\begin{tabular}[t]{l}$\tilde{\gamma}$\end{tabular}}}}%
    \put(0,0){\includegraphics[width=\unitlength,page=5]{tildegamma.pdf}}%
    \put(0.41677861,0.39028253){\color[rgb]{0,0,0}\makebox(0,0)[lt]{\lineheight{1.25}\smash{\begin{tabular}[t]{l}$C$\end{tabular}}}}%
    \put(0.92137831,0.41211212){\color[rgb]{0.50196078,0,0}\makebox(0,0)[lt]{\lineheight{1.25}\smash{\begin{tabular}[t]{l}$Y_1$\end{tabular}}}}%
    \put(0,0){\includegraphics[width=\unitlength,page=6]{tildegamma.pdf}}%
    \put(0.19800931,0.42424329){\color[rgb]{0.50196078,0,0.50196078}\makebox(0,0)[lt]{\lineheight{1.25}\smash{\begin{tabular}[t]{l}$\Delta'$\end{tabular}}}}%
    \put(0,0){\includegraphics[width=\unitlength,page=7]{tildegamma.pdf}}%
    \put(0.45797479,0.64277611){\color[rgb]{0,0,1}\makebox(0,0)[lt]{\lineheight{1.25}\smash{\begin{tabular}[t]{l}$Y_g$\end{tabular}}}}%
    \put(0.07063278,0.5444895){\color[rgb]{0,0.50196078,0}\makebox(0,0)[lt]{\lineheight{1.25}\smash{\begin{tabular}[t]{l}$G$\end{tabular}}}}%
    \put(0.55853717,0.12572755){\color[rgb]{0,0,0}\makebox(0,0)[lt]{\lineheight{1.25}\smash{\begin{tabular}[t]{l}tube $\Delta'$ to a parallel copy of $G$\end{tabular}}}}%
    \put(0,0){\includegraphics[width=\unitlength,page=8]{tildegamma.pdf}}%
    \put(0.41677861,0.74147995){\color[rgb]{0,0,0}\makebox(0,0)[lt]{\lineheight{1.25}\smash{\begin{tabular}[t]{l}$C$\end{tabular}}}}%
    \put(0.92137831,0.76330954){\color[rgb]{0.50196078,0,0}\makebox(0,0)[lt]{\lineheight{1.25}\smash{\begin{tabular}[t]{l}$Y_1$\end{tabular}}}}%
    \put(0,0){\includegraphics[width=\unitlength,page=9]{tildegamma.pdf}}%
    \put(0.19800931,0.77544071){\color[rgb]{0.50196078,0,0.50196078}\makebox(0,0)[lt]{\lineheight{1.25}\smash{\begin{tabular}[t]{l}$\Delta$\end{tabular}}}}%
    \put(0,0){\includegraphics[width=\unitlength,page=10]{tildegamma.pdf}}%
    \put(0.45797479,0.99397307){\color[rgb]{0,0,1}\makebox(0,0)[lt]{\lineheight{1.25}\smash{\begin{tabular}[t]{l}$Y_g$\end{tabular}}}}%
    \put(0.07063278,0.89568692){\color[rgb]{0,0.50196078,0}\makebox(0,0)[lt]{\lineheight{1.25}\smash{\begin{tabular}[t]{l}$G$\end{tabular}}}}%
    \put(0.17007166,0.72059674){\color[rgb]{0,0,1}\makebox(0,0)[lt]{\lineheight{1.25}\smash{\begin{tabular}[t]{l}$\gamma$\end{tabular}}}}%
    \put(0,0){\includegraphics[width=\unitlength,page=11]{tildegamma.pdf}}%
    \put(0.4754561,0.43390546){\color[rgb]{0,0,0}\makebox(0,0)[lt]{\lineheight{1.25}\smash{\begin{tabular}[t]{l}$\Delta'\cap Y_g$\end{tabular}}}}%
    \put(0,0){\includegraphics[width=\unitlength,page=12]{tildegamma.pdf}}%
  \end{picture}%
\endgroup%

%% file: stonghopf.pdf_tex
%% Accompanies image file 'stonghopf.pdf' (pdf, eps, ps)
\begingroup%
  \makeatletter%
  \providecommand\color[2][]{%
    \errmessage{(Inkscape) Color is used for the text in Inkscape, but the package 'color.sty' is not loaded}%
    \renewcommand\color[2][]{}%
  }%
  \providecommand\transparent[1]{%
    \errmessage{(Inkscape) Transparency is used (non-zero) for the text in Inkscape, but the package 'transparent.sty' is not loaded}%
    \renewcommand\transparent[1]{}%
  }%
  \providecommand\rotatebox[2]{#2}%
  \newcommand*\fsize{\dimexpr\f@size pt\relax}%
  \newcommand*\lineheight[1]{\fontsize{\fsize}{#1\fsize}\selectfont}%
  \ifx\svgwidth\undefined%
    \setlength{\unitlength}{312.57840134bp}%
    \ifx\svgscale\undefined%
      \relax%
    \else%
      \setlength{\unitlength}{\unitlength * \real{\svgscale}}%
    \fi%
  \else%
    \setlength{\unitlength}{\svgwidth}%
  \fi%
  \global\let\svgwidth\undefined%
  \global\let\svgscale\undefined%
  \makeatother%
  \begin{picture}(1,0.70888198)%
    \lineheight{1}%
    \setlength\tabcolsep{0pt}%
    \put(0,0){\includegraphics[width=\unitlength,page=1]{stonghopf.pdf}}%
    \put(0.47789984,0.5292098){\color[rgb]{0.50196078,0,0.50196078}\makebox(0,0)[lt]{\lineheight{1.25}\smash{\begin{tabular}[t]{l}$W$\end{tabular}}}}%
    \put(0,0){\includegraphics[width=\unitlength,page=2]{stonghopf.pdf}}%
    \put(0.44775822,0.5292098){\color[rgb]{0,0,1}\makebox(0,0)[lt]{\lineheight{1.25}\smash{\begin{tabular}[t]{l}$\beta$\end{tabular}}}}%
    \put(0.51553859,0.5292098){\color[rgb]{1,0,0}\makebox(0,0)[lt]{\lineheight{1.25}\smash{\begin{tabular}[t]{l}$\alpha$\end{tabular}}}}%
    \put(0.51873929,0.34682202){\color[rgb]{0,0,1}\makebox(0,0)[lt]{\lineheight{1.25}\smash{\begin{tabular}[t]{l}$\beta'$\end{tabular}}}}%
    \put(0.45503076,0.34682202){\color[rgb]{1,0,0}\makebox(0,0)[lt]{\lineheight{1.25}\smash{\begin{tabular}[t]{l}$\alpha'$\end{tabular}}}}%
    \put(0,0){\includegraphics[width=\unitlength,page=3]{stonghopf.pdf}}%
    \put(0.10312014,0.06464474){\color[rgb]{1,0,0}\makebox(0,0)[lt]{\lineheight{1.25}\smash{\begin{tabular}[t]{l}$\alpha$\end{tabular}}}}%
    \put(0.23165935,0.05693237){\color[rgb]{0,0,1}\makebox(0,0)[lt]{\lineheight{1.25}\smash{\begin{tabular}[t]{l}$\beta$\end{tabular}}}}%
    \put(0.52441242,0.13777755){\color[rgb]{1,0,0}\makebox(0,0)[lt]{\lineheight{1.25}\smash{\begin{tabular}[t]{l}$\alpha'$\end{tabular}}}}%
    \put(0.48752574,0.01097005){\color[rgb]{0,0,1}\makebox(0,0)[lt]{\lineheight{1.25}\smash{\begin{tabular}[t]{l}$\beta'$\end{tabular}}}}%
    \put(0.78439562,0.14460778){\color[rgb]{1,0,0}\makebox(0,0)[lt]{\lineheight{1.25}\smash{\begin{tabular}[t]{l}$\alpha$\end{tabular}}}}%
    \put(0.9139997,0.1394054){\color[rgb]{0,0,1}\makebox(0,0)[lt]{\lineheight{1.25}\smash{\begin{tabular}[t]{l}$\beta$\end{tabular}}}}%
    \put(0.72610672,0.05436163){\color[rgb]{1,0,0}\makebox(0,0)[lt]{\lineheight{1.25}\smash{\begin{tabular}[t]{l}$\alpha'$\end{tabular}}}}%
    \put(0.96090492,0.05607541){\color[rgb]{0,0,1}\makebox(0,0)[lt]{\lineheight{1.25}\smash{\begin{tabular}[t]{l}$\beta'$\end{tabular}}}}%
  \end{picture}%
\endgroup%

%% file: stongdisk1.pdf_tex
%% Accompanies image file 'stongdisk1.pdf' (pdf, eps, ps)
\begingroup%
  \makeatletter%
  \providecommand\color[2][]{%
    \errmessage{(Inkscape) Color is used for the text in Inkscape, but the package 'color.sty' is not loaded}%
    \renewcommand\color[2][]{}%
  }%
  \providecommand\transparent[1]{%
    \errmessage{(Inkscape) Transparency is used (non-zero) for the text in Inkscape, but the package 'transparent.sty' is not loaded}%
    \renewcommand\transparent[1]{}%
  }%
  \providecommand\rotatebox[2]{#2}%
  \newcommand*\fsize{\dimexpr\f@size pt\relax}%
  \newcommand*\lineheight[1]{\fontsize{\fsize}{#1\fsize}\selectfont}%
  \ifx\svgwidth\undefined%
    \setlength{\unitlength}{371.71073373bp}%
    \ifx\svgscale\undefined%
      \relax%
    \else%
      \setlength{\unitlength}{\unitlength * \real{\svgscale}}%
    \fi%
  \else%
    \setlength{\unitlength}{\svgwidth}%
  \fi%
  \global\let\svgwidth\undefined%
  \global\let\svgscale\undefined%
  \makeatother%
  \begin{picture}(1,0.70292827)%
    \lineheight{1}%
    \setlength\tabcolsep{0pt}%
    \put(0,0){\includegraphics[width=\unitlength,page=1]{stongdisk1.pdf}}%
    \put(0.55442255,0.68956373){\color[rgb]{0,0,0}\makebox(0,0)[lt]{\lineheight{1.25}\smash{\begin{tabular}[t]{l}$W'\cap S'$\end{tabular}}}}%
    \put(0,0){\includegraphics[width=\unitlength,page=2]{stongdisk1.pdf}}%
    \put(0.51210734,0.12651517){\color[rgb]{0,0,0}\makebox(0,0)[lt]{\lineheight{1.25}\smash{\begin{tabular}[t]{l}$W'\cap S'$\end{tabular}}}}%
    \put(0,0){\includegraphics[width=\unitlength,page=3]{stongdisk1.pdf}}%
  \end{picture}%
\endgroup%

%% file: stongdisk2.pdf_tex
%% Accompanies image file 'stongdisk2.pdf' (pdf, eps, ps)
\begingroup%
  \makeatletter%
  \providecommand\color[2][]{%
    \errmessage{(Inkscape) Color is used for the text in Inkscape, but the package 'color.sty' is not loaded}%
    \renewcommand\color[2][]{}%
  }%
  \providecommand\transparent[1]{%
    \errmessage{(Inkscape) Transparency is used (non-zero) for the text in Inkscape, but the package 'transparent.sty' is not loaded}%
    \renewcommand\transparent[1]{}%
  }%
  \providecommand\rotatebox[2]{#2}%
  \newcommand*\fsize{\dimexpr\f@size pt\relax}%
  \newcommand*\lineheight[1]{\fontsize{\fsize}{#1\fsize}\selectfont}%
  \ifx\svgwidth\undefined%
    \setlength{\unitlength}{392.40023876bp}%
    \ifx\svgscale\undefined%
      \relax%
    \else%
      \setlength{\unitlength}{\unitlength * \real{\svgscale}}%
    \fi%
  \else%
    \setlength{\unitlength}{\svgwidth}%
  \fi%
  \global\let\svgwidth\undefined%
  \global\let\svgscale\undefined%
  \makeatother%
  \begin{picture}(1,1.21884263)%
    \lineheight{1}%
    \setlength\tabcolsep{0pt}%
    \put(0,0){\includegraphics[width=\unitlength,page=1]{stongdisk2.pdf}}%
    \put(0.21692504,1.19882882){\color[rgb]{0,0,0}\makebox(0,0)[lt]{\lineheight{1.25}\smash{\begin{tabular}[t]{l}$W'\cap S'=\{$2 pts$\}$\end{tabular}}}}%
    \put(0,0){\includegraphics[width=\unitlength,page=2]{stongdisk2.pdf}}%
  \end{picture}%
\endgroup%

%% file: KlugMiller_Sep15.bbl
\begin{thebibliography}{DNPR18}

\bibitem[DNPR18]{dnpr}
Christopher~W. Davis, Matthias Nagel, JungHwan Park, and Arunima Ray.
\newblock Concordance of knots in {$S^1\times S^2$}.
\newblock {\em J. Lond. Math. Soc. (2)}, 98(1):59--84, 2018.

\bibitem[FQ90]{fq}
Michael~H. Freedman and Frank Quinn.
\newblock {\em Topology of 4-manifolds}, volume~39 of {\em Princeton
  Mathematical Series}.
\newblock Princeton University Press, Princeton, NJ, 1990.

\bibitem[Gab17]{Dave}
David Gabai.
\newblock The 4-dimensional light bulb theorem.
\newblock {\em arXiv e-prints}, page arXiv:1705.09989, May 2017.
\newblock To appear in {\emph{J. Amer. Math. Soc.}}

\bibitem[HKM19]{bandpaper}
Mark Hughes, Seungwon Kim, and Maggie Miller.
\newblock {Isotopies of surfaces in 4-manifolds via banded unlink diagrams}.
\newblock {\em arXiv e-prints}, page arXiv:1804.09169v4, Jan 2019.
\newblock To appear in {\emph{Geom. \& Topol.}}

\bibitem[Ker65]{kervaire}
Michel~A. Kervaire.
\newblock Les n\oe uds de dimensions sup\'{e}rieures.
\newblock {\em Bull. Soc. Math. France}, 93:225--271, 1965.

\bibitem[{Mil}19]{Maggie}
Maggie {Miller}.
\newblock {A concordance analogue of the $4$-dimensional light bulb theorem}.
\newblock {\em arXiv e-prints}, page arXiv:1903.03055, Mar 2019.
\newblock To appear in {\emph{Int. Math. Res. Notices IMRN}}.

\bibitem[Sat91]{Sato}
Yoshihisa Sato.
\newblock Locally flat {$2$}-knots in {$S^2\times S^2$} with the same
  fundamental group.
\newblock {\em Trans. Amer. Math. Soc.}, 323(2):911--920, 1991.

\bibitem[{Sch}19]{Hannah}
Hannah~R. {Schwartz}.
\newblock {Equivalent non-isotopic spheres in 4-manifolds}.
\newblock {\em J. Topol.}, 12:1396--1412, Dec 2019.

\bibitem[ST19]{ST}
Rob {Schneiderman} and Peter {Teichner}.
\newblock {Homotopy versus isotopy: spheres with duals in 4-manifolds}.
\newblock {\em arXiv e-prints}, page arXiv:1904.12350, Apr 2019.

\bibitem[Sto93]{stong}
Richard Stong.
\newblock Uniqueness of $\pi_1$-negligible embeddings in 4-manifolds: A
  correction to theorem 10.5 of freedman and quinn.
\newblock {\em Topology}, 32:677--699, 1993.

\bibitem[Sun15]{NS}
Nathan~S. Sunukjian.
\newblock Surfaces in 4-manifolds: concordance, isotopy, and surgery.
\newblock {\em Int. Math. Res. Not. IMRN}, 2015(17):7950--7978, 2015.

\bibitem[Yil18]{yildiz}
Eylem~Zeliha Yildiz.
\newblock A note on knot concordance.
\newblock {\em Algebra. Geom. Topol.}, 18:3119--3128, 2018.

\bibitem[Zee65]{zeeman}
Christopher Zeeman.
\newblock Twisting spun knots.
\newblock {\em Trans. Amer. Math. Soc.}, 115:471--495, 1965.

\end{thebibliography}
